\documentclass[a4paper,10pt]{article}
\usepackage[utf8]{inputenc} 
\usepackage{amsmath, amsfonts, amsthm,amssymb, geometry}
\usepackage{math}
\usepackage[usenames, dvipsnames]{color}
\newtheorem{lemma}{Lemma}[section]
\newtheorem{theorem}[lemma]{Theorem}
\newtheorem{proposition}[lemma]{Proposition}
\newtheorem{corollary}[lemma]{Corollary}
\newtheorem{remark}[lemma]{Remark}
\newtheorem{example}[lemma]{Example}

\newtheorem{definition}[lemma]{Definition}

\providecommand{\keywords}[1]{\textbf{\textit{Keywords: }} #1}\numberwithin{equation}{section}

\newcommand{\h}{\hbar}

\newcommand{\di}{\displaystyle}

\newcommand{\1}{\mathbb I}
\newcommand{\ta}{{\tt a}}
\newcommand{\tM}{{\tt M}}
\newcommand{\tJ}{{\tt J}}
\newcommand{\tN}{{\tt N}}

\def\beq{\begin{equation}}   \def\eeq{\end{equation}}
\def\bea{\begin{eqnarray}}  \def\eea{\end{eqnarray}}
 
\def\noi{\noindent}

\title{	On time dependent  Schr\"odinger equations: global well-posedness and growth of Sobolev norms}
\author{
A. Maspero
\footnote{
Laboratoire de Math\'ematiques Jean Leray, Universit\'e de Nantes, 2 rue de la Houssini\`ere
BP 92208, 44322 Nantes Cedex 3 \newline
 \textit{Email: } \texttt{alberto.maspero@univ-nantes.fr}},  
  D. Robert
 \footnote{Laboratoire de Math\'ematiques Jean Leray, Universit\'e de Nantes, 2 rue de la Houssini\`ere 
BP 92208, 44322 Nantes Cedex 3 \newline
 \textit{Email: } \texttt{didier.robert@univ-nantes.fr}}
  }

\date{\today}

\begin{document}
	
	\maketitle

	\begin{abstract}
	In this paper  we consider  
	time dependent  Schr\"odinger linear PDEs of the form $\im \partial_t \psi = L(t)\psi$, where  $ L(t)$  is a continuous family of self-adjoint operators. We give conditions for well-posedness  and polynomial growth for the evolution  in abstract Sobolev spaces.\\
	If  $L(t) = H +V(t)$ where  $V(t)$ is a  perturbation  smooth  in time  and $H$  is a self-adjoint positive operator
	whose spectrum can be enclosed in spectral clusters whose distance is increasing, we prove that the Sobolev norms of the solution grow at most as $t^\epsilon$ when $t\mapsto \infty$, for any $\epsilon >0$. If $V(t)$ is analytic in time we improve the bound to $(\log t)^\gamma$, for some $\gamma >0$. The proof follows the strategy, due to Howland,  Joye and Nenciu, of the adiabatic approximation of the flow.
	We recover most of known results  and obtain new estimates   for several models including  $1$-degree of freedom Schr\"odinger operators on $\R$ and Schr\"odinger operators on Zoll manifolds.
		\end{abstract}
	
	
	\section{Introduction and Statement of the main results}
	\subsection{Introduction}
	In this paper we study  properties of time dependent Schr\"odinger-type linear partial differential equations defined on scales of Hilbert spaces.  Our aim  is twofold: (i)  to put in a
	unified setting several results only   known in particular  cases  concerning well-posedness and growth of  norms for large time and (ii)  to generalize and extend such results to new models.

 More precisely, given a scale of Hilbert spaces $\{\cH^k\}_{k\in \R}$, we denote by $\la \cdot, \cdot \ra_0$ the scalar product of $\cH^0$, and  we consider  Cauchy problems of the form
 \begin{equation}
	\label{eq:sc0}
	\begin{cases}
	\im \partial_t\psi(t) =  L(t) \psi(t)\\
	\psi\vert_{t=s} = \psi_s\in\cH^k \ , \quad s \in \R
	\end{cases}
	\end{equation}
where $L(t)$  is a  time-dependent, linear, symmetric (w.r.t. $\la \cdot, \cdot \ra_0$) and unbounded operator in $\cH^0$. We  want  here  to establish a list of  simple criteria which ensure the global in time well-posedness, the unitarity in the base space $\cH^0$, as well as giving   bounds on the growth of the $\cH^k$-norms  for the solution of  \eqref{eq:sc0}. In all the paper we assume that the spaces $\cH^k$ are defined as the domains of the powers of a positive self-adjoint operator $H$, i.e. $\cH^k \equiv D(H^{k/2})$. \\

Our first result concerns a very general class of operators $L(t)$.  Roughly speaking, under the condition that  the commutator  $[L(t),H]$ is $H^\tau$-bounded   for some $\tau < 1$, we will prove that  the flow $\cU(t,s)$ of \eqref{eq:sc0} exists globally in time in $\cH^k$ and its norm grows at most polynomially in time as $t \to \infty$, and more precisely  we prove the upper bound 
\begin{equation}
\label{eq:bound0}
\norm{\cU(t,s)\psi}_{\cL(\cH^k)} \leq C \,  \la t- s\ra ^{\frac{k}{2(1-\tau)}}
\end{equation} 
for some constant $C$ independent of $t$. Here $\la x \ra = (1+x^2)^{1/2}$.\\
It is remarkable that such a bound,  in the case $\tau = 0$, is optimal, since   there exist  operators $L(t)$, $H$ with  $[L(t), H]$ bounded s.t.\ the solution of \eqref{eq:sc0}   fulfills $\norm{\cU(t,0)\psi}_{\cH^k} \geq C \,  \la t\ra^{\frac{k}{2}}$. Such an example was constructed by Delort in \cite{del1}, choosing $L(t) =  H+ V(t)$  where $H=  -\Delta + |x|^2 $ on $\R$ is the harmonic oscillator and  $V(t)$ is an ad-hoc pseudodifferential operator of order 0 (see Remark \ref{rem:del} for more details).

 \vspace{1em}
 
 However, with stronger assumptions on $L(t)$ and $H$, one might hope to improve the bound \eqref{eq:bound0}.
  Indeed it is well known that in many interesting situations the norm of flow of \eqref{eq:sc0}  grows much more slowly, in particular at most as $t^\epsilon$ when $t \to \infty$, for any $\epsilon >0$. 
  This is the case for example for equation \eqref{eq:sc0} on $\T$ with  $L(t) = -\Delta + V(t,x)$,  as proved by Bourgain in \cite{Bourgain1999}. Here $\Delta$ is the laplacian and $V(t,x)$ is a smooth potential. The same bound holds also when $L(t) = -\Delta + V(t,x)$ is defined on Zoll manifolds, as proved by Delort \cite{del2}.
 
 The crucial feature of these examples is a spectral property of  the principal operator  $-\Delta$ on Zoll manifolds. Indeed its spectrum can  be enclosed in clusters  whose distance is increasing (we will refer to such property as  {\em increasing spectral gap condition}). Note that, in the example of Remark \ref{rem:del}, the harmonic oscillator $-\Delta + |x|^2$  on $\R$ does not fulfill the increasing spectral gap condition.
 
 Such property motivates our second result. 
In order to improve the upper bound \eqref{eq:bound0}, we put ourselves in the situation where  $L(t)$ is of the form 
 $ L(t) = H + V(t)$  and we assume that $H$ has increasing spectral gaps. Then provided that $V(t)$ is smooth in time, we prove that for every $\varepsilon >0$ the  bound 
 \begin{equation}
 \label{eq:boun1}
  \Vert\cU(t,s)\Vert_{{\cal L}(\cH^k)} \leq C_{k,\varepsilon}\la t-s\ra^\varepsilon,\;\;\forall t, s\in\R
 \end{equation}
 holds. This is essentially the content of Theorem \ref{thm:grest} below. It is important to note that we allow $V(t)$ to be an {\em unbounded} perturbation. More precisely we can take $V(t)$ to be $H^\nu$-bounded, where $\nu<1$ depends only on the spectral properties of $H$.  \\
 In the case where $t \mapsto V(t)$ is analytic, we are able to further improve the bound \eqref{eq:boun1}, obtaining the logarithmic estimate 
  \begin{equation}
 \label{eq:boun2}
  \Vert\cU(t,s)\Vert_{{\cal L}(\cH^k)} \leq \gamma (\log \la t-s \ra )^{\sigma k} ,\;\;\forall t, s\in\R
 \end{equation}
 where the constant $\sigma>0$ can be explicitly calculated. This is the content of Theorem \ref{thm:grest2} below. Once again when $V(t)$ is a bounded perturbation the exponent $\sigma$ that we find is optimal   (see Remark \ref{rem:bou} below).
 
 Finally we apply our abstract theorems to several different models, including one degree of freedom Schr\"odinger operators, perturbations of the laplacian on compact manifolds, Dirac equations, a discrete NLS model and some classes of pseudodifferential operators. We recover many known results proven with different techniques, often improving such results (allowing e.g.  unbounded perturbations) but also obtaining new results.  More details and references will be given in  Section \ref{app}.

 \vspace{1em}
 The problem of estimating the growth of higher norms for equation \eqref{eq:sc0} is very old, and goes back to the pioneering works  initiated by Howland \cite{how}  and developed by    Joye \cite{joy0,joy},    Nenciu \cite{nen}   and  Barbaroux-Joye \cite{bajo}. \\
Such authors, roughly speaking, under  the increasing spectral gaps condition on $H$ and  the assumption that the perturbation  $V(t)$ is smooth in time  and  {\em bounded},    use the method of adiabatic approximation   to prove  that  for    every $\varepsilon >0$   we have 
$$
\Vert\cU(t,s)\Vert_{{\cal L}(\cH^1)} \leq C_{1,\varepsilon}\la t-s\ra^\varepsilon,\;\;\forall t, s\in\R.
$$
Our aim here is to  extend the adiabatic approximation  schema of  Joye and  Nenciu  to  a class of  {\em  unbounded}   perturbations $V(t)$ and to  control the growth of the  $\cH^k$-norm $\forall k >0$.

\vspace{1em}
As a final remark, we would like to mention some situations in which it is possible to prove better bounds, and in particular  to prove that  $\norm{\cU(t,s)}_{\cL(\cH^k)}$ is uniformly bounded in time $\forall k \geq 0$. 
 Such  results can be obtained  for instance provided that the  perturbation $V$ fulfills some stronger assumptions, for example being quasi-periodic in time and small in size. 
Indeed in these cases one might try to apply KAM methods to conjugate  $L(t)$ to a  diagonal operator  with constant coefficients, which in turn implies that  the $\cH^k$-norms are uniformly bounded in time $\forall k \geq 0$.
  The problem of the existence of such a conjugation goes in the literature under the name of {\em reducibility} and  had a tremendous development in the last 20 years. To list the achievements of such theory is out of the scope of this manuscript: we limit ourselves to state the latest results in the various models considered in Section \ref{app}.

\subsection{Main result}

We start to make more precise assumptions. We ask that the scale of Hilbert spaces is generated by a  positive self-adjoint operator  $H$ in ${\cal H}$, in the following sense: first $H$ has  a dense domain $ D(H)\equiv \cH^2$.
 Then, defining for every $k \geq 0 $ the operator  $H^k$  by functional calculus (spectral decomposition), we demand that   $\cH^k \equiv D(H^{k/2})$.
  For $k<0$,  $\cH^k$ is defined by duality  as the completion of $\cH$  with respect to  the norm $\Vert u\Vert_k = \sup\{\vert\la v,u\ra\vert,\;\Vert v\Vert_{-k}\leq 1\}$. Notice that  for every $m\in\R$  and $k\in\R$, $H^m$  is an isometry   from $\cH^{k+2m}$ onto $\cH^k$. Denote by $\cH^{\infty} := \cap_{k \in \R} \cH^k$.  \\
 Let us denote by $\Vert\cdot\Vert_k$  the natural norm on $\cH^k$, which in turns is equivalent to $\norm{H^{k/2} \cdot}_0$. Finally  given a Banach space $\cB$, we denote  by $C_b(\R, \cB)$ the Banach space of continuous  and bounded  maps $f : \R \mapsto \cB$  with the usual sup norm $\Vert f\Vert_\infty:= \sup_{t \in \R} \norm{f(t)}_\cB$. 
 We denote by $C_b^\infty(\R, \cB)$ the space of maps $f:\R \to \cB$ smooth. \\
 Given $\cA, \cB$ Banach spaces, we will denote by $\cL(\cA, \cB)$ the set of linear bounded maps from $\cA$ to $\cB$. In case $\cA \equiv \cB$ we will simply  write $\cL(\cA)$.\\
Given an operator $A$, we say that $A$ is $H^\nu$-bounded if $AH^{-\nu}$ is a bounded operator on $\cH^0$.

\begin{remark}  Recall that 
$\cH^\infty $ is dense in $\cH^k$  $\forall k\in [0, \infty[$.  This follows from the spectral decomposition   of $H$:
$ H = \int_0^\infty\lambda dE_H(\lambda)$ (see \cite{reed1975ii}). Let $E_H[a,b]=\int_a^bdE_H(\lambda)$  be the spectral  projector on $[a, b]$.
 If $\psi\in D(H^{k/2})$ then $E_H[0,N]\in\cH^\infty$   for all $N>0$  and $\di{\lim_{N\rightarrow\infty}\Vert H^{k/2}(\psi-\psi_N)\Vert_0=0}$.
\end{remark}
Let us introduce  now a time dependent family  of operators $L(t)$   and the following conditions:
	\begin{itemize}\item[(H0)] There exist  integers $\tm\geq 0$  and $k_0 > 2\tm$ such that  $t \mapsto L(t) \in C_b(\R, \cL(\cH^{k+2\tm} ,\cH^{k}))$  for $0\leq k\leq k_0$ . 
		\item[(H1)] For every $t \in \R$,  $L(t)$  is  symmetric  on $\cH^{k_0+2\tm}$ w.r.t.\ the scalar product of $\cH^0$  i.e.
	$$
	\la L(t) \psi, \phi \ra_0 = \la \psi, L(t) \phi \ra_0, \qquad \forall \psi, \phi\in\cH^{k_0+2\tm} \ . 
	$$
	
	\item[(H2)] There  exists $k_1>2\tm$  such that $ [L(\cdot), H]H^{-1}\in C_b(\R, \cL(\cH^{k}))$  for $0\leq k \leq 2k_1.$
\end{itemize}
The first theorem concerns  existence of a global in time flow of equation \eqref{eq:sc0}:
\begin{theorem}
	\label{thm:flow2}
Assume  that   $L(t)$ fulfills the assumptions {\rm (H0), (H1), (H2)}.
	Then for all $k$ with  $0\leq k \leq \min(k_0, k_1)-4\tm$, equation \eqref{eq:sc0}  admits a unique propagator   
	$\cU(t, s) \in C^0\left(\R \times \R, \ \cL(\cH^{k})\right)$ fulfilling
	\begin{itemize}
		\item[(i)]  {\em Well-posedness:} for every initial datum $\psi_s \in \cH^{k+2\tm}$, there exists a unique global solution 
		$\psi(t):= \cU(t,s)\psi_s \in \cH^{k}$  of \eqref{eq:sc0}  such that
		$\psi(\cdot)\in C^0(\R,\cH^{k+2\tm})\cap C^1(\R,\cH^k).$
		\item[(ii)] {\em Unitarity:} for every initial datum $\psi_s \in \cH^{k}$, the $\cH^0$ norm is preserved by the flow,  $\norm{\psi(t)}_{\cH^0} = \norm{\psi_s}_{\cH^0}$, $\forall t\in\R$.
		\item[(iii)] {\em Group property:} $\forall t, r, s\in\R$ 
		\begin{equation}\label{gr}
		\cU(t,s) = \cU(t,r)\, \cU(r,s),\;\; \; \cU(s,s)=\1 \ . 
		\end{equation}
		\item[(iv)]{\em Upper bound on growth:} for every $k \geq 0$, there  exists  $C_{k}>0$  such that 
		\begin{equation}
		\label{exp.gr}
		\Vert\cU(t,s)\Vert_{{\cal L}(\cH^k)}\leq C_{k}{\rm e}^{C_k\vert t-s\vert},\;\;\forall t, s\in\R.
		\end{equation}
	\end{itemize}
	In particular  $\cU(t,s)$  extends  to a  unitary operator in $\cH^0$  fulfilling the group property $(iii)$.\\
	Furthermore for every  $t\in\R$ and  every $k\geq 2\tm$,  $(L(t), \cH^k, \cH^0)$ is essentially self-adjoint.
\end{theorem}
It is remarkable that the assumptions of Theorem \ref{thm:flow2} are the time-dependent assumptions of Nelson commutator theorem to prove essentially self-adjointness, see Proposition \ref{propone}.  We shall see later that Theorem \ref{thm:flow2} has many  applications  for  proving existence and uniqueness for time dependent Schr\"odinger equations with time dependent  Hamiltonians.
\begin{remark}
For $k=1$ similar  results are proved in \cite[Appendix A]{lewin},   \cite[Theorem II.27]{sim}  and \cite{ki}. 
\end{remark}

\begin{remark}
\label{rem:exp.gr} 
At this level of generality, the estimate on the growth of Sobolev norms of Theorem \ref{thm:flow2} (iv)  is optimal. Indeed one  example is the following. Let 
 $H = -\frac{d^2}{dx^2} +x^2$ be the harmonic oscillator   and $L = x\frac{d}{\im dx} + \frac{d}{\im dx}x$  on $L^2(\R)$.  We have 
 $[H, L] = -2\left(\frac{d^2}{dx^2} + x^2\right)$ and the assumptions (H0)--(H2) are satisfied. But  we have
 $$
\cU(t,0)u(x) = {\rm e}^{t/2}u({\rm e}^tx).
 $$
 So we get 
 $$
 \int_\R\vert\frac{d}{dx}{\rm e}^{\im tL}u(x)\vert^2dx = {\rm e}^t\Vert\frac{d}{dx}f(x)\Vert_{L^2(\R)}.
 $$
\end{remark}

\vspace{1em}
A first improvement on the growth \eqref{exp.gr} can be obtained by asking that the commutator  $[L(t), H]$ is more regular than what is assumed in (H2). More precisely    we introduce the following assumption:
\begin{itemize}
\item[(H3)]   There  exist  $k_1>2\tm$ and  $ \tau <1$  real     such that  
$[L(\cdot), H] H^{-\tau} \in C_b\left(\R, \cL(\cH^{k})\right)$   for every 
$0\leq k \leq 2k_1$. 
\end{itemize}
\begin{theorem}
	\label{thm:norm}
	(i) Assume that $L(t)$  satisfies  the properties   {\rm (H0), (H1)} and {\rm (H3)}.   Let   $0\leq k \leq \min\{k_0, 2k_1\}-4\tm$, 
	and $ p\in \N$ such that 
$$
\frac{k}{1-\tau} \leq p \ .
$$	
	Then there exists  a positive constant $C_{k,\nu,p}$, independent of $t$,  such that 
	\begin{equation}
	\label{bound}
	\norm{\cU(t,s) \psi_s}_{{2k} } \leq C_{k,\tau, p} \, \la t- s \ra^{p} \norm{\psi_s}_{{2k}},\;\forall \psi_s\in\cH^{2k} \ . 
	\end{equation}
	 (ii) Assume that {\rm (H3)}  is satisfied for every $k\in\N$  and that $\tau < 1 $ is rational. Then     for every real $r>0$  we have
\beq
\label{bound3}
\norm{\cU(t,s)\psi_s}_{r} \leq C_r \, \la t- s \ra^{\frac{r}{2(1-\tau)}} \norm{\psi_s}_{r} \ . 
\eeq
	 
		\end{theorem}

This result shows that if $[L,H]$ is $H^{\tau}$-bounded with $\tau <1$, then the growth of the Sobolev norm is at most polynomial in time.
\begin{remark}
\label{rem:del}
At this level of generality, the bound obtained in \eqref{bound} is optimal, at least for $\tau = 0$. Indeed Delort \cite{del1} proved that there exists a time-dependent pseudodifferential operator $V(t)$ of order $0$ such that the propagator of the equation 
$\im \dot \psi = (-\Delta + |x|^2) \psi + V(t)\psi$, $x \in \R$,  fulfills $\norm{\cU(t,s) \psi}_{\cH^k} \sim |t-s|^{k/2}$ (where $\cH^k := D((-\Delta + |x|^2)^{k/2})$). In such example, $H=-\Delta + |x|^2$, and  condition (H3) is fulfilled with $\tau = 0$. Then one sees that \eqref{bound} is optimal.
\end{remark}

\vspace{1em}

In order to improve further the polynomial growth in \eqref{bound},  we make more restrictive assumptions on the structure of $L(t)$. First we ask that $L(t)$ is a perturbation of $H$, i.e. $L(t) =  H+V(t)$,
where $V(t)$ is a time-dependent self-adjoint operator. Clearly we assume that $L(t)$ satisfies  (H0), (H1) and (H2)   (in particular we can take $\tm = 1$).  Then we know from Theorem \ref{thm:flow2} that  the Hamiltonian $L(t):=H+V(t)$ generates  a   propagator
   in each space $\cH^k$, $k\in\N$, $k \leq \min\{k_0, k_1\}-2$, which is unitary in $\cH^0$. \\ 
    We make two further  assumptions. The first one concerns the structure of the spectrum of $H$, which is asked to fulfill the following condition on increasing spectral gaps:
 \begin{itemize}
  \item[(Hgap)] The spectrum $\sigma(H)$ of $H$ can be enclosed in clusters $\{\sigma_j\}_{j \in \N}$, 
  \begin{equation}
    \label{gap1}
 \sigma(H)\subseteq \bigcup_{1\leq j <\infty}\sigma_j \ ,
 \end{equation}
 where each $\sigma_j$  is a  bounded interval  of $\R$ (we assume that they are listed in increasing order).
 Define\footnote{Clearly $\Delta_j$ are the distances between  of the spectral clusters, while $\delta_j$ are  their diameters.}
  $$ \Delta_j := {\rm dist}(\sigma_{j+1}, \sigma_j) \ , \quad \delta_j := \sup_{\lambda_1, \lambda_2 \in \sigma_j}|\lambda_1 - \lambda_2| \ . $$
 
 Then there exist $\mu>0$ and positive constants $\alpha, \beta $ (independent of $j$) such that 
 \begin{align}
 \label{gap2}
    \frac{1}{\alpha} \, j^\mu \leq \Delta_j \leq \alpha \, j^\mu,\qquad \delta_j \leq \beta \, j^\mu \ , \quad \forall j \in \N \ . 
 \end{align}
 \end{itemize}
 \begin{remark}
\label{rem:loc}
If $H$ fulfills {\em (Hgap)}, then its spectrum is localized in the following sense:  there exist positive constants $C_1, C_2$ (independent of $j$) such that 
\begin{equation}
\label{loc.sc}
C_1 \, j^{\mu+1}\leq \min \sigma_j \leq \max \sigma_j  \leq C_2  \, j^{\mu+1} \ , \quad \forall j \in \N \ . 
\end{equation}
In particular
$$
 \max\sigma_j\leq \min\sigma_{j+1},\qquad \max\sigma_j \leq \frac{C_2}{C_1} \, \min\sigma_j \ , \quad \forall j \in \N \ . 
$$
\end{remark} 

 The second assumption concerns the perturbation  $V(t)$:
 
\begin{itemize} 
  \item[(Vs$)_n$]  Let $n\geq 1$.  There  exists  $\nu$ with\footnote{ here $\mu$ is the rate of growth of the spectral gap as defined in \eqref{gap2}}
  $$0 \leq \nu < \frac{\mu}{\mu+1}$$ 
     such that $V(\cdot)H^{-\nu}$  belongs to $C^\infty_b(\R,\cL(\cH^{k}))$ for all $0 \leq k \leq 2n$. In particular   $\forall \ell \geq 0$, there exists a positive $R_{n,\ell}$ s.t.   
\begin{eqnarray}
\label{Vrk}
\sup_{t \in \R} \Vert H^p \ \partial_t^\ell V(t) \ H^{-p-\nu}\Vert_{\cL(\cH^0)} \leq R_{n,\ell} \  ,  \ \ \  \forall 0 \leq p \leq n \ . 
\end{eqnarray}
 \end{itemize}

   The following result  is an extension of Theorem 2 of \cite{nen}.
   \begin{theorem}\label{thm:grest} 
Fix 	an arbitrary $n \geq 0$.    Assume that $H+V(t)$ fulfills {\rm (H0), (H1), (H2), (Hgap)} and {\rm (Vs$)_n$}.  Then for
    any  real $0 < k \leq 2n$ and  every  $\varepsilon>0$ there exists  $C_{k,\varepsilon}$, independent of $t$,  such that 
   \beq\label{ineqcinf0}
   \Vert\cU(t,s)\Vert_{\cL(\cH^{k})} \leq C_{k,\varepsilon}\la t-s\ra^\varepsilon,\;\; \forall t, s\in\R.
   \eeq
   \end{theorem}
    If one assumes that   $ V(t)$ is analytic in time,     better  estimates  were proved for 1-D Hamiltonians \cite{wang08}  or for perturbations of the laplace operator on the torus \cite{bo,zhang}.
We are able to extend such results to our more general situation, provided $V$ fulfills the following analytic estimates:
\begin{itemize}
\item[(Va$)_n$] Let $n \geq 0$. There  exists $\nu$ with   $0\leq \nu< \frac{\mu}{\mu+1}$ such that $V(\cdot) H^{-\nu}$ is an operator in $\cL(\cH^{k})$, $\forall 0 \leq k \leq 2n$, analytic in time. In particular  there  exist  $c_{0,n}, c_{1,n}>0$
 such that  $\forall \ell \geq 0$ 
 \begin{eqnarray}\label{Van}
\sup_{t \in \R}\Vert H^p \ \partial_t^\ell V(t) \ H^{-p-\nu}\Vert_{\cL(\cH^0)} \leq  c_{0,n}\, c_{1,n}^{\ell} \, \ell!   \ , \ \ \ \forall 0 \leq p \leq n \ . 
\end{eqnarray}
\end{itemize}    
   
    Then we have 
       \begin{theorem}
       \label{thm:grest2} 
      Fix 	an arbitrary $n \geq 0$.    Assume that $H+V(t)$ fulfills {\rm (H0), (H1), (H2), (Hgap)} and {\rm (Va$)_n$}.  Then for
    any  real $0 < k \leq 2n$ there exists a  positive  $\gamma$, independent of $t$,  s.t.  
  \begin{equation}
   \label{ineqcinf}
   \Vert\cU(t,s)\Vert_{\cL(\cH^{k})} \leq \gamma \left(\log \la t-s\ra\right)^{\frac{k}{2}\left( \frac{\mu}{\mu+1}-\nu\right)^{-1}},\qquad \forall t, s\in\R.
   \end{equation}
   \end{theorem}

   Notice that Theorem \ref{thm:grest}   and Theorem \ref{thm:grest2} hold true with only time regularity on $V(t)$ and a limited amount of regularity in the  scale spaces $\cH^k$. On the contrary, all the previous results deal with potentials which are smooth or analytic  in the  scale of spaces  $\cH^k$.  In particular in \cite{wang08, zhang}  the authors assume
  analyticity in $t$ and $x$. Here we only need    analyticity in $t$  and some  finite  amount of regularity in $x$.    
   \begin{remark}
   \label{rem:bou}
   The bound on the growth in \eqref{ineqcinf} is sharp at least  in the case $\nu = 0$, $\mu = 1$.  Indeed Bourgain \cite{Bourgain1999} constructed a potential $V(x,t)$ which is analytic in both $x$ and $t$, periodic in both variables, such that the solution of the equation
   $\im \dot \psi = - \partial_{xx} \psi + V(x,t)\psi$, $x \in \T$,  has Sobolev norms fulfilling
   $\norm{\psi(t)}_{H^s} \sim C(\log \la t\ra)^s$. Since $H\equiv -\partial_{xx}$ fulfills {\rm (Hgap)} with $\mu = 1$ and $V$ fulfills {\rm (Va$)_n$} with $\nu = 0$, we see that the bound in \eqref{ineqcinf} is optimal.
   \end{remark}

\begin{remark}  
Theorem  \ref{thm:grest2}   could be extended,  with a different  exponent, replacing analytic estimates (Va$)_n$  by  Gevrey estimates:
\end{remark}
\begin{itemize}
\item[(Vg$)_n$] Fix $n \geq 0$. There exist $0 \leq \nu <\frac{\mu}{\mu+1}$ and $s >1$ s.t. 
\begin{eqnarray}\label{Vams2}
\sup_{t \in \R} \ 
\Vert  H^p \ \partial_t^\ell V(t) \ H^{-p - \nu}\Vert_{\cL(\cH^0)}
 \leq  c_{0,m} \, C_{1,m}^{\ell}\, (\ell!)^s, \; \forall  \ell \geq 0,\; 0\leq p\leq n \ . 
\end{eqnarray}
\end{itemize}

   \subsection{Scheme of the proof} The proof proceeds essentially in three steps. First we  prove Theorem \ref{thm:flow2}.  The strategy is to regularize the operator $L(t)$ obtaining  a sequence of bounded operators $L_N(t)$ for which we are able to prove uniform estimates on the flow they generate, and then to pass to the limit. This in turn is possible  thanks to the  boundedness of $[L(t), H] H^{-1}$.
   Theorem \ref{thm:norm} then follows easily by a recursive argument.
   
The strategy to prove   Theorem \ref{thm:grest} and Theorem \ref{thm:grest2} is to extend the scheme of Nenciu  \cite{nen} to deal with unbounded perturbations. The idea is to construct an adiabatic approximation $\cU_{ad}(t,s)$ of the flow $\cU(t,s)$, for which the norms $\cH^k$ are bounded uniformly in time. In case of time-analytic perturbations, special care is needed in order to perform estimates.

   \vspace{1em}
 \noindent  {\em Organization of the paper:} In Section 2 we prove Theorem \ref{thm:flow2} and Theorem \ref{thm:norm}. In Section 3 we prove the control of the growth of the Sobolev norms in  case of  perturbations depending smoothly in time, namely we prove Theorem \ref{thm:grest}. In Section 4 we consider perturbations  depending analytically in time and we prove Theorem \ref{thm:grest2}. In Section 5 we apply the abstract theorems to different kind of Schr\"odinger equations.

   \vspace{2em}
   \noindent{\bf Acknowledgements.} We thank Joe Viola for useful comments, Mathieu  Lewin for pointing us interesting references
    and Dario Bambusi for several stimulating   exchanges.\\
   The first author is supported by ANR -15-CE40-0001-02 "BEKAM" of the Agence Nationale de la Recherche. 

%

\section{Existence of the propagator}
The aim of this section is to prove Theorem \ref{thm:flow2} and Theorem \ref{thm:norm}.
It is technically more  convenient  to consider the integral form of equation (\ref{eq:sc0})
\beq\label{eq:scint}
 \psi(t) = \psi_s + \im^{-1}\int_s^t L(r)\psi(r)dr
\eeq
We begin with an easy lemma:
\begin{lemma}
\label{lem:comm2}
	 Assume  that the condition {\rm (H3)}  is satisfied. Let $\theta := 1-\tau$. Then \\
	 (i) 
	For  $k\in\N$, $1\leq k\leq k_1$, we have  $[L, H^k]H^{-k+\theta}\in C_b\left(\R, \cL(\cH^0)\right)$.\\
	(ii)  For any $\theta^\prime <\theta  $   and any real  $p$ such that $0<p <k_1$  we have 
	$[L, H^p]H^{-p+\theta^\prime}\in C_b\left(\R, \cL(\cH^0)\right)$
	\end{lemma}
\begin{proof}
	(i) The proof is by induction on $k$. First write 
	$ [L, H^{k+1}] =  [L, H^k] \, H - H^k \, [H, L],  $ 
		 which shows that
	$$
	[L, H^{k+1}] H^{-k-1+\theta} = [L, H^k] \, H^{-k+\theta} - H^k \, [H, L]  H^{-1+\theta} H^{-k} \ .
	$$
	The inductive assumption and the hypothesis $[H, L]  H^{-1+\theta} $ bounded as an operator from 
	$\cH^{\ell} \to \cH^{\ell}$, $\forall 0 \leq  \ell \leq 2k_1$,  imply the inductive assumption. \\
		(ii) For   simplicity let us give the proof for $0< p<1$.
We use the following Cauchy formula
	$$
	H^p\psi  = \frac{1}{2\im\pi}\oint_\Gamma z^{p-1}(H-z)^{-1}H\psi  dz
	$$
	for a suitable  complex contour $\Gamma$. 
	Using that $[L,(H-z)^{-1}] = (H-z)^{-1}[L,H](H-z)^{-1}$  we get
	\bea
	[L,H^p] &=&  \frac{1}{2\im\pi}\oint_\Gamma z^{p-1}(H-z)^{-1}[L,H]H(H-z)^{-1}Hdz + 
	\frac{1}{2\im\pi}\oint_\Gamma z^{p-1}(H-z)^{-1}[L,H]dz \nonumber\\
	&=& I + II
	\eea
	We  have
	$$
	IIH^{-p+\theta^\prime} = \frac{1}{2\im\pi}\oint_\Gamma z^{p-1}(H-z)^{-1}H^{1-p-\theta+\theta^\prime}H^{-s}[L,H]H^{-1+\theta}H^sdz,
	$$
	where $s= \theta^\prime -\theta +1-p$. It results  that  $IIH^{-p+\theta^\prime}$  is bounded on $\cH$ if $\theta^\prime<\theta$.\\
	Using the  same trick   we get  that  $IH^{-p+\theta^\prime} $ is bounded on $\cH$  if $\theta^\prime < \theta$.\\
	The same  proof can be done for $ k <p<k+1$.
\end{proof}
\begin{remark}
It is not clear  that  the above estimate can be proved   under assumption (H3)  with $\theta^\prime=\theta$  if $p$ is not an integer.
\end{remark}
Let $\tm$ as in Theorem \ref{thm:flow2} and suppose further that $\tm >0$ (the case $\tm = 0$ corresponds to bounded $L(t)$).   The main idea of the proof is to regularize $L(t)$ is such a way that it becomes a bounded operator, for which it is possible to construct a unitary flow. To do so, for any $N  \geq 1 $ introduce the smoothing  operator
$$
R_N := \left( 1 + \frac{H^\tm}{N} \right)^{-1} \ .
$$
The following lemma describes the main properties of the  smoothing operator $R_N$.
\begin{lemma}
	\label{lem:RN}
	There exists a positive $C_\tm$ such that $\forall k, N >0  $ one has:
	\begin{itemize}
		\item[(i)]  $R_N: \cH^k \to \cH^{k+2\tm}$ and $\norm{R_N}_{\cL(\cH^{k}, \cH^{k+2\tm})} \leq N$.
		\item[(ii)]  $\norm{R_N}_{\cL(\cH^k, \cH^{k})} \leq C_\tm$.
		\item[(iii)] $\norm{R_N - \uno}_{\cL(\cH^{k+2\tm}, \cH^{k})} \leq \frac{C_\tm}{N}$.
		\item[(iv)] $\norm{R_N - \uno}_{\cL(\cH^{k+2\tm\eta}, \cH^{k})} \leq \frac{C_{\tm,\eta}}{N^\eta}$, $\forall \eta\in]0, 1]. $
	\end{itemize}
\end{lemma}
\begin{proof}

	The proof is an easy computation, and it is skipped. Notice that (iv) follows from (ii) and (iii) using interpolation.
\end{proof}
Now we regularize the operator $L(t)$ by defining
$$
L_N(t) := R_N \, L(t) \, R_N \ . 
$$
\begin{lemma}\label{lem:LN}  
For every $N\geq 1$ ,  $L_N(t)$  is symmetric on $\cH^0$  and bounded 
on $\cH^k$  for $0\leq k\leq k_0-2\tm$.  
Furthermore  for every $\eta\in]0, 1]$  there exists  $C_\eta>0$ such that for $0\leq k\leq k_0-2\tm$  we have: 
	\begin{equation}
	\label{est:LN} 
	\norm{ L_N(t) - L(t)}_{\cL(\cH^{k+2\tm(1+\eta)}, \cH^k)} \leq \frac{C_\eta}{N^\eta},\quad N\geq1, \quad t\in\R  \ . 
	\end{equation}
\end{lemma}
\begin{proof}
	We prove only the estimate. By Lemma \ref{lem:RN} one has 
	\begin{align*}
	\norm{L_N(t) - L(t)}_{\cL(\cH^{k+2\tm(1+\eta)}, \cH^k)} & \leq \norm{R_N L(t) (R_N - \uno)}_{\cL(\cH^{k+2\tm(1+\eta)}, \cH^k)} + \norm{(R_N - \uno) L(t)}_{\cL(\cH^{k+2\tm(1+\eta)}, \cH^k)}\\
	& \leq (\norm{L(t)}_{\cL(\cH^{k+2\tm}, \cH^{k})} + \norm{L(t)}_{\cL(\cH^{k+2\tm(1+\eta)}, \cH^{k+2\tm\eta})})\frac{C_{\tm,\eta}}{N^\eta} \leq \frac{C_\eta}{N^\eta} \ ,
	\end{align*}
	where the last inequality follows from (H0) using that $k+2\tm \leq k_0$.
	\end{proof}
We use $L_N(t)$ as a propagator for a regularized differential equation. More precisely  consider the regularized Schr\"odinger equation
\begin{equation}
\begin{cases}
\im \partial_t \psi =  L_N(t) \psi \\
\psi\vert_{t=s} = \psi_s \ , \quad s \in \R
\end{cases}
\end{equation}
Since the operator $ L_N(t)$ is bounded on $\cH^k$ to itself for every $k$,  $0\leq k\leq k_1$, it generates a flow $\cU_N(t,s)\in C(\R\times \R, \cL(\cH^k))$ 
 for $0\leq k\leq k_1$, which is unitary in $\cH^0$.
\begin{lemma}
\label{lem:UN}
For any  $0<k \leq 2k_1 $,  there exists a positive constant $C_k$, independent of $N$, such that  
	$$
	\norm{\cU_N(t,s)}_{\cL(\cH^k)} \leq \ e^{C_k |t-s|}\ , \quad \forall N  > 0 \ .
	$$
\end{lemma}
\begin{proof}
First we control $\norm{\cU_N(t,s)}_{\cL(\cH^{2k})}$ for $0 \leq k\leq k_1$.	We must show that $H^{k} \,\cU_N(t,s) \, H^{-k}$ is bounded uniformly in $N$ as an operator from $\cH^0$ to itself. Remark that, due to the unitarity of $\cU_N(t,s)$ in $\cH^0$, one has 
	$$
	\norm{H^k \,\cU_N(t,s) \, H^{-k}}_{\cL(\cH^0)} = \norm{\cU_N(t,s)^* \, H^k \,\cU_N(t,s) \, H^{-k}}_{\cL(\cH^0)} \ .
	$$
	Now one has 
	\begin{align*}
	\cU_N(t,s)^* \, H^k \,\cU_N(t,s) \, H^{-k} & = \uno + \int_s^t \cU_N(r,s)^* \, [L_N(r), H^k ] \, \cU_N(r,s) \, H^{-k}dr\\
	& = \uno + \int_s^t \cU_N(r,s)^* \, R_N \, [L(r), H^k ] \, H^{-k} \, R_N \, H^k \, \cU_N(r,s) \, H^{-k}dr
	\end{align*}
	where we used that $[R_N,H^k]=0$  and
	\begin{equation*}
	[L_N(t), H^k] = R_N \, [L(t), H^k] \, R_N \ .
	\end{equation*}
	By Lemma \ref{lem:comm2},  for $0 \leq k \leq k_1$, one has the bound   $\norm{[L(t), H^k ] \, H^{-k} }_{\cL(\cH^0)} \leq C_k$ for some positive constant $C_k$,   thus it follows (using also Lemma \ref{lem:RN} $(ii)$) that  uniformly in $N$
	$$
	\norm{R_N \, [L(t), H^k ] \, H^{-k} \, R_N}_{\cL(\cH^0)} \leq C_k \ , \quad \forall N >0 \  , \qquad 0 \leq k \leq k_1 \ . 
	$$
	Such estimate combined with the unitarity of $\cU_N(t,s)$ in $\cL(\cH^0)$ gives 
	\begin{align*}
	\norm{H^k \,\cU_N(t,s) \, H^{-k}}_{\cL(\cH^0)}  & \leq 1+ \int_s^t \norm{\cU_N(r,s)^* \, R_N \, [L(r), H^k ] \, H^{-k} \, R_N \, H^k \, \cU_N(r,s) \, H^{-k}}_{\cL(\cH^0)}dr \\
	& \leq 1 + C_k \int_s^t \norm{H^k \, \cU_N(r,s) \, H^{-k}}_{\cL(\cH^0)}dr
	\end{align*}
	which by Gronwall allows us to conclude that
	$$
	\norm{\cU_N(t,s)}_{\cL(\cH^{2k})}  \leq e^{C_k |t-s|} \ , \qquad \forall 0 \leq k \leq k_1 \ . 
	$$
	Interpolating with the trivial bound $\norm{\cU_N(t,s)}_{\cL(\cH^{0})} =1$ gives the result for general $k$.
\end{proof}

\begin{proof}[Proof of Theorem \ref{thm:flow2}.]
Fix arbitrary $t,s \in \R$. Choose $\eta>0$  small enough.
	The first step is to show that  for every $\psi\in \cH^{k+2\tm(1+\eta)}$, the sequence $\{ \cU_N(t,s)\psi \}_N$ is a Cauchy sequence in the space  $\cH^k$.
	For $k \leq k_2 \equiv \min\{k_0, 2k_1\}-4\tm$ one has
	\begin{align*}
	\norm{\cU_N(t,s)\psi - \cU_{N'}(t,s)\psi}_{k}   = \norm{\int_s^t \partial_r( \cU_{N'}(t, r) \, \cU_N(r,s)\psi) \, dr}_{k} &\\
	= \norm{\int_s^t  \cU_{N'}(t, r) \, \left( L_N(r) - L_{N'}(r) \right) \cU_N(r,s) \psi  \, dr}_{k}  &\\
	\leq |t-s| \sup_{r \in [s, t]}  \norm{ \cU_{N'}(t, r)}_{\cL(\cH^{k})}  \,\norm{ L_N(r) - L_{N'}(r) }_{\cL(\cH^{k+2\tm(1+\eta)}, \cH^k)}
	\norm{ \cU_{N}(r, s)}_{\cL(\cH^{k+2\tm(1+\eta)})} \norm{\psi}_{k+2\tm(1+\eta)} & \\
	\leq C\left(\frac{1}{N^\eta} + \frac{1}{(N')^\eta} \right) \ |t-s| \  e^{(C_k + C_{k +2\tm(1+\eta)})|t-s|} \ \norm{\psi}_{k+2\tm(1+\eta)}, 
	\end{align*}
	where in the last inequality we used a easy variant of estimate \eqref{est:LN} in  Lemma \ref{lem:LN}.
	For any $t,s$ in a bounded interval, and $\psi\in \cH^{k+2\tm(1+\eta)}$, the sequence 
	$\{ \cU_N(t,s)\psi \}_N \subset \cH^k$ is a Cauchy  sequence.   Since $\cH^{k+2\tm(1+\eta)}$ is  dense in  $\cH^{k}$ and $\norm{ \cU_N(t,s)}_{\cL(\cH^k)} \leq e^{C_k|t-s|}$ uniformly in $N$, by an easy density argument one shows that for any  $\psi  \in \cH^{k}$ the sequence $\{ \cU_N(t,s)\psi \}_N$ is also  Cauchy     in $\cH^k$, $k \leq k_2$. Thus for every $\psi  \in \cH^{k}$ the  limit
	$$
	\cU(t,s)\psi := \lim_{N \to \infty}  \cU_N(t,s)\psi 
	$$
	exists in $\cH^k$, $k < k_2$.  Moreover  we have  the following error estimate, for $N>0$ large enough, 
	\beq
	\label{ineq:est}
	\Vert\cU(t,s)\psi-\cU_N(t,s)\psi\Vert_k \leq \frac{C}{N^\eta}\, \vert t-s\vert\, {\rm e}^{C\vert t-s\vert}\, \Vert\psi\Vert_{k+2\tm(1+\eta)} \ , \quad 0 \leq k \leq  k_2 \ . 
	\eeq
	By the principle of uniform boundedness (Banach-Steinhaus Theorem), $\cU(t,s) \in \cL(\cH^k)$. 
	Since $\cU_N(t,s)$ is an isometry in $\cH^0$,
	$$
	\norm{\cU(t,s)\psi}_{0} = \lim_{N \to \infty} \norm{ \cU_N(t,s)\psi }_{0} = \norm{\psi}_{0}
	$$
	which shows that $\cU(t,s)$ is an isometry on $\cH^0$. \\
	Let us prove  now that $\psi(t) = \cU(t,s)\psi_s$   satisfies the  integral equation \eqref{eq:scint}. Denote 
	$\psi^{N}(t) = \cU_N(t,s)\psi_s$. Then we have 
	\beq\label{eq:Nscint}
 \psi^N(t) = \psi_s + \im^{-1}\int_s^t L_N(r)\psi^N(r)dr \ . 
\eeq
	Using Lemma \ref{lem:LN}   and estimate (\ref{ineq:est}) there  exists $C>0$, depending on $a, b, k$ but not on $N$,  such that for 
	 $a\leq s\leq r\leq t\leq b$, $k \leq  k_2$  we have  
	$$
	\Vert L_N(r)\psi^N(r) - L(r)\psi(r)\Vert_k \leq C\left(\Vert\psi(r) -\psi^N(r)\Vert_{k+2\tm} +\frac{1}{N^\eta}\Vert\psi\Vert_{k+2\tm(1+\eta)}\right).
	$$
	So we can pass to the limit  in (\ref{eq:Nscint})   and we get 
	\beq\label{eq:scint2}
 \psi(t) = \psi_s + \im^{-1}\int_s^t L(r)\psi(r)dr \ . 
\eeq
	In particular if $\psi_s\in\cH^{k+2\tm(1+\eta)}$  then $t\mapsto\psi(t)$  is strongly derivable from $\R$ into $\cH^k$ and satisfies  the Schr\"odinger equation (\ref{eq:sc0}).
		Furthermore
	$$
	\cU(t,s)\psi =  \lim_{N \to \infty}  \cU_N(t,s)\psi = \lim_{N \to \infty}  \cU_N(t,r) \cU_N(r, s)\psi = \cU(t, r) \cU(r, s)\psi 
	$$
 	where the limits are in the $\cH^k$ topology. This shows the group property.
	
	Finally we  have shown  that 
	$(t,s) \mapsto \cU(t,s) \in \cL(\cH^{k+2\tm}, \cH^k)$ is strongly continuously differentiable with strong derivatives
	$$
	\partial_t \cU(t,s)  = - \im L(t) \cU(t,s) \ . 
	$$
	With the same proof we get also
	$$
	\partial_s \cU(t,s) = \im \, \cU(t,s) L(s)  \ . 
	$$
\end{proof}

We now prove the second theorem, concerning the growth of the norms.
\begin{proof}[Proof of Theorem \ref{thm:norm}.]
(i) It is enough to prove (\ref{bound})  for $\psi_s\in\cH^\infty$. We have proved in Theorem \ref{thm:flow2}    that $\cU(t,s)$ is an isometry in $\cH^0$  so  we have
$$
\Vert\cU(t,s)\psi_s\Vert_{2k} = \Vert\cU^*(t,s)\, H^k \, \cU(t,s)\psi_s\Vert_0.$$
 But  we have
 $$
 \cU^*(t,s)\, H^k\, \cU(t,s)\psi_s = H^k\psi_s + \im^{-1}\int_s^t \cU^*(r,s)\, [L(r), H^k]\, \cU(r,s)\psi_s \, dr
 $$
 Hence using assumption (H3)    and Lemma \ref{lem:comm2},  we get the first estimate
\beq\label{step1}
 \Vert\cU(t,s)\psi_s\Vert_{2k} \leq \Vert\psi_s\Vert_{2k} + C_k\int_s^t \Vert\cU(r,s)\psi_s\Vert_{2(k-\theta)}dr
\eeq
After $m$  iterations of (\ref{step1}), with another constant $C_{k,m}$,   we get 
\bea\label{stepm}
 \Vert\cU(t,s)\psi_s\Vert_{2k} \leq 
C_{k,m}\left(\Vert\psi_s\Vert_{2k} + \vert t-s\vert(\Vert\psi_s\Vert_{2(k-\theta)}+\cdots +\vert t-s\vert^{m-1}\Vert\psi_s\Vert_{2(k-(m-1)\theta)} \right) \nonumber\\
+C_{k,m}\int_s^t\int_s^{t_1}\cdots \int_s^{t_{m-1}} \Vert\cU(t_m,s)\psi_s\Vert_{2(k-m\theta)}dt_mdt_{m-1}\cdots dt_{1}.
\eea
Now choose  $m$  such that $m\theta \geq k$ in such a way that $\Vert\cU(t_m,s)\psi_s\Vert_{2(k-m\theta)} \leq \Vert\cU(t_m,s)\psi_s\Vert_{0}$. Then use the unitarity of $\cU(t,s)$ in $\cH^0$ to obtain the bound \eqref{bound}.
	
	

(ii) If $\theta =\frac{p}{q}$  we get  the inequality for $r=2k$ with $k=p\ell$, $m=\ell q$ from Theorem \ref{thm:norm}. We conclude by an usual interpolation argument.
\end{proof}

With very similar arguments one can prove the following result about convergence of flows. 
\begin{theorem}
Let $L(t)$ be an operator fulfilling (H0)--(H2) with $k_0 = k_1 = \infty$. Let  $\{L_n(t)\}_{n \geq 1}$ be a sequence of operators fulfilling  (H0), (H1) and (H2) with $k_0 = k_1 = \infty$ uniformly in $n$, namely  $\forall k \geq 0$, there exists $C_k >0$ s.t.
\begin{equation}
\label{conv2}
\sup_{t \in \R} \norm{ [L_n(t), H] H^{-1} }_{\cL( \cH^k)} \leq C_k \ , \quad \forall n  \ .
\end{equation}
Assume that there exists $m \geq 0$ s.t. $\forall k\geq 0$
\begin{equation}
\label{conv}
\sup_{t \in \R} \norm{L_n(t) - L(t)}_{\cL(\cH^{k+m}, \cH^k)} \to 0 , \quad  n \to \infty \ .
\end{equation}
Denote by $\cU_n(t,s)$ the propagator of $L_n(t)$ and by $\cU(t,s)$ the propagator of $L(t)$. Then for every $\psi \in \cH^{k+m}$, for every $t,s \in \R$ fixed, one has 
\begin{equation}
\label{conv3}
\norm{\cU_n(t,s)\psi - \cU(t,s)\psi}_{k} \to 0 , \quad n \to \infty \ .
\end{equation}
\end{theorem}
\begin{proof}
By Theorem \ref{thm:flow2} the flows $\cU_n(t,s)$ and $\cU(t,s)$ are well defined and fulfill (i)--(iv) of Theorem \ref{thm:flow2}. We claim that for every $k \geq 0$,  $\exists \wt C_k >0$ s.t. 
\begin{equation}
\label{conv4}
\norm{\cU_n(t,s)}_{\cL(\cH^k)} \leq {\rm e}^{\wt C_k |t-s|} \ , \quad \forall n \geq 0 \ . 
\end{equation}
Such estimate  follows by arguing similarly to  the proof  of  Lemma \ref{lem:UN} and using  estimate \eqref{conv2}  to estimate $[L_n(t), H^k] H^{-k}$. We skip the details.
Now we have
\begin{align*}
	\norm{\cU_n(t,s)\psi - \cU(t,s)\psi}_{k}   
	= \norm{\int_s^t  \cU(t, r) \, \left( L_n(r) - L(r) \right) \cU_N(r,s) \psi  \, dr}_{k}  &\\
	\leq |t-s| \sup_{r \in [s, t]}  \norm{ \cU(t, r)}_{\cL(\cH^{k})}  \,\norm{ L_n(r) - L(r) }_{\cL(\cH^{k+m}, \cH^k)}
	\norm{ \cU_{n}(r, s)}_{\cL(\cH^{k+m}) }\norm{\psi}_{k+m} & \\
	\leq \sup_{r \in [s, t]}\norm{ L_n(r) - L(r) }_{\cL(\cH^{k+m}, \cH^k)} \ |t-s| \  e^{(C_k + \wt C_{k +m})|t-s|} \ \norm{\psi}_{k+m}, 
	\end{align*} 
	which converges to $0$ by \eqref{conv}.
\end{proof}

\section{Growth of norms for perturbations smooth in time}

In this section we prove  Theorem \ref{thm:grest}. 
First we show that under assumptions (Hgap) and  (Vs$)_n$, the operator $H+V(t)$ satisfies a spectral gap property. Then we describe the algorithm  which will allow us to construct an adiabatic approximation $\cU_{ad}(t,s)$ of the flow of the operator $H+V(t)$. Here we  follow the strategy of  \cite{nen}, adding analytic estimates to the construction. Finally we show how to use the adiabatic approximation $\cU_{ad}(t,s)$ to control the growth of the Sobolev norm.\\
   
  \subsection{Spectral properties of $H+V(t)$} 
It is more  convenient to  have dyadic gaps  between the clusters, so we define a new sequence of clusters as  follows.  Fix a large integer $\tJ \geq 1$ (to be chosen later on). Define the new clusters
   \begin{align}
\wt \sigma_1 := \bigcup_{1 \leq l \leq 2^{\tJ}} \sigma_l \ , \qquad \wt \sigma_j = \bigcup_{2^{\tJ+j-2} + 1 \leq l \leq 2^{\tJ + j-1}} \sigma_l \qquad  \mbox{ for } j \geq 2 \ . 
\end{align}
We define as well
\begin{align}
\wt\Delta_j :=  {\rm dist}(\widetilde\sigma_{j+1}, \widetilde\sigma_j)  \ , \qquad \wt\delta_j := \sup_{\lambda_1, \lambda_2 \in \wt\sigma_j}|\lambda_1 - \lambda_2| \ , \\
\label{lambda.def}
\lambda_j^+= \max_{\lambda \in \wt\sigma_j} \lambda \ , \quad \lambda_j^-=\min_{\lambda \in \wt\sigma_j} \lambda \ . 
\end{align}
 So condition (Hgap)   is  written   now  as
\begin{itemize}
\item[$\widetilde{{\rm (Hgap)}}$] The spectrum of $H$ fulfills  $\sigma(H)\subseteq \bigcup_{1\leq j <+\infty}\widetilde\sigma_j$ and there exist positive constants $\wt\alpha, \wt\beta$ (independent of $\tJ$) s.t. $\forall j \in \N$
\begin{align}
\label{diamcluster}
\wt\alpha^{-1} \, 2^{(\tJ+j-1)\mu} \leq \wt\Delta_j \leq \wt\alpha \, 2^{(\tJ+j-1)\mu} , \qquad \wt\delta_j \leq \wt\beta \,  2^{(\tJ+j-1)(\mu+1)}  \ . 
\end{align}
\end{itemize} 

\begin{remark}
\label{rem:loc2}
Let  $H(t)$ be an operator  fulfilling $\wt{{\rm(Hgap)}}$ uniformly in time $t\in\R$. Then  there exist positive constants $\wt{C}_1 , \wt{C}_2$ (independent of $\tJ, j$) such that 
\begin{equation}
\label{loc2}
\begin{aligned}
&\lambda_1^+ \leq \wt{C}_2 \, 2^{\tJ(\mu+1)}  \ , \\
 & \wt{C}_1 \, 2^{(\tJ+j-1)(\mu+1)}\leq \lambda_{j}^- \leq \lambda_{j}^+  \leq \wt{C}_2 \, 2^{(\tJ+j-1)(\mu+1)} \ , \quad \forall j \geq 2 \ .
 \end{aligned}
\end{equation}
In particular  we have the very useful property
\beq
\label{eqenergy}
\max \wt\sigma_j \leq \frac{\wt C_2}{\wt C_1} \min \wt\sigma_j \ , \qquad \forall j \in \N \ . 
\eeq
\end{remark}

We will  denote by $\Gamma_j$, $j \geq 1$,  an  anti clock-wise oriented rectangle in the complex plane which isolates the cluster $\wt\sigma_j$, that is  $\Gamma_j$ contains  only  $\wt\sigma_j$ at its interior. We fix such contours so that 
\begin{equation}
\label{gamma}
\inf_{\lambda \in \Gamma_1} {\rm dist}(\lambda, \sigma(H)) \geq  \frac{\wt\Delta_{1}}{2} \ , \qquad 
\inf_{\lambda \in \Gamma_j} {\rm dist}(\lambda, \sigma(H)) \geq  \frac{\wt\Delta_{j-1}}{2} \ , \qquad j \geq 2 \ .
\end{equation}
 Finally define
\begin{equation}
\label{delta}
\delta := 1- \frac{\mu+1}{\mu} \nu \ .
\end{equation}
It is important to remark that  by our assumptions $0 < \delta \leq  1$. \\

We prove now a  perturbative result.  It is in this lemma which enters into play the restriction $\nu < \frac{\mu}{\mu+1}$. This is indeed the condition which guarantees that the operator $H + V(t)$ has a spectrum with increasing spectral gaps.
 \begin{lemma}
 \label{perturb}
 Let  $H$ satisfy {\rm (Hgap)} and  $V(t)$ satisfy {\rm (Vs$)_n$} for some $n \geq 0$. There exists a constant $C_H$ (depending only on $H$), such that if $\tJ$  is large enough to fulfill
 \begin{equation}
 \label{condz}
2^{\tJ \mu \delta } \geq 2^{4} \,  C_H \,  \sup_{t \in \R}\norm{V(t) \, H^{-\nu}}_{\cL(\cH^0)} 
\end{equation}
then  $H+V(t)$  fulfills
$\widetilde{{\rm (Hgap)}}$ uniformly in $t\in\R$, with new clusters
\begin{equation}
\label{new.clus}
\wt\sigma_{j}' = [\lambda_j^- - \frac{\wt\Delta_{j-1}}{4}] \cup \wt\sigma_j \cup  [\lambda_j^+ + \frac{\wt\Delta_{j-1}}{4}]  \ , \quad j \in \N \ . 
\end{equation}
Here we defined $\wt\Delta_0 := \wt\Delta_1$.
 \end{lemma}
\begin{proof}
We show that  any $z \in \bigcup_j [\lambda_j^+ +\frac{\wtDelta_j}{4} , \, \lambda_{j+1}^- - \frac{\wtDelta_j}{4}]$ belongs to the resolvent set of $H+V(t)$.
For  $z\in\C\backslash\R$ we have
$$
H+V(t) -z = \left(V(t)(H-z)^{-1} +\1\right)(H-z)= \left([V(t)\, H^{- \nu}]\, [H^{\nu}\, (H-z)^{-1}] +\1\right)\, (H-z) \ .
$$
By spectral decomposition
\beq\label{specdecomp}
H^{\nu}(H-z)^{-1} = \sum_{j\geq 1}\int_{\Gamma_j}\frac{\zeta^\nu}{\zeta -z} \, dE_H(\zeta)
\eeq
where $\{E_H(\zeta)\}_{\zeta\in\R}$  is the  spectral decomposition of $H$. One has 
$$
\norm{H^{\nu}(H-z)^{-1} }_{\cL(\cH^0)} = \sup_{\zeta \in \sigma(H)} \abs{\frac{\zeta^\nu}{\zeta - z} }  \leq \Big(1+\frac{z}{{\rm dist}(z, \sigma(H))} \Big)^\nu  \ \frac{1}{{\rm dist}(z, \sigma(H))^{1- \nu}} \ . 
$$
Fix $z \in \bigcup_j [\lambda_j^+ +\frac{\wtDelta_j}{4} , \, \lambda_{j+1}^- - \frac{\wtDelta_j}{4}]$. Then  (using also \eqref{diamcluster}, \eqref{loc2})
$$
\norm{H^{\nu}(H-z)^{-1} }_{\cL(\cH^0)} \leq  \frac{4 \ (\lambda_{j+1}^-)^\nu}{\wt\Delta_{j}} \leq 4\,  \wt\alpha \,  \wt C_2^\nu \,  2^\mu \, 2^{(\tJ +j) [(\mu+1)\nu - \mu ] } \leq C_H \, 2^{-(\tJ +j)\mu \delta} \ , 
$$
where $\delta  >0$ is defined in \eqref{delta}.
Thus provided  \eqref{condz} holds one has 
\begin{equation}
\label{cond00a} 
\sup_{t \in \R}\norm{V(t) \, H^{-\nu}}_{\cL(\cH^0)} \ \norm{H^{\nu}(H-z)^{-1} }_{\cL(\cH^0)} \leq 1/2 
\end{equation}
and we can invert $[V(t)\, H^{- \nu}]\, [H^{\nu}\, (H-z)^{-1}] +\1$ by Neumann series and define the resolvent
$$
R_V(t, \lambda) := (H- z)^{-1} \, \left([V(t)\, H^{- \nu}]\, [H^{\nu}\, (H-z)^{-1}] +\1\right)^{-1}
 \ . 
$$
This shows that 
any $z \in \bigcup_j [\lambda_j^+ +\frac{\wtDelta_j}{4} , \, \lambda_{j+1}^- - \frac{\wtDelta_j}{4}]$ belongs to the resolvent set of $H+ V(t)$, $\forall t \in \R$. 
Thus
$$
\sigma(H+V(t)) \subset \bigcup_{j \geq 1} \wt\sigma_j' \ . 
$$
The lemma follows easily.
Notice that we get in particular that    for every $t\in\R$, $H+V(t)$ is self-adjoint  on the domain $D(H)$ of $H$. 
 \end{proof}
  \begin{remark}
  One has that $\wt\Delta_j':= {\rm dist}(\wt\sigma_{j+1}', \wt\sigma_{j}')$, $\wt\delta_j':= \sup_{\lambda_1, \lambda_2 \in \wt\sigma_j'} |\lambda_1 - \lambda_2|$ fulfill \eqref{diamcluster} with new constants $\wt\alpha$, $\wt\beta$. 
  \end{remark} 
   {\bf In the following we will always use the clusters $\wt\sigma_j'$'s.} By abusing the notation we will suppress the up-script $'$ and write only $\wt\sigma_j \equiv \wt\sigma_j'$.\\
   
    \begin{lemma} 
    \label{lem:est:res0}
There exists  $\wt C_H>0$, independent on $j, \tJ$, such that  for all $ j \geq 1$ 
\begin{equation}
\label{est:res0}
\sup_{z \in \Gamma_j} \Vert H^\nu(H-z)^{-1}\Vert_{\cL(\cH^0)}  \leq \frac{\wt C_H}{\wt\Delta_{j-1}^\delta} \ , 
\end{equation}
where $\delta$ is defined in \eqref{delta}.
\end{lemma}
\begin{proof}
We show that there exists a constant $\wt C>0$, independent on $j$,$\tJ$,   s.t.  for every $z \in \Gamma_j$, 
\begin{equation}
\label{est:res1}
 \Vert H^\nu(H-z)^{-1}\Vert_{\cL(\cH^0)} \leq \wt C\frac{2^{(\tJ+j-1)\nu}}{{\rm dist}(z,\sigma(H))^{1-\nu}} \ . 
\end{equation}
Then \eqref{est:res0} follows easily using \eqref{gamma} and \eqref{diamcluster}.\\
To prove \eqref{est:res1}, recall that $\norm{H^{\nu}(H-z)^{-1} }_{\cL(\cH^0)}  = \sup_{\zeta \in \sigma(H)} \frac{|\zeta|^\nu}{|\zeta - z|}  $ and write
$$
\frac{\vert\zeta\vert^\nu}{\vert\zeta-z\vert} = 
\left(
\frac{\vert\zeta\vert}{\vert\zeta-z\vert}\right)^\nu\frac{1}{\vert\zeta-z\vert^{1-\nu}} \ . 
$$
Let $z\in\Gamma_j$. If $\zeta \in\wt\sigma_j$ we have by \eqref{loc2}   and  \eqref{diamcluster} 
$$
\frac{\vert\zeta\vert}{\vert\zeta-z\vert} \leq \wt C_2 \frac{2^{(\tJ + j -1)(\mu+1)}}{\wt\Delta_{j-1}} \leq C_3 \,  2^{(\tJ+j-1)} \ , 
$$
where $C_3 >0$ is independent of $j, \tJ$.\\
Now if $\zeta \in\wt\sigma_{j^\prime}$, $j^\prime\neq j$ then 
 $\zeta\approx 2^{(\tJ+j^\prime-1)(\mu+1)}$  and there exists $ C_4>0$ s.t.  $\vert\zeta-z\vert \geq \wt C_4 2^{(\tJ+j^\prime-1)(\mu+1)}$ (notice that ${\rm length}(\wt\sigma_k)\geq c2^{(\tJ+k-1)(\mu+1)}$)  so 
$$ 
\frac{\vert\zeta\vert}{\vert\zeta-z\vert} \leq \wt C_4.
$$
Hence \eqref{est:res1}  follows with $\wt C = \max (C_3, C_4)$.
\end{proof}

   \subsection{Adiabatic approximation}
Let us start now the adiabatic approximation as  explained  in \cite{ nenf, nen,joy,joy0}.\\ 
We present first the formal construction. In a second step we perform analytic estimates to prove that all the objects are well defined.\\

The idea is to construct a sequence of operators $B_m(t)$ such that for every $m \geq 0$ the flow  $\cU_{ad,m}(t,s)$ of $H + V(t) - B_m(t)$ is adiabatic, in particular  it preserves   the $\cH^k$-norm, and $B_m(t)$ is a more and more regularizing operator  in a suitable sense.  The $B_m(t)$  are  constructed step by step  such that at each step we have an adiabatic  transport  for spectral projectors. Let us recall here the  adiabatic approximation used at each step following \cite{nenf, nen}.\\
Consider $H_W(t) = L(t)+W(t)$ a perturbation of $L(t) := H+V(t)$  such that $\sigma({H_W})\subseteq\bigcup_{j\geq 1}\sigma^W_j$,
 a   splitting of the spectrum of $H_W(t)$ into  clusters $\sigma^W_j$, uniform in time $t\in\R$. $\Pi_j^W(t)$  denotes the spectral projector of $H_W(t)$  onto $\sigma^W_j$.  We are looking  for an adiabatic transport for all the $\{\Pi_j^W(t)\}_{j\geq 1}$which means that we want to find an Hamiltonian $H_{ad}(t) = L(t) - B(t)$ (a   "small" perturbation of $L(t)$) such that 
 \beq\label{adia0}
  \Pi_{m,j}(s) = \cU_{ad,m}^*(t,s)\, \Pi_{m,j}(t)\, \cU_{ad,m}(t,s),\qquad \forall t, s\in\R, j\geq 1.
   \eeq
Taking the time derivative we see that (\ref{adia0})  is  satisfied if and only if 
\beq\label{adia1}
\im[B,\Pi^W_j] = \partial_t\Pi^W_j +\im [L,\Pi^W_j] := F_j.
\eeq
It is not difficult to solve the homological equation (\ref{adia1})  using the decomposition
$\di{B = \sum_{k,k^\prime \geq 1}\Pi_{k}^WB\Pi_{k^\prime}^W}$.
First note that, by the properties of orthogonal projectors, one has $\Pi_k^W \, F_j \,\Pi_{k^\prime}^W = 0$ $\forall k\neq j$, $k^\prime \neq j$ and  
$\Pi_j^W \, F_j  \, \Pi_{j}^W = 0$, hence there are no diagonal terms in the homological equation.
 We can thus assume that   $\Pi_{k}^WB\Pi_{k^\prime}^W = 0$ if $k, k^\prime \neq j$ and
   we have $\Pi_j^WB\Pi_{k^\prime}^W = \im\, \Pi_j^WF_{j} \Pi_{k^\prime}^W$. \\
  So,  by a computation using that 
 $\{\Pi^W_k\}_{k\geq1}$  is a complete family  of orthogonal projectors, we get 
 \beq\label{adia2}
 B= \im \left(\sum_{k\geq 1}\Pi_{k}^W\left(\partial_t\Pi_{k}^W +\im[L,\Pi_{k}^W]\right)\right) \ . 
 \eeq
 The Nenciu algorithm \cite{nenf} is obtained by iterating this  formal computation:
 \bea
 W\longrightarrow W+B,\;\;\;\; 
 \Pi_{k}^W\longrightarrow \Pi_{k}^{W+B} \nonumber\\
 B_{new}  =   \im\left(\sum_{k\geq 1}\Pi_{k}^{W+B}\left(\partial_t\Pi_{k}^{W+B}+\im [L,\Pi_{k}^{W+B}]\right)\right) \ . 
 \eea
We describe now how to construct the  $B_m(t)$'s. A   sequence $H_m(t)$   of  perturbations  of $L(t)$  is constructed 
   by induction as follows: 
\begin{align*}
& H_0(t) := L(t) \\
& H_{m+1}(t) := H_m(t) +  B_m(t) \ , \quad \forall m \geq 0 \ ,
\end{align*}   
 where the $B_m(t)$ are obtained from the spectral projectors of $H_m(t)$. More precisely, we will prove that at each step $\sigma(H_m(t)) \subseteq \bigcup_{j\geq 1} \wt\sigma_j$, where the $\wt \sigma_j$'s are the ones of \eqref{new.clus}. Denote by $\Pi_{m,j}(t)$ the spectral projector of $H_m(t)$ on the cluster $\wt\sigma_j$. Then following (\ref{adia2}) we define
 \begin{equation}
   \label{Bk}
   B_m(t) := \im \sum_{1\leq j <+\infty}\Pi_{m,j}(t) \, \partial_{(t,L)}\Pi_{m,j}(t),
   \end{equation}
 where 
 $$\partial_{(t,L)}A(t) := \partial_t A(t) + \im [L(t), A(t)]$$
  is the Heisenberg derivative  of $A$.\\
So that according (\ref{adia0}) and  (\ref{adia1}) the flow  $\cU_{ad,m}(t,s)$ of  $H + V(t) - B_m(t)$ fulfills the adiabatic property
 $$
\cU_{ad,m}(t,s) \Pi_{m,j}(s) = \Pi_{m,j}(t) \cU_{ad,m}(t,s) \ , \quad \forall j \geq 1 \ , \ \forall t, s \in \R \ , 
$$
(see Lemma \ref{lem:adia} below) and thus it is an adiabatic approximation of the flow $\cU(t,s)$  of $H+V(t)$. The reason to iterate the procedure is that at each step the $B_m(t)$'s are more regularizing operators  (see Corollary  \ref{cor:prop}). \\
 
 Let us give  some technical details  to justify this construction  under our assumptions. 
 Let us denote 
 $$B_{m,j}(t):=\Pi_{m,j}(t)\, \partial_{(t,L)}\Pi_{m,j}(t) \ . $$
  Notice that $B_{0,j}(t) = \Pi_{0,j}(t)\, \partial_t\Pi_{0,j}(t)$, since $[L(t), \Pi_{0,j}(t)] = 0$.\\
 In the following we shall   denote $\Vert\cdot\Vert \equiv \norm{\cdot}_{\cL(\cH^0)}$ the operator norm in $\cL(\cH^0)$.		
\begin{lemma}
\label{estadB}
Under the same assumptions of Theorem \ref{thm:grest}, fix an arbitrary $\tM \in \N$.
 If $\tJ$ is sufficiently large,  for every integers $\ell \geq 0$, $0 \leq m \leq \tM$, $0 \leq p\leq n$,   there  exists  $C_{m,n,\ell}>0$, independent of $\tJ$,  such that  
\begin{equation}
\label{bm}
\sup_{t \in \R}\Vert H^p\partial_t^\ell B_{m,j}(t)H^{-p}\Vert \leq \frac{C_{m,n, \ell }}{\wt \Delta_{j-1}^{\delta(1+m)}},\;\quad \forall j\in\N \ .
\end{equation}
Therefore $B_m(t)$ in \eqref{Bk} is well defined, $\ B_m(\cdot) \in C^\infty_b(\R, \cL(\cH^k))$ for any $0 \leq k \leq 2n$ and
\begin{equation}
\label{bm1} 
\sup_{t \in \R} \norm{H^p \ \partial_t^\ell B_m(t) \ H^{-p} } \leq \frac{\wt C_{m,n, \ell }}{2^{\tJ\mu \delta(1+m)}}  \ , \quad \forall \ell \geq 0, \ 0 \leq p \leq n \ . 
\end{equation}
Finally $B_m(t)$ is a self-adjoint  operator in $\cH^0$.
\end{lemma}

Lemma \ref{estadB} is quite technical and  we postpone its proof  at the end of the section.

\begin{remark}
\label{cor:Bsmooth}
In particular $B_m(t)$ satisfies  the   condition {\rm (Vs$)_n$} (with $\nu=0$).
\end{remark}

   Define    for $m\geq 1$ the operators 
   \begin{align*}
   & H_m(t) := H +V(t) + B_0(t) +\cdots +B_{m-1}(t) \equiv L(t)+W_m(t) \\
   & H_{ad,m}(t) := H + V(t) - B_m(t) \equiv L(t) - B_m(t) 
   \end{align*}
The following corollary follows immediately from Lemma \ref{estadB}:
  \begin{corollary}
  \label{cor:hgap}
Fix $\tM \geq 1$. Then provided $\tJ$ is sufficiently large, the following holds true:
\begin{itemize}
\item[(i)]  For every $1 \leq m \leq \tM$ and $t \in \R$,  the operators $H_m(t) $ and $H_{ad,m}(t)$ are  self-adjoint operators   generating a unitary flow in $\cH^0$.
\item[(ii)] For every $1 \leq m \leq \tM$ and $t \in \R$,  $H_m(t)$ and $H_{ad,m}(t)$ fulfill
  $\wt{{\rm(Hgap)}}$  uniformly in time $t\in\R$  with $\wt\sigma_j$'s as in Lemma \ref{perturb}. 
\end{itemize}   
  \end{corollary}
  \begin{proof}
  $(i)$ By Lemma \ref{estadB} $\forall 0 \leq m \leq \tM$  the operator $B_m(t)$ ia a  bounded self-adjoint operator. Hence  $H_m(t) = L(t) + W_m(t)$ and $H_{ad,m}(t) = L(t) - B_m(t)$  are  bounded perturbations of $L(t)$, and thus they are  self-adjoint operators   generating a unitary flow in $\cH^0$.\\
$(ii)$ Write $H_m(t)$ as   $H_{m}(t) \equiv H + W(t)$ with $W(t):=V(t) + B_0(t) + \cdots + B_m(t)$.  By (Vs$)_n$ and  \eqref{bm1} it fulfills 
$$
\sup_{t \in \R} \norm{H^p \, \partial_{t}^\ell W(t) \, H^{-p-\nu}} \leq  R_{n,\ell} + (m+1)\frac{\wt C_{m,n,\ell}}{2^{\tJ \mu \delta}} \leq 2 R_{n,\ell} \ ,
$$
provided $\tJ$ is sufficiently large (depending on $\tM$). Then  Lemma \ref{perturb} gives  the claim.\\
The proof for $H_{ad,m}(t)$ is analogous.
  \end{proof}

We will denote by $\cU_{ad,m}(t,s)$ the propagator of $H_{ad,m}(t)$.
  The two key points, proved in Corollary \ref{cor:prop} below,  are the following:
  \begin{itemize}
  \item[(i)]  $\cU_{ad,m}(t,s)$  is an adiabatic approximation of $\cU(t, s)$  which  preserves the $\cH^k$-norms.
  \item[(ii)] the operators $B_m$'s are smoothing operators.
  \end{itemize}
  In order to prove those two properties it is convenient to measure the $\cH^k$-norm   with the help of the  projectors $\Pi_{m,j}$'s. More precisely perform the construction at order $m$. 
Introduce the block diagonal operator
  $$
  \Lambda_m(t) := \sum_{1\leq j <\infty} \, 2^{(j-1)(\mu+1)} \ \Pi_{m,j}(t) \ . 
  $$
As the $\Pi_{m,j}$'s are  orthogonal projectors one has that 
$$
\norm{\Lambda_m(t) \psi}_0^2 = \sum_{j \geq 1} 2^{2(j-1)(\mu+1)} \, \norm{\Pi_{m, j}(t) \psi}_0^2 \ , \qquad \forall \psi \in \cH^0 \ .  
$$
The next lemma shows that  the norm  $\norm{H^p \cdot}$ is equivalent to the  norm   $\norm{\Lambda_m(t)^p \cdot}$:
  \begin{lemma}
  \label{compk} Fix  $n\in\N$, $n\geq 0$. 
  Assume that  {\rm (Vs$)_{n}$}  is satisfied. Then for any $1 \leq m \leq \tM$,  there exist positive  $c_1$ and $c_2$, depending  on $n, V, \sigma(H), \norm{H^n B_i(t) \ H^{-n}}$,    such that    $\forall 0 \leq p \leq n$,  $\forall \psi\in\cH^{2p}$, $\forall t\in\R$
 \begin{align}
 \label{Hadk}
 c_1 \Vert\psi\Vert_{2p} \leq & \Vert(H_{ad,m}(t) +c_0)^p \psi \Vert_0 \leq  c_2 \Vert\psi\Vert_{2p}  \\ 
 \label{E9} 
 c_1 \, 2^{\tJ p(\mu+1)}\,  \Vert\Lambda_m^p(t)\psi\Vert_0 \leq
 & \Vert(H_{ad,m}(t)+c_0)^p\psi\Vert_0 \leq c_2 \, 2^{\tJ p(\mu+1)}\,  \Vert\Lambda_m^p(t)\psi\Vert_0 \ . 
\end{align}

\end{lemma}
The proof is postponed in Appendix \ref{app:normcontrol}.\\

We prove now some properties of the adiabatic evolution.

  \begin{lemma}\label{lem:adia}
  For every integer $0 \leq m \leq \tM$ and $j \geq 1$ we have
  \beq\label{adia}
  \Pi_{m,j}(t) = \cU_{ad,m}(t,s)\, \Pi_{m,j}(s)\, \cU_{ad,m}(t,s)^*,\qquad \forall t, s\in\R.
   \eeq
  \end{lemma}
 \begin{proof}
 For any  propagator $\cU(t,s)$ with generator  $t \mapsto L(t)$ of class $C^1$  and any  $C^1$  and bounded  operator $A(t)$ we have 
 $$
 \partial_t\left(\cU(t,s)^*\, A(t)\, \cU(t,s)\right) =   \cU(t,s)^*\,\partial_{(t,L)}A(t)\,\cU(t,s).
 $$
Since the generator of $\cU_{ad,m}(t,s)$ is $L(t) - B_m(t)$,   it is enough to prove that 
\beq\label{derheis}
 \partial_{(t,L-B_m)}\Pi_{m,j}(t) =0,\quad \forall t\in\R \ , \qquad \forall j \geq 1 \ , \  m \geq 0 \ .
 \eeq
 This follows easily using  the definition of $B_m(t)$   and  properties  of  orthogonal  projectors. 
 \end{proof}

\vspace{1em}
    
\begin{corollary}
\label{cor:prop} 
(i) For every $0 \leq m \leq \tM$, $\cU_{ad, m}(t,s)$ preserves the $\cH^k$-norms. More precisely for every $0 \leq p \leq n$, there exists $C_p >0$ s.t.  
$$
\norm{\cU_{ad, m}(t,s)}_{\cL(\cH^{2p})} \leq C_p \, \qquad \ \forall t,s \in \R \ .
$$
(ii)  For every $0 \leq m \leq \tM$, $B_{m}(t): \cH^0 \mapsto \cH^{2p}$ provided $p < \frac{\mu}{\mu+1} \delta (1+m)$.
\end{corollary}    
\begin{proof}
(i) First note that by Lemma \ref{lem:adia} one has that
$\Lambda_m(t) \cU_{ad,m}(t,s) = \cU_{ad,m}(t,s) \Lambda_m(s)$. Then by  Lemma \ref{compk} and the unitarity of $ \cU_{ad,m}(t,s)$ in $\cH^0$ one has
\begin{align*}
\norm{\cU_{ad, m}(t,s) \psi_s}_{2p} & \leq C \norm{ \Lambda_m^p(t) \cU_{ad, m}(t,s) \psi_s}_{0} \leq C \norm{ \cU_{ad, m}(t,s) \Lambda_m^p(s)\psi_s}_{0}  \\
& \leq  C \norm{  \Lambda_m^p(s)\psi_s}_{0} \leq  C \norm{\psi_s}_{2p} \ . 
\end{align*}
(ii) Recall that $\Pi_{m,j}(t) B_m(t) = B_{m,j}(t)$. Now we have
\begin{align*}
\norm{B_m(t) \psi_s}_{2p}^2 & \leq C \norm{ \Lambda_m^p(t) B_m(t) \psi_s}_{0}^2 \leq 
C \sum_{j \geq 1} 2^{(j-1)(\mu+1)2p}  \norm{ B_{m,j}(t)\psi_s}_{0}^2  \\
& \leq  C \norm{\psi_s}_0  \sum_{j \geq 1} 2^{2(j-1)[(\mu+1)p - \mu\delta(1+m)]} \leq  C \norm{\psi_s}_0 \ 
\end{align*}
provided $p < \frac{\mu}{\mu+1} \delta (1+m)$.
\end{proof}

    We are finally ready to prove Theorem \ref{thm:grest}.
    
   \begin{proof}[Proof of Theorem \ref{thm:grest}] 
   
Fix $\epsilon >0$ and  choose $\tM$ such that 
\begin{equation}
\label{M.choice}
\frac{1}{\epsilon} \frac{(\mu+1) n}{ \mu \delta } \leq \tM+1 \ .
\end{equation}
Choose $\tJ$ sufficiently large to perform the construction at step $\tM$.
As the evolution $\cU(t,s)$ is unitary in $\cH^0$ and $\Pi_{\tM,j}(t)$ is a projector we have 
\begin{equation}
\label{est.11}
\norm{\Pi_{\tM, j}(t)\, \cU(t,s) \psi_s}_0 \leq \norm{\psi_s}_0  \ , \quad \forall j\geq 1, \ \ \ \forall t,s \in \R \ . 
\end{equation}  
   We  compare  the evolution $\cU(t,s)$ with the adiabatic evolution $\cU_{ad, \tM}(t,s)$ defined above.  In order to do this, write
   $$
\im \dot \psi = (H + V(t)) \psi = H_{ad,\tM}(t) \psi + B_{\tM}(t)\psi   
   $$
and use the  Duhamel formula
\beq\label{duhaB}
 \cU(t,s)  = \cU_{ad, \tM}(t,s)-i\int_s^t\cU_{ad,\tM}(t,r)\, B_{\tM}(r)\, \cU(r,s)dr.
\eeq
By  equation \eqref{adia},  the property $\Pi_{\tM,j}(t) \,  B_\tM(t) = B_{\tM,j}(t)$ and Lemma \ref{estadB} one has
\begin{equation}
\label{est.12}
\begin{aligned}
\norm{\Pi_{\tM, j}(t)\cU(t,s) \psi_s}_0 &  \leq \norm{\cU_{ad,\tM}(t,s) \, \Pi_{\tM, j}(s)\psi_s}_0 + \norm{\int_s^t\cU_{ad,\tM}(t,r)\, B_{\tM,j}(r)\, \cU(r,s)\psi_s dr}_0  \\
& \leq  \norm{ \Pi_{\tM, j}(s)\psi_s}_0 + \la t-s\ra  2^{-(j-1)(\tM +1)\mu \delta}\norm{\psi_s}_0 \ , 
\end{aligned}
\end{equation}
where in the last line we used that, provided  $\tJ$ is sufficiently large,
$$
   \sup_{t \in \R} \Vert B_{\tM,j}(t)\Vert\leq \frac{C_{\tM,n,0}}{\wt\Delta_{j-1}^{(\tM+1)\delta}} \leq \frac{1}{2^{(j-1)(\tM +1)\mu \delta}}\ ,\qquad  \forall j\geq 1,\; t\in\R.
$$
We compute now the norm of $\cU(t,s)\psi_s$ in $\cH^{2n}$. Fix  $\tN \equiv \tN(t)  $ to be chosen later. By Lemma \ref{compk}
\begin{align*}
\norm{\cU(t,s)\psi_s}_{2n}^2 & \leq  \frac{c_2}{c_1} \, 2^{\tJ(\mu +1)2n} \,  \norm{\Lambda_\tM^n(t)\,  \cU(t,s)\psi_s}_0^2   \leq  \frac{c_2}{c_1} \, 2^{\tJ(\mu +1)2n} (I + II) \ , 
\end{align*}
where
\begin{align*}
I :=\sum_{ 1 \leq j \leq \tN} 2^{(j-1)(\mu+1)2n} \norm{\Pi_{\tM, j}(t)\cU(t,s) \psi_s}_0^2  \ , \quad II:=  \sum_{ j \geq \tN+1} 2^{(j-1)(\mu+1)2n}  \norm{\Pi_{\tM, j}(t)\cU(t,s) \psi_s}_0^2 \ .
\end{align*}
To estimate $I$, use \eqref{est.11} to obtain
\begin{equation}
\label{est80}
I \leq \norm{\psi_s}_0^2 \sum_{ 0 \leq j \leq \tN-1} 2^{j(\mu+1)2n} \leq   \norm{\psi_s}_0^2 \ \frac{2^{\tN(\mu+1)2n} -1}{2^{(\mu+1)2n} -1}
\leq C \norm{\psi_s}_0^2 \ 2^{\tN(\mu+1)2n}\ ,
\end{equation}
where $C$ depends only on $n, \mu$.
To estimate the second summand, we use \eqref{est.12} and Lemma \ref{compk} to obtain
\begin{align}
\notag
II &\leq 4 \sum_{j \geq \tN}  2^{j(\mu+1)2n} \norm{\Pi_{\tM, j}(s)\psi_s}_0^2 + 4\la t-s\ra^2  \norm{\psi_s}_0^2 \ \sum_{j\geq \tN } 2^{2j [(\mu+1)n-(\tM+1)\mu\delta]} \\
\label{est81} 
&\leq 4 \norm{\psi_s}_{2n}^2 + 4  \la t-s \ra^2      \norm{\psi_s}_0^2 \  \frac{2^{[(\mu+1)n-(\tM+1)\mu\delta]2\tN}}{1-2^{2 [(\mu+1)n-(\tM+1)\mu\delta]}}
\end{align}
where we used that $(\mu+1)n/\mu\delta \leq \tM + 1$.
Thus, \eqref{est80} and \eqref{est81} give
\begin{align}
\label{est82}
\norm{\cU(t,s)\psi_s}_{2n}^2 & \leq \wt C\, 2^{\tJ(\mu +1)2n}  \, \norm{\psi_s}_{2n}^2 \, \left[ 2^{\tN(\mu +1)2n} + \la t-s \ra^2       2^{[(\mu+1)n-(\tM+1)\mu\delta]2\tN}\right] \ ,
\end{align}
where $\wt C$ does not depend on $\tN$.
Now choose $\tN(t)$ in such a way to optimize \eqref{est82}, i.e. pick
$$
\tN(t) = \frac{1}{(\tM+1)\mu \delta} \log \la t-s \ra
$$
to obtain  
\begin{equation}
\label{est.13}
\norm{\cU(t,s)\psi_s}_{2n}^2 \leq C \,  2^{\frac{2(\mu+1)n}{(\tM+1)\mu \delta} \log \la t-s \ra} \, \norm{\psi_s}_{2n}^2 \ .
\end{equation}
Using \eqref{M.choice} one has 
$$
\norm{\cU(t,s)\psi_s}_{2n}^2 \leq C \, \la t-s \ra^{\frac{2(\mu+1)n}{(\tM+1)\mu\delta}} \,  \norm{\psi_s}_{2n}^2 \leq 
C \, \la t-s \ra^{2\epsilon} \,  \norm{\psi_s}_{2n}^2 \ ,
$$
which is the desired estimate.

 \end{proof}
We also  get  the following application of the adiabatic approximation concerning the spectra of Floquet operators (\cite{how, nen, joy}).\\
Let assume that  conditions {\rm (Hgap), (Vs$)_n$} are  satisfied  and  suppose that $V(t)$  is periodic with period $T>0$. Denote ${\cal F} := \cU(T, 0)$ the Floquet operator (or monodromy operator). Let us recall that 
$\cU(nT, 0)= {\cal F}^N$ so the spectrum of ${\cal F}$ gives informations on the large time behavior of the propagator.
\begin{theorem}\label{floq}
Let us  assume that  conditions {\rm (Hgap), (Vs$)_n$}  are  satisfied, $V$ is $T$-periodic and that $(H+\im)^{-N}$  is in the trace class  for $N$ large enough. 
Then the Floquet operator ${\cal F}$ has no absolutely continuous  spectrum.
\end{theorem}
\begin{proof}  It results from Lemma \ref{estadB} that $B_m(t)$ is in the trace class  for $m$ large enough.  So from (\ref{duhaB}) we infer  that  $\cU(T,0)  -\cU_{ad, m}(T,0)$ is in the trace class.  It is easy  to see that the Hamiltonians $H_m(t)$ are  $T-$periodic (from the induction construction).  So it results from (\ref{adia})  that   $ \cU_{ad, m}(T,0)$
 commutes with $\Pi_{m,j}(0)$, the spectral projectors of $H_m(T)$. But $(H_m(T) +\im)^{-1}$ is a compact operator
  hence the spectrum of  $ \cU_{ad, m}(T,0)$ is purely  discrete. Applying the Birman-Krein-Kato \cite{bikr}  theorem on the stability of the absolutely spectrum under  class trace perturbations  we get Theorem \ref{floq}.
\end{proof}

\subsection{Proof of Lemma \ref{estadB} }
The proof is by induction.
Through all the proof, we will denote by $C_{m,n,\ell}$ some positive constants which depend on $m,n, \ell$ but not on $j, \tJ$. 

 We will prove \eqref{bm} together with the estimate $\forall j \in \N$
 \begin{align}
  \label{Pm01}
 &\sup_{t \in \R}
 \norm{H^p \,  \partial_t^{\ell+1}  \Pi_{m,j}(t) \, H^{-p} }  \ , \ 
  \sup_{t \in \R}
  \norm{ H^p \, \partial_t^\ell  \ \partial_{(t,L)} \Pi_{m,j}(t) \, H^{-p}} 
  \leq 
 \frac{C_{m,n,\ell}}{\wt\Delta_{j-1}^{\delta }} \ , \quad  \forall 0 \leq p \leq n  , \  \ell \geq 0
  \end{align}
  {\bf Step $m=0$. \ } Recall  that $H_0(t) = H + V(t) \equiv L(t)$.  Provided $\tJ$ is sufficiently large, by Lemma  \ref{lem:an02} the projectors 
   \begin{equation}
   \label{pi0j}
   \Pi_{0,j}(t) := - \frac{1}{2\pi \im} \oint_{\Gamma_j} R_0(t, \lambda) \ d \lambda \ , \qquad \forall j \geq 1 \   
   \end{equation}
   are well defined and fulfill 
\begin{align*}
\sup_{t \in \R}\norm{H^p \, \partial_t^{\ell+1}  \Pi_{0,j}(t) \, H^{-p}} \leq  \frac{C_{0,n,\ell} }{\wt \Delta_{j-1}^{\delta}} \ , \quad \forall 0 \leq p \leq n, \ \ell \geq 0 \ , 
\end{align*}
for some constants $C_{0,n,\ell}$ independent of $j$. 
This proves \eqref{Pm01} for $m=0$.
  Recall that  $B_{0,j}(t) := \Pi_{0,j}(t) \, \partial_t \Pi_{0,j}(t)$. Then by Leibnitz rule and \eqref{Pm01} it follows immediately \eqref{bm} for $m = 0$.\\
  
\noindent   {\bf Step $m\rightsquigarrow m+1$.} Assume that we performed already $m$ steps, with $m < \tM$. Then we constructed the operators  $B_i(t)= \sum_j B_{i,j}(t)$   $\  \forall \, 1 \leq i \leq m$. 
In order to construct $B_{m+1}(t)$, we need the spectral projectors of the operator  $H_{m+1}(t) \equiv H+ V(t) + B_0(t) + \cdots + B_m(t)$ (see formula \eqref{Bk}).
By  Corollary \ref{cor:hgap}  $H_{m+1}(t)$ fulfills  $\wt{{\rm(Hgap)}}$
provided $\tJ$ is sufficiently large. 
Therefore  we can apply  Lemma \ref{lem:an02} and obtain that 
$H_{m+1}(t)$ fulfills  $\wt{{\rm(Hgap)}}$ and that the projectors 
$$
\Pi_{m+1,j}(t) = - \frac{1}{2\pi \im } \oint_{\Gamma_j} R_{m+1}(t, \lambda) \ d\lambda \ , \qquad  R_{m+1}(t, \lambda):= (H_{m+1}(t) - \lambda)^{-1}
$$
are well defined $\forall j \geq 1$  and   fulfill
   \begin{equation}
   \label{est66}
    \sup_{t \in \R} \norm{H^p \, \partial_{t}^{\ell+1}  \Pi_{m+1}(t) \ H^{-p}} \leq  \frac{C_{m+1,n,\ell} }{\wt \Delta_{j-1}^{\delta}} \ , \qquad \forall 0 \leq p \leq n , \ \ell \geq 0 \ . 
   \end{equation}
This proves the first of \eqref{Pm01} for $m+1$. \\
 We pass to estimate $\partial_t^\ell \, \partial_{(t,L)} \Pi_{m+1}(t)$.
Using  the definition of $H_{m+1}$  we get 
\begin{equation}
\label{pr.norma0}
\partial_{(t,L)} \Pi_{m+1,j}(t) = \partial_t \Pi_{m+1,j}(t) - \im \sum_{l=0}^m [B_l(t) , \, \Pi_{m+1,j}(t) - \Pi_{l,j}(t)] - \im \sum_{l=0}^m [B_l(t), \Pi_{l,j}(t)]  \ . 
\end{equation}
Consider the last term in the r.h.s. above.
Note that $\partial:=\partial_{(t,L)}$  is a derivative in the algebra $\cL(\cH^0)$. So  for any projector  $\Pi$   we have 
$\Pi \ \partial\Pi = \partial\Pi - \partial\Pi \ \Pi$. Using the definition of $B_l$ and the properties of the projectors, one gets the identity
$$
[B_l(t), \Pi_{l,j}(t)] = - \im \partial_{(t,L)} \Pi_{l,j}(t) \ .
$$
Therefore using the inductive estimates \eqref{bm1} and \eqref{est66} we get 
\begin{equation}
\label{est67}
\sup_{t \in \R}\norm{H^p \ \partial_t^\ell [B_l(t), \Pi_{l,j}(t)] H^{-p}} \leq \frac{C_{l,n, \ell}'}{\wt\Delta_{j-1}^{\delta}} \ , \quad \forall \ell \geq 0, \ 0 \leq p \leq n \ .  
\end{equation}
Consider now the term in the middle of \eqref{pr.norma0}. To estimate it, remark that 
\begin{equation}
\label{pi-pi}
\Pi_{m+1,j}(t) - \Pi_{l,j}(t) = - \frac{1}{2\pi \im } \oint_{\Gamma_j} \, R_{m+1}(t, \lambda) \, (H_{m+1}(t) - H_l(t) ) \, R_{l}(t, \lambda) \, d\lambda \ .
\end{equation}
As $H_{m+1}(t)- H_{l}(t)= \sum_{k=l}^m B_k(t)$,  by   \eqref{bm1}
$$
\sup_{t \in \R} \norm{H^p \ \partial_{t}^\ell \left(H_{m+1}(t)- H_{l}(t) \right)\  H^{-p}}   \leq  \wt C_{m,n,\ell} , \quad \forall \ell \geq 0, \ 0 \leq p \leq n \ ,
$$
thus we can apply Lemma \ref{lem:an03} and get that
\begin{equation}
\label{est68}
\sup_{t \in \R} \norm{H^p \ \partial_{t}^\ell \left(\Pi_{m+1,j}(t) - \Pi_{l,j}(t)\right) \ H^{-p}} 
\leq    \frac{ \wt C_{m,n, \ell}}{ \wt \Delta_{j-1}}   \   , \quad \forall \ell \geq 0, \ 0 \leq p \leq n \ .  
\end{equation}
Therefore by Leibnitz rule and estimates \eqref{bm1}, \eqref{est68} we find that 
\begin{equation}
\label{est69}
\sup_{t \in \R} \norm{H^p   \left(\partial_{t}^\ell  
\sum_{l=0}^m [B_l(t) , \, \Pi_{m+1,j}(t) - \Pi_{l,j}(t)]\right) \ H^{-p}} 
\leq    \frac{  C_{m,n, \ell}'}{ \wt \Delta_{j-1}}  \  , \quad \forall \ell \geq 0, \ 0 \leq p \leq n \ .  
\end{equation}
We come back to the estimate of $\partial_{(t,L)} \Pi_{m+1,j}(t)$. Using \eqref{est66}, \eqref{est67} and \eqref{est69} we get
\begin{equation*}
\sup_{t \in \R} \norm{H^p \ \partial_t^\ell \, \partial_{(t,L)} \Pi_{m+1,j}(t)  \ H^{-p}} \leq \frac{  C_{m+1,n, \ell}}{ \wt \Delta_{j-1}^\delta} \ , \quad \forall 0 \leq p \leq n \ , \ell \geq 0 \ ,   
\end{equation*}
proving the inductive estimate \eqref{Pm01}. \\

We define now a series of objects and in the next lemma we give the estimates.
Let
\begin{align}
\label{L.def}
L_{m,j}(t) := \Pi_{m+1,j}(t) - \Pi_{m,j}(t) \\
\label{K.def}
K_{m,j}(t) := \Pi_{m,j}(t) \, L_{m,j}(t) \\
\label{D.def}
D_{m,j}(t) := \Pi_{m+1,j}(t)\, L_{m,j}(t)
\end{align}
\begin{lemma}
	For every integers $\ell \geq 0$, $0 \leq p \leq n$,  $ 0 \leq m < \tM$, provided $\tJ$ is sufficiently large, 	one has 
	\begin{align}
	\label{L.norm0}
 \sup_{t \in \R} 	\norm{H^p \ \partial_{t}^\ell  L_{m,j}(t) \ H^{-p}} \leq \frac{ c_{m,n,\ell}}{\wt \Delta_{j-1}}  \ , \\
	\label{K.norm0} 
 \sup_{ t \in \R} 	\norm{ H^p \ \partial_{t}^\ell  K_{m,j}(t) \ H^{-p} } \leq  \, \frac{\wt c_{m,n, \ell}}{\wt \Delta_{j-1}^{(m+1)\delta + 1}} \ , \\
	\label{D.norm0}
 \sup_{ t \in \R} 	\norm{H^p \ \partial_{t}^\ell  D_{m,j}(t) \ H^{-p}} \leq \frac{\widehat c_{m,n,\ell}}{\wt \Delta_{j-1}^{(m+1)\delta + 1}} \ .
	\end{align}
\end{lemma}
\begin{proof}
	First we prove \eqref{L.norm0}. One has 
	\begin{equation}
	\label{L}
	L_{m,j}(t) = - \frac{1}{2\pi \im } \oint_{\Gamma_j} R_{m+1}(t, \lambda) \ B_m(t) \ R_{m}(t, \lambda) \ d \lambda  \ .
	\end{equation}
We apply Lemma \ref{lem:an03} with $P = V + B_0 + \cdots + B_m$, $Q=V + B_0 + \cdots + B_{m-1}$,  $B = B_m$,  and get 
$$
\sup_{t \in \R} \norm{H^p \ \partial_{t}^\ell L_{m,j}(t) \ H^{-p} }  \leq \frac{c_{m,n, \ell}}{\wt \Delta_{j-1}}  \  , \quad \forall \ell \geq 0, \ 0 \leq p \leq n \ .   
$$
	For later use we study  the operator $(1-L_{m,j}(t))^{-1}$. Provided $m < \tM$ and $\tJ$ is sufficiently large, estimate \eqref{L.norm0} with $\ell = 0$ guarantees  that
		$(1-L_{m,j}(t))$ is invertible by Neumann series in $H^p$ and
	$$
	\sup_{t \in \R} \norm{H^p \ (1-L_{m,j}(t))^{-1} \ H^{-p}} \leq 2 \ , \quad \forall 0 \leq p \leq n \ .  
	$$
	To study its derivatives we can proceed as in \eqref{ris.der0},  and get 
	\begin{align}
	\label{der.inv.L0}
	\sup_{t \in \R}\norm{H^p \ \partial_{t}^\ell (1-L_{m,j}(t))^{-1} \ H^{-p}} \leq   \frac{\wt c_{m,n,\ell}'}{\wt \Delta_{j-1}} \ , \quad \forall 0 \leq p \leq n \ , \ell \geq 1 \ ,
	\end{align}
	provided $\tJ$ is sufficiently large.
	
	Estimate  \eqref{K.norm0} follows immediately from the identity
	\begin{equation}
	\label{K}
	K_{m,j}(t) = - \Pi_{m,j}(t) \frac{1}{2\pi \im } \oint_{\Gamma_j} R_{m}(t, \lambda) \ B_m(t) \ R_{m+1}(t, \lambda) \ d \lambda  = - \frac{1}{2\pi \im } \oint_{\Gamma_j} R_{m}(t, \lambda) \ B_{m,j}(t) \ R_{m+1}(t, \lambda) \ d \lambda
	\end{equation}
and the application of  Lemma \ref{lem:an03} with $B=B_{m,j}$, using the bound
$$
 \sup_{t \in \R}  \norm{H^p \ \partial_t^\ell B_{m,j}(t) \ H^{-p}} \leq \frac{c_{m,n,\ell}}{\wt \Delta_{j-1}^{(m+1)\delta}}  \  , \quad \forall \ell \geq 0, \ 0 \leq p \leq n \ .  
$$
which follows from the   inductive assumption.
 
	Finally we prove \eqref{D.norm0}. Using $\Pi_{m+1,j}^2 = \Pi_{m+1,j}$ and simple algebraic manipulations one proves  that   \cite[(2.41)]{nen} 
	\begin{equation}
	\label{D}
	D_{m,j}(t) = \Pi_{m+1,j}(t) \, K_{m+1,j}(t) \, (1-L_{m,j}(t))^{-1} \ . 
	\end{equation}
	Then Leibnitz rule,   \eqref{est66}, \eqref{K.norm0}, \eqref{der.inv.L0} give the claimed estimate.
\end{proof}

\vspace{1em}

We can now conclude the proof by calculating the norm of $B_{m+1,j}$. 
	One has the formula \cite[(2.42)]{nen} 
	\begin{equation}
	\label{B}
	B_{m+1,j}(t) = \im D_{m,j}(t) \,  \partial_{(t,L)} \Pi_{m+1,j}(t)  + \Pi_{m+1,j}(t) \,  \partial_{(t,H-H_{m+1})} K_{m,j}(t) \ . 
	\end{equation}
	Consider first the term $ D_{m,j}(t) \partial_{(t,L)} \Pi_{m+1,j}(t)$. Then \eqref{D.norm0} and the inductive assumption give
\begin{equation}
\label{B.110}
\sup_{t \in \R}\norm{H^p \ \partial_{t}^\ell( D_{m,j}(t) \, \partial_{(t,L)} \Pi_{m+1,j}(t)) \ H^{-p}}\leq \frac{c_{m,n,\ell}}{\wt \Delta_{j-1}^{(m+2)\delta + 1}} \ , \qquad \forall 0 \leq p \leq n, \ \ell \geq 0 \ . 
\end{equation}
	To estimate $\partial_{(t,H-H_{m+1})} K_{m,j}$ use that $H - H_{m+1} = -\sum_{i=0}^m B_i$, so that 
	$$
	\partial_{(t,H-H_{m+1})} K_{m,j} = \partial_{t} K_{m,j}  - \im \sum_{0 \leq i \leq m} [B_i, K_{m,j}] \ . 
	$$
	Using once again \eqref{K.norm0} and the inductive assumption we get 
	\begin{equation}
	\label{B.120}
	\sup_{t \in \R}\norm{H^p \ \partial_{t}^\ell \partial_{(t,H-H_{m+1})}  K_{m,j}(t) \ H^{-p}}\leq \frac{c_{m,n,\ell}}{\wt \Delta_{j-1}^{(m+1)\delta + 1}}   \ , \qquad \forall 0 \leq p \leq n, \ \ell \geq 0 \ . 
	\end{equation}
	Then \eqref{B.110}, \eqref{B.120} and $ 0 < \delta \leq 1$  give
			\begin{align*}
\sup_{t \in \R}\norm{H^p \ \partial_{t}^\ell B_{m+1,j}(t) \ H^{-p}} \leq  \frac{\wt c_{m,n,\ell}}{\wt \Delta_{j-1}^{(m+1)\delta + 1}}  \leq \frac{C_{m+1,n,\ell}}{\wt \Delta_{j-1}^{(m+2)\delta }} \ , \qquad \forall 0 \leq p \leq n, \ \ell \geq 0 
	\end{align*}
thus proving 	the inductive step.

The estimate on $B_m(t)$ is trivial. The self-adjointness can be proved using the arguments of  \cite[Lemma 2]{nen}.

\section{Growth of norms for perturbations analytic in time}
In this section we prove the upper bound on the growth of the norm in case of perturbations which are analytic in time. The proof is essentially the same as in case of perturbations smooth in time, but we need extra attention to compute the dependence of all the constants from the parameters $\tJ$ and $\tM$. Indeed in this case we want to optimize $\tJ$ and $\tM$ by choosing them as a function of $t-s$, so we need to know exactly how all the constants depend on such parameters.\\
Notice that  perturbations   analytic in time  were  considered  in \cite{joy0, nene}.\\

First rewrite assumption {\rm (Va$)_n$} in the following way: there exist $\ta, c, A >0$ such that for any integer $\ell \geq 0$, $0 \leq p \leq n$
\begin{equation}
\label{V.est02}
\sup_{t \in \R} \norm{H^p \ \partial_t^\ell V(t) \ H^{-p-\nu}}   \leq \ta \ \frac{c^\ell \, \ell!}{A \ (1+\ell)^2} \ . 
\end{equation}
Here $A$ is a  constant such that
\begin{equation}
\label{sum.in}
(1+ \ell)^2 \sum_{n_1 + \cdots +n_k = \ell} \frac{1}{(1+n_1)^2}\cdots \frac{1}{(1+n_k)^2} \leq A^{k-1}  \ ,
\end{equation}
and can be chosen to be $A\geq 2 \pi^2/3$.\\ 
Note that \eqref{V.est02} can always be achieved simply by choosing  $A \, c_{0,n} \leq \ta$, $4\, c_{1,n} \leq c$.\\

The next step is to extend  Lemma \ref{estadB} in the analytic setting. 
Define
$$
\td :=\frac{2^{\mu\delta}}{2^{\mu\delta}-1}  \  
$$
and
$$
\tC_H := \max ( C_H, \ \wt C_H ) \ ,
$$
where $C_H$ and $\wt C_H$ are the constants of Lemma \ref{perturb} respectively Lemma \ref{lem:est:res0}. \\

Finally we fix a time $T \gg 1$. We obtain the following
\begin{proposition}
\label{ind.ass}
Fix a positive $\tM \in \N$ and choose $\tJ$ such that
\begin{equation}
 \label{cond}
2^{12} \, \left[ \tC_H \, (\ta + 2\td) c \right] \tM \leq 2^{\tJ\mu \delta} \ .
\end{equation}
For every integers $m, \ell, p$ such that 
$$
0 \leq \ell + m \leq \tM  \ , \quad 0 \leq p \leq n 
$$
the following holds true:
\begin{itemize}
\item[(i)]
 the operators $\Pi_{m,j}$ fulfill for every $j \geq 1$
 \begin{align}
  \label{Pm1}
 &\sup_{t \in [0,T]}\norm{H^p \ \partial_t^\ell  \Pi_{m,j}(t) \ H^{-p}} \leq \frac{ \ell! \ c^{\ell}}{(1+ \ell)^2} \left(\frac{1}{A \,2^{(j-1)\mu\delta}}
\right)^{\min(\ell, 1)}  \ , \\
\label{Pm2}
 &\sup_{t \in [0,T]} \norm{H^p \ \partial_t^\ell   \sum_{k=0}^{m}\partial_{(t,L)} \Pi_{k,j}(t) \ H^{-p}} \leq \frac{ \ell! \ c^{\ell}}{A (1+ \ell)^2} \, \frac{1}{ \,2^{(j-1)\mu\delta}} \, \frac{1}{4} \ , \\
 \label{Pm3}
&\sup_{t \in [0,T]}  \Vert H^p \ \partial_t^\ell \partial_{(t,L)} \Pi_{m,j}(t) \ H^{-p}\Vert \leq   \,  \frac{ \ell! \ c^{\ell}}{A(1+ \ell)^2} \frac{1}{\,2^{(j-1)\mu\delta}} \ . 
  \end{align}

\item[(ii)] the operators $B_{m,j}$ fulfill for every $j \geq 1$
 \begin{equation}
 \label{Bm1}
 \sup_{t \in [0,T]} \Vert H^p \ \partial_t^\ell B_{m,j}(t) \ H^{-p} \Vert \leq   \, \frac{\ell! \, c^\ell }{A (1+ \ell)^2} \frac{1}{2^{(j-1)(m+1)\mu\delta}} \frac{1}{(1+m)^2} \ . 
  \end{equation}
  \item[(iii)] the operators $B_{m}(t)$ fulfill 
   \begin{equation}
 \label{Bm11}
 \sup_{t \in [0,T]} \Vert H^p \ \partial_t^\ell B_{m}(t) \ H^{-p} \Vert \leq   \, \frac{\ell! \, c^\ell }{A (1+ \ell)^2} \frac{\td}{(1+m)^2} \ . 
  \end{equation}
  \item[(iv)] The Hamiltonians $H_m(t)$, $H_{ad,m}(t)$ fulfill $\wt{{\rm (Hgap)}} $ with $\wt\sigma_j$'s as in \eqref{new.clus}.
  \end{itemize}
\end{proposition}
Before proving  Proposition \ref{ind.ass}, we show how  Theorem \ref{thm:grest2} follows.

\vspace{2em}
\begin{proof}[Proof of Theorem \ref{thm:grest2}]
Having fixed $T \gg 1$, we consider the evolution $\cU(t,0)$ on a time interval  $0 \leq t \leq T$. 
Choose $\tJ$ in such a way that \eqref{cond} is fulfilled, namely 
		\begin{equation}
		\label{def.J}
\tJ := \frac{1}{	\mu\delta } \log (\tM+1) + \frac{1}{	\mu\delta }\log \left( 2^{12} \, \left[ \tC_H \, (\ta + 2\td) \, c \right]  \right) \ ,
	\end{equation}
	and choose $\tM$  as a function of $T$:
	\begin{equation}
	\label{def.M2} 
	\tM+1 = \lfloor \frac{1}{4} \log \la T \ra \rfloor   \ , 
	\end{equation}
	where $\lfloor \cdot \rfloor$ denotes the integer part. 
	
Now remark that the constants $c_1, 2_2$ of Lemma \ref{compk} {\em do not } depend on $\tM$ and $\tJ$. Indeed they depend only on $\sigma(H)$, $n$, and the norm of $H^{-n} ( V(t) + \sum_{i=0}^\tM B_i(t)) H^n$ . But by \eqref{Bm11} it follows easily that  such norm depends only on $\ta$, $\td$ (see  \eqref{est.Btot} below for the precise computation). 
Hence we can repeat the arguments of the proof of Theorem \ref{thm:grest} and using  estimate \eqref{Bm1} to estimate $B_{m,j}$, one gets 
\begin{equation}
\label{est.14}
\sup_{t \in [0,T]}\norm{\cU(t,0)\psi_0}_{2n}^2 \leq C \, 2^{\tJ (\mu+1)2n} 2^{\frac{2(\mu+1)n}{(\tM+1)\mu \delta} \log \la T \ra} \, \norm{\psi_0}_{2n}^2 \ ,
\end{equation}
where the constant $C$ {\em does not} depend on $\tJ$ and $\tM$.
Now substitute $\tJ$ as in \eqref{def.J} and $\tM$ as \eqref{def.M2}
	to get 
\begin{equation}
\label{est.15}
\sup_{t \in [0,T]}\norm{ \ \cU(t,0)\psi_0}_{2n}^2 \leq \gamma \, (\log \la T \ra)^{2n(\mu+1)/\mu\delta}\, \norm{\psi_0}_{2n}^2 \ ,
\end{equation}	
for some $\gamma >0$ 	which does not depend on $T$. 
Since $T$ was arbitrary, the estimate above holds $\forall t$. It is easy to adapt the proof to consider also the case $\cU(t,s)$.

Finally interpolating with $k=0$ gives the general case. The exponent in \eqref{ineqcinf} is obtained by simply replacing $\delta$ with its definition \eqref{delta}.
	
\end{proof}

   \subsection{Proof of Proposition \ref{ind.ass}}

 {\bf Step $m=0$.}  Define $\Pi_{0,j}(t)$ as in \eqref{pi0j}.
 We  apply  Lemma \ref{lem:an2} with $P = V$, $a=\ta$, $b=1$ (it is easy to see that   \eqref{cond2} is fulfilled) and get that for every $\ell \geq 0$, $ 0 \leq p \leq n$
 \begin{equation}
 \label{est50}
   \sup_{t \in [0,T]} \| H^p \ \partial_{t}^\ell \Pi_{0,j}(t) \ H^{-p} \| \leq 
    \frac{ \ell! \ c^{\ell}}{(1+ \ell)^2} \left(\frac{2^{4} \, \tC_H \, \ta}{A \, \wt \Delta_{j-1}^{\delta}}
\right)^{\min(\ell, 1)} \leq  \frac{ \ell! \ c^{\ell}}{(1+ \ell)^2} \left(\frac{1}{A \,2^{(j-1)\mu \delta}} 
\right)^{\min(\ell, 1)}.
   \end{equation}
   Consider now $ \partial_{(t,L)} \Pi_{0,j}(t)\equiv  \partial_t \Pi_{0,j}(t)$. For $\ell +1 \leq \tM$ one has by  \eqref{cond}, \eqref{est50} 
       \begin{equation}
       \label{est51} 
   \sup_{t \in [0,T]} \| H^p \ \partial_{t}^\ell \partial_{(t, L)} \Pi_{0,j}(t) \ H^{-p} \| \leq 
    \frac{ (\ell+1)! \ c^{\ell+1}}{(1+ \ell)^2} \, \frac{2^{4} \, \tC_H \, \ta}{A \, \wt \Delta_{j-1}^{\delta}} \leq  \frac{ \ell! \ c^{\ell}}{A(1+ \ell)^2} \frac{1}{\,2^{(j-1)\mu \delta}} \frac{1}{4}
\ .
   \end{equation}
   Thus we proved \eqref{Pm1}, \eqref{Pm2} and \eqref{Pm3} for $m=0$.
   Consider now 
  $B_{0,j}(t) =\Pi_{0,j}(t) \, \partial_{t} \Pi_{0,j}(t)$. 
  We apply Lemma \ref{lem:an1} with $P = \Pi_{0,j}$,
   $a=1,
    b=\frac{2^{4} \,\tC_H \, \ta}{A \, \wt \Delta_{j-1}^{\delta}}$, $k=0$ and  $Q=\partial_{t} \Pi_{0,j}$, $d=\frac{2^{4} \, \tC_H \, \ta}{A \, \wt \Delta_{j-1}^{\delta}}$, $f = 1$, $i=1$ and obtain that for $\ell +1 \leq \tM$
   \begin{align*}
   \sup_{t \in [0,T]}\norm{H^p \ \partial_{t}^{\ell}B_{0,j}(t) \ H^{-p}} 
   & \leq \frac{2^{4} \,\tC_H \, \ta}{A \, \wt \Delta_{j-1}^{\delta}}
   \left(\frac{2^{5} \,\tC_H \, \ta}{A \, \wt \Delta_{j-1}^{\delta}} +1 \right)^{\min(\ell,1)} \, 
   \frac{(\ell+1)! \, c^{\ell+1}}{A(1+\ell)^2}
    \leq 
   \frac{2^{5} \, \tC_H \,  \ta }{ \wt \Delta_{j-1}^{\delta}} \frac{(\ell+1)! \, c^{\ell+1}}{A(1+\ell)^2}   \\
  &  \leq \frac{\ell! \, c^{\ell}}{A(1+\ell)^2} \frac{1}{2^{\mu (j-1)\delta} }
   \end{align*}
   provided
   \begin{equation}
   \label{cond.it} 
  \max \left( \frac{2^{5} \,\tC_H \, \ta}{A } , \    2^{5} \,\tC_H \, \ta \,  \tM  \, c   \right)  \leq 2^{\tJ\mu\delta} \ ,
   \end{equation}
   which is clearly fulfilled using \eqref{cond}.
 This proves \eqref{Bm1} for $m=0$. 
\\

{\bf Step $m\rightsquigarrow m+1$.} Assume that we performed already $0<m<\tM$ steps. By the inductive assumption  $ \forall \, 0 \leq i \leq m$ one has that $B_i(t)= \sum_j B_{i,j}(t)$ fulfills \eqref{Bm11} $\forall \ell+i+1 \leq \tM$.\\
Thus $H_{m+1}(t) = H+ W(t)$, where we defined  $W(t):=V(t) + \sum_{i=0}^m B_i(t)$. It fulfills $\forall \ell + m +1 \leq \tM$, $\forall 0 \leq p \leq n$
\begin{align}
\label{est.Btot} 
\sup_{t \in [0,T]} \norm{H^p \ \partial_{t}^\ell W(t) \ H^{-p-\nu}}  & \leq  \frac{\ell! \, c^{\ell}}{A(1+\ell)^2}\,\left(\ta + \td \sum_{i=0}^m\frac{1}{(1+i)^2} \right)  \leq \frac{\ell! \, c^{\ell}}{A(1+\ell)^2}\,  \left( \ta +  2\td\right) \ .
\end{align}
Thus $W(t)$ is a perturbation of $H$  analytic in time  which fulfills the conditions of Lemma \ref{lem:an2}. Indeed with $a = \ta +  2\td$ we have that 
$2^{4} \, \tC_H \,  a  \leq 2^{\tJ \mu\delta}$, 
hence by Lemma \ref{lem:an2} the projectors 
$$
\Pi_{m+1,j}(t) = - \frac{1}{2\pi \im } \oint_{\Gamma_j} R_{m+1}(t, \lambda) \ d\lambda \ , \qquad  R_{m+1}(t, \lambda):= (H_{m+1}(t) - \lambda)^{-1}
$$
are well defined $\forall j \geq 1$. Furthermore they fulfill $\forall \ell + 1+ m \leq \tM$, $\forall 0 \leq p \leq n$
\begin{align}
\label{pr.norm}
 \sup_{ t \in \R} \| H^p \  \partial_{t}^\ell \Pi_{m+1,j}(t)  H^{-p} \| \leq 
 \frac{ \ell! \ c^{\ell}}{ (1+ \ell)^2} \  \left(\frac{2^{4} \, \tC_H}{A \wt \Delta_{j-1}^{\delta}} \left( \ta +  2\td \right)\right)^{\min(\ell, 1)} \ ,
 \end{align}
 where we used  \eqref{est.Btot} and Lemma \ref{lem:an2}.
 
%
To estimate $\partial_t^\ell \, \partial_{(t,L)} \Pi_{m+1,j}(t)$ we use again formula \eqref{pr.norma0}.
Consider its last term. Since $[B_l(t), \Pi_{l,j}(t)] = - \im \partial_{(t,L)} \Pi_{l,j}(t)
$, one gets the identity 
\begin{equation}
\label{pr.norma2}
\sum_{l=0}^{m+1} \partial_{(t,L)} \Pi_{l,j}(t) = \partial_t \Pi_{m+1,j}(t) - \im \sum_{l=0}^m [B_l(t) , \, \Pi_{m+1,j}(t) - \Pi_{l,j}(t)]  \ . 
\end{equation}
This identity allows us to estimate \eqref{Pm2} at step $m+1$.
We estimate  the two terms in the r.h.s above separately. 
To estimate the second one  we use formula \eqref{pi-pi}. Since
$$
\sup_{t \in [0,T]} \norm{H^p \ \partial_{t}^\ell \left(H_{m+1}(t)- H_{l}(t) \right) \ H^{-p}}    \leq \frac{\ell! \, c^{\ell}}{A(1+\ell)^2}\, \left( \ta +  2\td\right) \ , 
$$
by  Lemma \ref{lem:an3} we get  that $\forall \ell + m + 1 \leq \tM$, $\forall 0 \leq p \leq n$, 
\begin{equation}
\label{pr.norma3}
\sup_{t \in [0,T]} \norm{H^p \ \partial_{t}^\ell \left(\Pi_{m+1,j}(t) - \Pi_{l,j}(t)\right) \ H^{-p}} 
\leq \frac{\ell! \, c^\ell}{A (1+ \ell)^2} \, 
   \frac{ \left( \ta +  2\td\right)  \, 2^5}{ \, \wt \Delta_{j-1}}  \ . 
\end{equation}
Hence by Lemma \ref{lem:an1}, \eqref{Bm11},  \eqref{pr.norma3} we get
\begin{equation}
\label{pr.norma31}
\sup_{t \in [0,T]} \norm{H^p \ \sum_{l=0}^m [B_l(t) , \, \Pi_{m+1,j}(t) - \Pi_{l,j}(t)] \ H^{-p}} 
\leq \frac{\ell! \, c^\ell}{A (1+ \ell)^2} \, 
   \frac{ \left( \ta + 2 \td\right)   \, \td\, 2^8}{ \, \wt \Delta_{j-1}}  \ . 
\end{equation}

The first term of \eqref{pr.norma2} is estimated by \eqref{pr.norm} with $\ell+1$ replacing $\ell$. Together with \eqref{pr.norma31} we get for $\ell +1 + m \leq \tM$, $\forall 0 \leq p \leq n$
\begin{equation}
\label{pr.norma4}
\sup_{t \in [0,T]} \norm{H^p \ \partial_{t}^\ell \sum_{l=0}^{m+1} \partial_{(t,L)} \Pi_{l,j}(t) \ H^{-p}} 
\leq \frac{ \ell! \ c^{\ell}}{A (1+ \ell)^2} \, \frac{1}{ \,2^{(j-1)\mu\delta}} \, \frac{1}{4}
\end{equation}
using \eqref{cond}. This proves \eqref{Pm2} at step $m+1$.\\
Now consider $\partial_{(t,L)} \Pi_{m+1,j}(t)$. Using \eqref{pr.norma4} and the inductive assumption \eqref{Pm2} we get
\begin{align}
\notag
\norm{H^p \ \partial_t^\ell \, \partial_{(t,L)} \Pi_{m+1,j}(t) \ H^{-p} } &\leq 
\norm{H^p \ \partial_t^\ell \,  \sum_{k=0}^{m+1} \partial_{(t,L)} \Pi_{k,j}(t) \ H^{-p}} 
+
 \norm{H^p \ \partial_t^\ell \, \sum_{k=0}^{m} \partial_{(t,L)}\Pi_{k,j}(t) \ H^{-p} } \\
 \label{pr.norma5}
 & \leq \frac{ \ell! \ c^{\ell}}{A (1+ \ell)^2} \, \frac{1}{ \,2^{(j-1)\mu\delta}}
\end{align}
proving \eqref{Pm2} at step $m+1$.

Next we estimate $L_{m,j}$, $K_{m,j}$, $D_{m,j}$ defined in \eqref{L.def}--\eqref{D.def}.
\begin{lemma}
For every $0 \leq p \leq n$,  $\ell + m+1 \leq \tM$	one has 
	\begin{align}
	\label{L.norm}
 \sup_{t \in [0,T]} 	\norm{H^p \ \partial_{t}^\ell  L_{m,j}(t) \ H^{-p}} \leq \frac{\ell! \, c^\ell}{A(1+\ell)^2} \frac{2^5 \, \td}{\wt \Delta_{j-1}} \\
	\label{K.norm} 
 \sup_{ t \in \R} 	\norm{H^p \  \partial_{t}^\ell  K_{m,j}(t) \ H^{-p}} \leq \frac{\ell! \, c^\ell}{A(1+\ell)^2} \frac{2^5}{\wt \Delta_{j-1}} \, \frac{1}{2^{\mu(j-1)(m+1)\delta}} \, \frac{1}{(1+m)^2}\\
	\label{D.norm}
 \sup_{ t \in \R} 	\norm{H^p \ \partial_{t}^\ell  D_{m,j}(t) \ H^{-p}} \leq \frac{\ell! \, c^\ell}{A(1+\ell)^2} \, \frac{2^8 }{2^{\mu(j-1)(m+1)\delta} \, \wt \Delta_{j-1}} \, \frac{1}{(1+m)^2}
	\end{align}
\end{lemma}
\begin{proof}
	First we prove \eqref{L.norm}. Using the definition of $L_{m,j}$ given by  \eqref{L}, we apply  Lemma \ref{lem:an3} with $P = V + B_0 + \cdots + B_m$, $Q=V + B_0 + \cdots + B_{m-1}$,  $B = B_m$, $h= \td/(1+m)^2$ and get that for $\ell +1+m \leq \tM$, $\forall 0 \leq p \leq n$
$$
\sup_{t \in [0,T]} \norm{H^p \ \partial_{t}^\ell L_{m,j}(t) \ H^{-p}}  \leq \frac{\ell! \, c^\ell}{A(1+\ell)^2}\, \frac{2^5 \, \td}{\wt \Delta_{j-1}} \frac{1}{(1+m)^2}  , 
$$
provided $2^4\tC_H (\ta + 2\td) \leq 2^{\tJ \mu \delta}$.
	For later use consider the operator $(1-L_{m,j}(t))^{-1}$.  Provided 
	$
	\frac{2^5}{A\, \wt\Delta_{j-1}} \, \td \leq \frac{1}{2} \ , 
	$
	the operator
	$(1-L_{m,j}(t))$ is invertible by Neumann series  and
	$$
	\sup_{t \in [0,T]} \norm{(1-L_{m,j}(t))^{-1}} \leq 2 \ . 
	$$
	To study its derivatives we can proceed as in \eqref{ris.der0}, \eqref{ris.est00} to get for $\ell + m+1 \leq \tM$, $\ell \geq 1$
	\begin{align}
	\label{der.inv.L}
	\sup_{t \in [0,T]}\norm{H^p \ \partial_{t}^\ell (1-L_{m,j}(t))^{-1} \ H^{-p}} \leq \frac{\ell! \, c^\ell}{A(1+\ell)^2} 2\, \sum_{k=1}^{\ell}  \left(\frac{2^6 \, \td}{\wt \Delta_{j-1}}  \right)^{k}  \leq \frac{\ell! \, c^\ell}{A(1+\ell)^2} \,  \frac{\td \, 2^8}{\wt \Delta_{j-1}}
	\end{align}
	provided 
	$2^7 \, \td \leq \wt \Delta_{j-1}$.
	Thus for $\ell+m+1 \leq \tM$
	\begin{align}
	\label{der.inv.L1}
	\sup_{t \in [0,T]}\norm{H^p \ \partial_{t}^\ell (1-L_{m,j}(t))^{-1} \ H^{-p}} \leq  \frac{\ell! \, c^\ell}{(1+\ell)^2} \, 2 \,\left(\frac{ 2^7 \, \, \td}{A \, \wt \Delta_{j-1}}\right)^{\min(\ell,1)} \ . 
	\end{align}
	Estimate  \eqref{K.norm} follows immediately from the identity \eqref{K}
and   Lemma \ref{lem:an3} with $B=B_{m,j}$, $h=2^{-\mu(j-1)(m+1)\delta} \  (1+m)^{-2}$.\\
	Finally we prove \eqref{D.norm}. Consider formula \eqref{D}. 
	Then Lemma \eqref{lem:an1} applied twice  and  \eqref{pr.norm}, \eqref{K.norm}, \eqref{der.inv.L1} give $\forall \ell + 1 + m \leq \tM$
	\begin{equation}
	\sup_{t \in [0,T]}\norm{H^p \ \partial_{t}^\ell D_{m,j}(t)  \ H^{-p}}\leq  \frac{\ell! \, c^\ell}{A(1+\ell)^2} \, \frac{2^8 }{2^{\mu(j-1)(m+1)\delta} \, \wt \Delta_{j-1}} \ \frac{1}{(1+m)^2}  \ , \qquad \forall 0 \leq p \leq n  \ , 
	\end{equation}
	where we used that \eqref{cond} implies 
	$	\max \left[\frac{2^8 \, \td}{A \, \wt \Delta_{j-1}}  \ ,\frac{2^{4} \, \tC_H \, }{A \wt \Delta_{j-1}^{\delta}} \left( \ta +  \td\right) \right] \leq 1 \ . 
	$
\end{proof}
\vspace{1em}
We can now conclude the proof by calculating the norm of $B_{m+1,j}$. 
	We use again \eqref{B}. 
	Consider first the term $ D_{m,j} \partial_{(t,L)} \Pi_{m+1,j}$. We can compute its first $\ell$ derivatives provided 
	$	\ell + 1 + m \leq \tM \ . $
We apply once again  Lemma \ref{lem:an1} with $P=  D_{m,j}$,  $k=0$ and  $Q= \partial_{(t,L)} \Pi_{m+1,j}$, $i=0$, and use estimates \eqref{pr.norma5}, \eqref{D.norm} to get 
\begin{equation}
\label{B.11}
\sup_{t \in [0,T]}\norm{H^p \ \partial_{t}^\ell( D_{m,j}(t) \, \partial_{(t,L)} \Pi_{m+1,j}(t)) \ H^{-p} }
\leq
 \frac{\ell! \, c^{\ell}}{A(1+\ell)^2} \, \frac{2^{10} }{2^{(j-1)\mu(m+2)\delta} \, \wt \Delta_{j-1} \, (1+m)^2 } \ . 
\end{equation}
The other term to estimate is $\Pi_{m+1,j} \partial_{(t,H-H_{m+1})} K_{m,j}$. 
	To estimate $\partial_{(t,H-H_{m+1})} K_{m,j}$ write
	$
	\partial_{(t,H-H_{m+1})} K_{m,j} = \partial_{t} K_{m,j}  - \sum_{0 \leq i \leq m} [B_i, K_{m,j}] \ . 
	$
	By \eqref{K.norm} 
	\begin{equation}
	\label{B.12}
	\sup_{t \in [0,T]}\norm{H^p \ \partial_{t}^\ell \partial_{t} K_{m,j}(t) H^{-p}}\leq \frac{(1+\ell)! \, c^{\ell+1}}{A(1+\ell)^2} \frac{1}{2^{\mu(j-1)(m+1)\delta}} \frac{2^5}{\wt \Delta_{j-1}} \frac{1}{(1+m)^2}  \ . 
	\end{equation}
	Consider now  $ \sum_{i=0}^m[B_i, K_{m,j}]$. For $\ell  +1+ m \leq \tM$ we have that    Lemma \ref{lem:an1}, estimates \eqref{Bm11} and  \eqref{K.norm} imply  that
	 \begin{equation}
	 \label{B.13}
	 \sup_{t \in [0,T]}\norm{H^p \ \partial_{t}^\ell \sum_{i=0}^m[B_i(t), K_{m,j}(t)] \ H^{-p}} \leq  \frac{\ell! \, c^\ell}{A(1+\ell)^2} \frac{1}{2^{\mu(j-1)(m+1)\delta}} \frac{2^7 \, \td}{\wt \Delta_{j-1}} \, \frac{1}{(1+m)^2}  \ . 
	 \end{equation}    
Then \eqref{B.12} and \eqref{B.13} imply that, for $\ell+1+m\leq \tM$,  
\begin{equation}
	 \label{B.14}
	 \sup_{t \in [0,T]}\norm{H^p \ \partial_{t}^\ell \Big( \Pi_{m+1,j}(t) \partial_{(t,H-H_{m+1})} K_{m,j}(t)\Big) \ H^{-p}} \leq  \frac{\ell! \, c^\ell}{A(1+\ell)^2} \frac{1}{2^{\mu(j-1)(m+1)\delta}} \frac{2^5 }{\wt \Delta_{j-1}} \, \frac{1}{(1+m)^2} ( \tM c + 4\td) \ . 
	 \end{equation}    
	Then \eqref{B.11} and  \eqref{B.14}  give
			\begin{align*}
\sup_{t \in [0,T]}\norm{H^p \partial_{t}^\ell B_{m+1,j}(t) \ H^{-p}} & \leq\frac{\ell! \, c^{\ell}}{A(1+\ell)^2} \frac{1}{2^{\mu(j-1)(m+1)\delta}} \frac{2^5}{\wt\Delta_{j-1}} \frac{1}{(1+m)^2}
\left[
 \frac{2^{5}}{2^{(j-1)\mu\delta}}+\tM c + 4\td \right] \\
& \leq \frac{\ell! \, c^{\ell}}{A(1+\ell)^2} \frac{1}{2^{\mu(j-1)(m+2)\delta}} \frac{1}{(2+m)^2}
	\end{align*}
	where we used that
	\begin{equation}
	\label{con6}
	 2^5 \frac{(m+2)^2}{(m+1)^2} [ 2^5 + \tM c + 4 \td] \leq 2^{12} \, c \,  \td \, \tM  \leq {2^{\tJ\mu\delta}}  \ .
	\end{equation}
	The inductive step is proved.

\section{Applications}
\label{app}
In this section we apply our abstract theorems to different models. We are able to recover many already known results  and to prove new estimates.

\subsection{One  degree   of freedom  Schr\"odinger  operators}
Let us consider  here equation \eqref{eq:sc0} where $L(t)$ is a time  dependent perturbation of the anaharmonic oscillator, namely
\begin{equation}
\label{apl:eq1}
L(t) = -\frac{d^2}{dx^2} + x^{2k} + p(x) + V(t, x) = H_k + V(t,x) , \qquad x \in \R 
\end{equation}
where  $k\in\N$,   $p(x)$  is a polynomial of degree less   than $2k-1$, and $V(t, x)$  is a  real valued time dependent perturbation with a  polynomial  growth  in $x$
 of degree $\leq m$  fulfilling 
\beq
\label{HVS}
 \sup_{t \in \R} \vert\partial_t^\ell\partial^j_x V(t,x)\vert \leq  C_\ell\la x\ra^{(m-j)_+} \ , \qquad  \forall  x\in\R \ .
\eeq
 Without restriction we can always assume that $H_k$ is positive and invertible.\\ 
 The following lemma  is an easy computation
 \begin{lemma}\label{nohineq}
 For  every $\mu >0$  there  exists $C_\mu>0$  such that   for   every
  $(j,k)\in\N\times\N$  such that  $\frac{j}{2k}+\frac{\ell}{2} \leq \mu$   we have
  \beq
  \Vert x^j\frac{d^\ell}{dx^\ell}u\Vert_{L^2(\R^d)} \leq C_\mu \Vert H_k^\mu u\Vert_{L^2(\R^d)} \ .
  \eeq
 \end{lemma}
Under the condition that  $m\leq k+1$  we   get  that the  commutator  $[V(t,x), H_k]$  is 
 $H_k$-bounded. 
 By  Theorem \ref{thm:flow2} $L(t)$  generates  a propagator
 $\cU(t,s)$  in the  Hilbert  spaces scale  $\cH_{k}^r:=D((H_k)^{r/2})$. Furthermore  if
   $m< k+1$ then $[V(t,x), H_k]$  is 
 $H_k^{1-\theta}$-bounded with $\theta =  \frac{k-m+1}{2k}$. Thus   Theorem \ref{thm:norm}  can be applied   
     and we get the following polynomial bound for the growth
    of the $\cH^r_k$-norm $\forall r>0$, 
  \beq
  \label{grk}
    \Vert\cU(t,s)\psi_s\Vert_{r} \leq C 
    \la t-s\ra^{\frac{kr}{k-m+1}}\Vert\psi_s\Vert_{r} \ .
  \eeq
 For $k=1$  and $m=0$ we recover a known  bound  for time dependent perturbations of  the harmonic oscillator:
 \beq
 \label{harpert}
    \Vert\cU(t,s)\psi_s\Vert_{r} \leq C
    \la t-s\ra^{\frac{r}{2}}\Vert\psi_s\Vert_{r} \ .
  \eeq
As mentioned in Remark \ref{rem:del}, Delort \cite{del1}  suggests that estimate (\ref{harpert})  may   be  sharp  for $V(t,x)$ satisfying \eqref{HVS} with $m=0$.
 Actually the example  constructed by him  is a zero order pseudo-differential operator. Construct a  local  potential to saturate the estimate \eqref{harpert}
  is still an open problem.\\
  
When  $k>1$ (namely the anaharmonic case)  we can improve the bound \eqref{grk} by applying Theorem \ref{thm:grest} and Theorem \ref{thm:grest2}. Indeed it is well known (see e.g. \cite{rob82}) that in this case  $H_k$ satisfies (Hgap). Indeed the resolvent of $H_k$ is a compact operator in $L^2(\R)$, hence its spectrum is discrete, $\sigma(H_k) =  \{\lambda_j\}_{j\geq 1}$ and furthermore it is known to be simple. To verify the gap condition we use the following lemma:
  \begin{lemma}
  \label{muk}
  There  exists $c_k>0$ such that 
  $$
  \lambda_{j+1}-\lambda_j \geq c_k \,  j ^{\mu_k} \ , \ \ \   \forall  j\geq 1 \ , 
  $$
  where $\mu_k= \frac{k-1}{k+1}$.
  \end{lemma}
  \begin{proof}
  It is known   that  the eigenvalues  $\{\lambda_j\}_{j \geq 1}$ of $H_k$ are given at  all order in $j$ by a Bohr-Sommerfeld rule  \cite{rob82}: one has that 
  $$
  \lambda_j^{\frac{k+1}{2k}} = b_k\left(j+\frac{1}{2}\right) + O(1)
  $$
 where $b_k$  is a smooth function   such that $b_k(x) =  c_0 x  + o( x)$.
Lemma \ref{muk} follows easily.
 \end{proof}
 
\noindent Lemma \ref{muk} shows that $H_k$ satisfies (Hgap) defining $\forall j \geq 1$  the clusters $\sigma_j := \{ \lambda_j \}$ and $\mu_k= \frac{k-1}{k+1}$.  \\
Consider now the perturbation $V(t,x)$. 
The critical index  to apply Theorem \ref{thm:grest} is here $ \frac{\mu_k}{\mu_k+1} = \frac{k-1}{2k}$. One verifies easily that  $V(t,x)$ is $H_k^\frac{m}{2k}$-bounded. Hence provided $m < k-1$, we have that $\nu := \frac{m}{2k}$ fulfills $\nu < \frac{\mu_k}{\mu_k+1}$   (such  condition appears already in a work by Howland \cite{how}  in order to study the Floquet spectrum when $V(t,x)$ is a periodic in time perturbation).
\begin{theorem}[smooth case]
\label{thm:app1} Fix an integer $k >1$ and let  $m < k-1$.
Assume that $V$  satisfies  the estimate \eqref{HVS}. Then for every $r >0$, for every $\varepsilon >0$, there exists a positive $C_{r,\epsilon}$ s.t. 
$$
\Vert\cU(t,s)\Vert_{\cL(\cH^r)} \leq C_{r,\varepsilon}\la t-s\ra ^\varepsilon \ . 
$$
\end{theorem}
\begin{proof}
Having fixed $r>0$, choose an integer $n$ s.t. $r \leq 2n$. 
To apply Theorem \ref{thm:grest} we have to check that $V$ fulfills assumption (Vs$)_n$. Remark that $H_k$ is a pseudodifferential operator whose symbol is in the class $\wt S^{2k}_{1,k}$ of Definition \ref{def:symbol}, while $V(t)$ belongs to $\wt S^{m}_{1,k}$.  But under assumption \eqref{HVS}   
 $H_k^p\partial_t^\ell V(t)H_k^{-p-\nu}$  is a pseudo-differential operator  of order 0 (see the symbolic calculus of Theorem \ref{symb.cal}). So applying the Calderon-Vaillancourt theorem (Theorem \ref{thm:cv}) we get that (Vs$)_n$ is satisfied (see Appendix \ref{app:pseudo} for some well known properties of pseudodifferential operators).
\end{proof}
In case $V(t,x)$ is analytic in time, we obtain better estimates:
\begin{theorem}[analytic case]
\label{thm:app2}
Fix an integer $k >1$ and let  $m < k-1$.
Assume that there  exist $C_0, C_1>1$ such that $\forall \ell, j \geq 0$
we have 
\beq
\label{HVS1}
\sup_{t \in \R} \Vert\la x\ra^{-(m-j)_+}\partial_t^\ell\partial^j_x V(t,x)
\Vert_{L^\infty_x(\R)} \leq C_{1} \,  C_0^\ell \, \ell!  \ . 
\eeq
Then  we have that $\forall r >0$
$$
\Vert\cU(t,s)\Vert_{\cL(\cH^r)} \leq C_r
\left(\log \la t-s\ra \right)^{\frac{rk}{k-1-m}} \ . 
$$
\end{theorem}
\begin{proof}
We apply Theorem \ref{thm:grest2}. Having fixed $r >0$, we choose an arbitraty integer $n$ with $r \leq 2n$.  We   check assumption (Va$)_n$   using again the Calderon-Vaillancourt Theorem.
\end{proof}

\vspace{1em}
\noindent{\em Comparison with previous results:} To the best of our knowledge Theorem \ref{thm:app1} and Theorem \ref{thm:app2} are new. \\
In same cases better  estimates on the $\cH^r_k$-norm of the flow are known. For example if $V(t,x)$ is a quasi-periodic function of time and small in size, one might try to prove reducibility, which in turn implies that  the Sobolev norms are uniformly bounded  in time. We mention just the latest results: Bambusi \cite{bam16,bam162} proved reducibility for $L(t)$ on $\R$ in several cases, including   $k >1$  and  $V(t,x)$ fulfilling \eqref{HVS} with $m < k+1$ (in some cases even for $m\leq 2k$).  Gr\'ebert and Paturel \cite{grebert16} proved reducibility for $L(t)$ on $\R^d$, $d \geq 1$, with $k=1$ and  $V(t,x)$ a small bounded quasi-periodic perturbation.

\subsection{Operators on compact manifolds}
 Let $(M,g)$ be a Riemaniann  compact  manifold   with  metric $g$ and let $\triangle_g$ be  the Laplace-Beltrami operator. Denote by $S_{cl}^m(M)$  the space  of classical symbols of order $m \in\R$ on the cotangent $T^*(M)$ of $M$ (see H\"ormander \cite{ho} for more details). \\
  Let $H=1-\triangle_g$ and  $V(t)\equiv V(t,x,D_x)$   be  an Hermitian    classical pseudodifferential operator of order $m\leq 1$. 
We want to consider the Schr\"odinger equation \eqref{eq:sc0} with $L(t)$ defined by
$$
L(t) = -\triangle_g + 1 + V(t) = H + V(t) \ ,
$$  
and  study its flow  in the usual  scale of Sobolev spaces $H^k(M)\equiv D(H^{k/2})$.
  
  By semiclassical calculus one verifies that $[L(t),H] H^{-1}$ is a pseudodifferential operator of order $0$, hence the assumptions    of Theorem \ref{thm:flow2}
   are satisfied and   $L(t)$  has  a well defined propagator  $\cU(t,s)$  in $H^k(M)$   and it  is unitary in $L^2(M)$. \\
   
   Moreover   one has that $[L(t), H] H^{-\tau}$, $\tau = \frac{m+1}{2}$ is a pseudodifferential operator of order 0. Provided $m <1$, one has $\tau <1$, hence by applying  Theorem \ref{thm:norm} we get for the flow $\cU(t,s)$  the following uniform  estimate  in the  space $H^k(M)$:
   \beq
   \label{est:apl1}
   \Vert\cU(t,s)\Vert_{\cL(H^k(M))} \leq C_k\la t-s\ra^{\frac{k}{1-m}}.
   \eeq

Better estimates can be obtained if  the  spectrum of $\triangle_g$  satisfies a gap condition.  A typical   example is the Laplace-Beltrami operator  on  Zoll manifolds. 
We recall that Zoll manifolds are manifolds where all geodesics are closed and have the same period, for examples  spheres in any dimension.
It is a classical result due to Colin de Verdi\`ere \cite{cdv} that the  spectrum  of $\sqrt \triangle_g$ is concentrated in $\bigcup_{j\geq 1}[j+\sigma -\frac{C}{j}, j+\sigma +\frac{C}{j}]$, 
where $\sigma\in\Z/4$  and $C>0$.   Defining  $\forall j \geq 1$ the cluster $\sigma_j :=[(j+\sigma -\frac{C}{j})^2, (j+\sigma +\frac{C}{j})^2]$, one sees immediately that  the gap condition  is satisfied with $\mu=1$.
Hence $H$ fulfills (Hgap). The critical regularity for $V$ is then $\frac{\mu}{\mu+1}=\frac{1}{2}$.

\begin{theorem}
\label{thm:apl3}
Assume that   $\forall t\in\R$, $V(t) $ is  an Hermitian pseudodifferential operator   on $M$ of order $m< 1$. Assume that   in local charts its  symbol
 $v(t,x,\xi)$  fulfills the following condition: there exists $C_1 >0$ s.t. $\forall \ell \geq 0$,  for every multi-indices $\alpha, \beta$ 
there   exists $C_{\alpha\beta}>0$ such that 
\begin{equation}
\label{symb.cond} 
\Vert\la\xi\ra^{-m+\vert\beta\vert} \partial_x^\alpha\partial_\xi^\beta\partial_t^\ell v(t, x,\xi)\Vert_{L^\infty(\R_t \times M\times\R^d)}\leq C_{\alpha\beta} \, C_1^\ell \, \ell!.
\end{equation}
Then for any $r >0$ the propagator $\cU(t,s)$ for $H+V(t)$ satisfies
\begin{equation}
\label{apl3.est}
\Vert\cU(t,s)\Vert_{\cL(\cH^r)} \leq C_n\left(\log\la t-s \ra \right))^{\frac{r}{1-m}}
\end{equation}
\end{theorem}
\begin{proof}
Having fixed $r >0$, choose an integer $n$ with $r \leq 2n$. We verify that (Va$)_n$ holds.
By semiclassical calculus, $V(t) H^{-\frac{m}{2}} \in S_{cl}^0(M)$. For $m<1$, $\nu:=  \frac{m}{2} $ is strictly smaller than  $ \frac{1}{2}$ the critical regularity. To verify that $V(t)$ satisfies (Va$)_n$ it suffices to work in local charts (since $M$ is compact one can considered just a finite number of them). Then by Calderon-Vaillancourt theorem, the norm of $\partial_t^\ell V$ as an operator  $H^{n+2\nu}(M) \to H^n(M)$ is controlled by 
$$
C \, \sum_{|\alpha| + |\beta| \leq N} \Vert\la\xi\ra^{-m+\vert\beta\vert} \partial_x^\alpha\partial_\xi^\beta\partial_t^\ell v(t, x,\xi)\Vert_{L^\infty(\R_t \times M\times\R^d)}
$$
for some universal constants  $C, N$ sufficiently large and  depending only on $n$ and the dimension of $M$. Then using \eqref{symb.cond} one verifies that (Va$)_n$ is fulfilled.
\end{proof}

\vspace{1em}
\noindent{\em Comparison with previous results:} Theorem \ref{thm:apl3} for Zoll manifolds and with unbounded perturbations  is a new result.\\
In case  $M = \T$,  Theorem \ref{thm:apl3}  was proved by Bourgain \cite{Bourgain1999} when $V(t,x)$ is an analytic periodic function in both $x$ and $t$ and extended by  Wang \cite{wang08} for $V(t,x)$  real analytic function with arbitrary dependence on $t$. Such authors obtained the  bound $\Vert\cU(t,s)\Vert_{\cL(\cH^r)} \leq C_r\left(\log\la t-s \ra \right))^{\varsigma r}$, for some constant $\varsigma >3$. Remark that our Theorem \ref{thm:apl3} improves this estimate: indeed for  bounded potentials one can take $m=0$ in \eqref{apl3.est}, leading to the optimal  estimate (see Remark \ref{rem:bou}).   \\
Later Fang and Zhang \cite{zhang}  extended the results of \cite{wang08} to the $d$-dimensional torus $\T^d$, $d > 1$ (such result is not covered by Theorem \ref{thm:apl3} since $-\triangle$ on $\T^d$ does not fulfill (Hgap)). \\
In case $V(t,x)$ is a smooth function of $x$ and $t$, the estimate $\Vert\cU(t,s)\Vert_{\cL(\cH^r)} \leq C_r \la t-s \ra^\epsilon $ was proved by Bourgain  \cite{bo} for $M = \T^d$, $d \geq 1$, and by Delort when $M$ is a Zoll manifold.

If $V(t)$  is quasi-periodic in time and small in size, some results of reducibility are known.  We cite here only the latest achievements in this direction (see their bibliography for more references).
In case $M = \T$,  Feola and Procesi \cite{feola15} proved reducibility when $V(t,x)$ quasi-periodic in time, small in size, and in some class of unbounded operator. 
In case $M = \T^d$,  $d > 1$, Eliasson and Kuksin \cite{eliasson09} proved reducibility when $V(t,x)$ is a small analytic potential.
For $M = \S^2$ (2-dimensional sphere) reducibility was proved by Corsi Haus and Procesi \cite{corsi15}.

\subsection{Time dependent electro-magnetic fields}
Consider the Schr\"odinger equation \eqref{eq:sc0} with $L(t) = H_{a,V}(t)$ the time dependent electro-magnetic field 
$$
H_{a,V}(t):= \frac{1}{2}(D+ a(t,x))^2 + V(t,x) \ , \qquad x \in \R^d \ ,
$$
where we denoted   $D:=\im^{-1}\grad$. Here we assume that the electromagnetic potential $(a(t,x), V(t,x))$  is continuous in $t\in\R$ and smooth  in
 $x\in\R^d$. Furthermore we assume that for every multi-index $\alpha$  we have the  following uniform  estimate in $(t,x)$:
\beq\label{subquad}
\abs {\partial_x^\alpha a(t,x)} \leq C_\alpha \la x \ra^{(1- |\alpha|)_+} \ , \qquad \abs {\partial_x^\alpha V(t,x)} \leq C_\alpha \la x \ra^{(2- |\alpha|)_+} \ , \qquad \forall t \in \R \ .  
\eeq
We choose $H = H_{osc}$  where  
$H_{osc} = \frac{1}{2}\left(-\triangle + \vert x\vert^2\right)$ is the harmonic oscillator and define $\forall r \geq 0$  the spaces $\cH^r = D(H_{osc}^{r/2})$.  By  direct computations we can prove that the assumptions  of Theorem \ref{thm:flow2} are  satisfied. Indeed write first 
$$
H_{a,V}= -\triangle + V(t,x) + 2a(t,x)\cdot D + i^{-1}{\rm div}(a(t,x)) + a^2(t,x).
$$
Denote 
$\partial_j = \frac{\partial}{\partial x_j}$.    Then we get that 
$$
K: = [H_{a,V}, H_{osc}] = \sum_{1\leq j, k\leq d}\gamma_{j,k}\partial^2_{j,k} + \sum_{1\leq j\leq d}\gamma_j\partial_j + \gamma_0 \ 
$$
where  for any multi-index $\alpha$, there exists a $C_\alpha >0$ s.t. for any $1 \leq j,k \leq d$ 
$$
\vert D^\alpha\gamma_{j,k}(t,x)\vert \leq C_\alpha,\qquad 
\vert D^\alpha\gamma_{j}(t,x)\vert \leq C_\alpha\la x\ra^{(1-\vert\alpha\vert)_+}, \qquad  \vert D^\alpha\gamma_0(t,x)\vert \leq C_\alpha\la x\ra^{(2-\vert\alpha\vert)_+}.
$$
The following Lemma is well known  and can be  easily proved by induction:
\begin{lemma}
For every  multi-index $\alpha, \beta$  we  have
\beq\label{intineq}
\Vert x^\alpha D^\alpha u\Vert_{L^2(\R^d)}  \leq C_{\alpha, \beta}\Vert H_{osc}^{\frac{\vert\alpha+\vert\beta\vert}{2}}u\Vert_{L^2(\R^d)} ,\; \forall u\in L^2(\R^d).
\eeq
\end{lemma}
From this Lemma it results  that $K$  is $H_{osc}$-bounded.  Moreover  if $a(t,x)$   does not depend on $x$ and $V(t,x)$ grows at most linearly in $x$, i.e. $\abs {\partial_x^\alpha V(t,x)} \leq C_\alpha \la x \ra^{(1- |\alpha|)_+} \ , $ $ \forall t \in \R $, then 
$K$  is  $H_{osc}^{1/2}$-bounded.  Then  we can apply our general results (Theorems 1.2 and 1.3)  to get
\begin{theorem}
\label{thm:apl4}
Under assumptions  (\ref{subquad})  we have:\\
(i)  For each   $t$ , the Hamiltonian $H_{a,V}(t)$ is essentially self-adjoint in $L^2(\R^d)$  with core ${\cal S}(\R^d)$. \\
(ii) 
For every $k\in\N$,  the Cauchy problem \eqref{eq:sc0} with $L(t) \equiv H_{a,V}(t)$ is globally well-posed
in the weighted Sobolev space  ${\cal H}^k(\R^d)=D(H_{osc}^{k/2})$. \\
(iii) If furthermore $a(t, x)=a(t)$ depends only on time $t$  and $\abs {\partial_x^\alpha V(t,x)} \leq C_\alpha \la x \ra^{(1- |\alpha|)_+} \ , $ $ \forall t \in \R $, then  for any $k\in\N$, we have  the  bound:
$$
\Vert\cU(t,s)\Vert_{\cL(\cH^k)} \leq C_k\la t-s\ra^k \ . 
$$
\end{theorem}
\noindent{\em Comparison with previous results:}  Theorem \ref{thm:apl4} (i)  and (ii) where proved by Yajima in \cite{ya1, ya2}  by a  different method. We recover  them  as a  consequence of our  general  results. Notice that $V(t,x)$ has no fixed sign.

\subsection{ Differential systems of first  order}
Let $A_j(t)$, $1 \leq j \leq d$,   and $B(t,x)$   be  Hermitian 
$N\times N$ matrices, the $A_j$'s  depend only on time,  $A_j\in C_b(\R, M_n(\C))$,
 while   $B(t, x)\in C_b(\R, C^\infty(\R^n,M_n(\C))$ satisfies
 $\forall$ multi-indexes $\alpha$
$$
\vert\partial_x^\alpha B(t,x)\vert \leq C_\alpha\la x\ra^{(m-\vert\alpha\vert)_+} \ , \qquad \forall t \in \R \ . 
$$
 Let  us  consider  equation \eqref{eq:sc0} with $\di{L(t) = \sum_{1\leq j\leq d}A_j(t)D_j + B(t,x)}$. Such  equation is symmetric-hyperbolic. A basic example is the Maxwell system. An other example is 
the Dirac equation with a time dependent  electro-magnetic field:
$$
\im \h \frac{\partial\psi(x,t) }{\partial t}  = \left(\beta mc^2 + c\left(\sum_{n \mathop =1}^{3}\alpha_n(\h D_n)\right) + V(t,x)\right) \psi (x,t)
$$
where $D_n=\im^{-1}\frac{\partial}{\partial_{x_n}}$,  $(\alpha_n, \beta)$ are the Dirac  matrices
  and $V(t, x)$  is $4\times 4$ Hermitian matrix (the electro-magnetic potential).

Le us introduce the  reference operator 
 $H = (-\triangle + \vert x\vert^{2k})\1_{\C^N}$ and the scale of Hilbert spaces $\cH^r_k =D\left((-\triangle + \vert x\vert^{2k})^{r/2}\right)$, for any $r\geq 0$.  We  compute the  commutator $[L(t), H]$. If $m\leq k+1$  we can  check that 
 $ [L(t), H]$  is $H$-bounded  and if $m\leq k$  then $[L(t), H]$  is  $H^{1-\theta}$-bounded  with $\theta = \frac{1}{2k}$. So Theorem \ref{thm:flow2}  and 
 Theorem \ref{thm:norm}   can be applied to give 
 \begin{theorem}
 \label{syst} 
 Let $m \leq k+1$. Then    problem \eqref{eq:sc0}   is well-posed in the weighted  Sobolev spaces $\cH^r_k$ for any $r \geq 0$. 
 Moreover if $m \leq k $ then we have   for any $r\geq 0$, 
 $$
\Vert\cU(t,s)\Vert_{\cL(\cH^r_k)} \leq C_r\la t-s\ra^{kr} \ . 
$$
 \end{theorem}
 \begin{remark}
 It is easy to see that Theorem \ref{syst} holds true if $A_j(t)= A_j(t,x)$ are smooth  in $x$ and  satisfy
 $$
 \vert \partial_x^\alpha A_j(t,x)\vert \leq C_\alpha \la x\ra^{(1-\vert\alpha\vert)_+} \ , \qquad \forall t \in \R \ . 
 $$
 \end{remark} 



\subsection{A discrete model example}
This model  was considered in \cite{bajo}.  We keep  our notations  which are  different from \cite{bajo}.\\
Let us consider  the Hilbert space $\cH^0=\ell^2(\Z^d)$ and its canonical Hilbert base $\{{\rm e}_n\}_{n\in\Z^d}$  defined by ${\rm e}_n(k) =\delta(n-k)$, $k\in\Z^d$.  We consider equation \eqref{eq:sc0} with   Hamiltonian
 $L(t) = H_0 +V(t)$  where $H_0$ is the discrete Laplacian and $V(t)$ is a diagonal operator:
 $$
 H_0 u(n) = \sum_{\vert k-n\vert =1} u(k) \ , \qquad V(t)u(n) = \omega_n(t)u(n) \ , 
 $$ 
  (here $\vert\cdot\vert$  denotes the sup norm). 
 Assume that  $\omega_n(t) $ are real and that  there  exists  $M \geq 0$ such that 
 \bea\label{polydi}
 \vert\omega_n(t)\vert \leq C\la n\ra^M \ , \quad \forall t\in\R \ .
 \eea
   Introduce  the reference operator 
$ Hu(n) = \la n\ra u(n)$ and the usual scale of Sobolev spaces $\cH^r = D(H^{r/2}) \equiv  \{\{u(n)\}_{n \in \Z^d} \colon \sum_{n\in\Z^d}\la n\ra^{r}\vert u(n)\vert^2 < +\infty\}$.\\
Let us check that  assumptions (H0), (H1), (H3)  are satisfied  with $\tau=0$. With (\ref{polydi}) assumption (H0)  and (H1)  are satisfied.  Now we verify (H3).
\begin{lemma}
The  commutator $[H, H_0]$  is bounded on $\cH^r$  for every $r\geq 0$.
\end{lemma}
\begin{proof}
A direct computation gives
$$
\Big([H_0, H] u \Big)(n)  = \sum_{  |\epsilon | = 1, \ \epsilon \in \Z^d} (\la n + \epsilon \ra - \la n \ra ) \, u(n+\epsilon) \ .
$$
Thus for any $u, v \in \cH^0$ we have
$$
 \abs{ \la v, [H_0, H]u\ra_{\cH^0}} = \abs{ \sum_{|\epsilon| = 1} \sum_{n \in \Z^d}   (\la n + \epsilon \ra - \la n \ra ) \, u(n+\epsilon) \, v(n) } \leq 
 2 d \, \norm{u}_{\cH^0} \, \norm{v}_{\cH^0} 
 \ ,
 $$
 which shows that $[H_0, H]$ is bounded on $\cH^0$.\\
%
  Now we prove that $[H_0, H]$  is bounded on $\cH^r$  for any  $r>0$. 
  An easy computation gives
\bea
  H^{r}[H_0, H]H^{-r}u &
 =&  \sum_{m\in\Z^d, \vert\epsilon\vert=1} \Big(u(m+\epsilon) \left[ \la m \ra^r \, (\la m+\epsilon \ra - \la m \ra ) \, \la m + \epsilon \ra^{-r} \right] \Big)\  {\rm e}_m
\eea
Since 
$$
\sup_{m \epsilon \in \Z^d, \ |\epsilon|=1}\abs{\la m \ra^r \, (\la m+\epsilon \ra - \la m \ra ) \, \la m + \epsilon \ra^{-r} }\leq C \ , 
$$
 it 
 results that $[H_0, H]$  is bounded on $\cH^r$  for any  $r>0$. 
\end{proof}
Thus it follows that $[H_0+V(t), H]$ is a bounded operator. Applying Theorem \ref{thm:norm} with $\tau = 0$ we get in particular  that  the propagator $\cU(t,s)$  associated with $L(t)$  is well defined as a bounded operator
 on $\cH^r$  and  satisfies
 \begin{equation}
 \label{eq:estapl5}
 \Vert\cU(t,s)\Vert_{\cL(\cH^r)} \leq C_r\la t-s\ra^{\frac{r}{2}} \ , \qquad \forall t, s\in\R.
 \end{equation}

 \noindent{\em Comparison with previous result:}
Estimate \eqref{eq:estapl5} appeared first in the work of Barbaroux and Joye  \cite{bajo}.  Zhao \cite{zhao16}  showed that when $d = 1$  there exists a family of functions $\omega_n(t)$ s.t. $\Vert\cU(t,0)\Vert_{\cL(\cH^1)} \leq C_1\la t-s\ra^{\frac{1}{2}}$, saturating the bound \eqref{eq:estapl5}. 
In \cite{zhang16}, Zhang and Zhao extended this result to general $r >1$ and a larger family of functions $\omega_n(t)$.

\subsection{Pseudodifferential operators on $\R^n$}
We consider here  equation \eqref{eq:sc0} in case $L(t)$ is a   time dependent pseudifferential  operator on $\R^n$. 
A very general  Weyl calculus is detailed in the book \cite{ho}. 
We   recall  some  basic facts needed here on some particular cases and  some  more properties in Appendix \ref{app:pseudo}.\\
Recall that for  smooth symbols  $A(x,\xi)$, $x, \xi \in \R^n$,  one defines the Weyl-quantization ${\rm Op}^W_\hbar(A)$ by the formula 
\begin{equation}
\Big({\rm Op}^W_\hbar(A) u\Big)(x) := \frac{1}{(2 \pi \hbar)^n} \iint_{y, \xi}  e^{\frac{\im}{\hbar} ( x-y)\cdot  \xi } \ A\left(\frac{x+y}{2}, \xi \right) u(y) \ dy  d\xi \ . 
\end{equation}
This formula  is valid for $A$ in the space ${\cal S}(\R^{2n})$ of Schwartz functions and 
one can extend it  to  functions in more general  classes. 
To introduce the class we are interested in, let us introduce the weight 
$$\lambda_{k, \ell}(x,\xi) = (a+\vert x\vert^{2\ell} + \vert\xi\vert^{2k})^{\frac{1}{2k\ell}} \ . $$
Here the real number $a>0$  will be chosen large enough.

\begin{definition}
\label{def:symbol} 
Fix $\nu \in \R$, $k, \ell \in \R_+$. A  function  $A(x, \xi) \in C^{\infty}(\R^n_x\times \R^n_\xi, \C)$ will be called a {\em symbol} in the class  $\wt S_{k,\ell}^\nu$ if for every $\alpha, \beta \in \N^n$ there exists a constant $C_{\alpha, \beta} >0$ s.t.
\beq\label{skl}
\abs{ \partial_x^\alpha \partial_\xi^\beta A(x, \xi) } \leq C_{\alpha, \beta}
\  \lambda_{k,\ell}(x,\xi)^{(\nu-k\vert\alpha\vert -\ell\vert\beta\vert)_+} \ , 
\eeq
where $r_+ := \max(0, r)$.
\end{definition}

Such class was introduced in   \cite{ro1, hero}, where it is proved that  ${\rm Op}^W_\hbar(A)$ is well defined for  $A \in \wt S_{k,\ell}^\nu$.

\begin{remark}
\label{rem:pseudo1}
(i) For $\nu = 2$, $k=\ell=1$, $\wt S^2_{1, 1}$ is the class of symbols satisfying the  {\em sub-quadratic }growth condition
$$ \abs{ \partial_x^\alpha \partial_\xi^\beta A(x, \xi) } \leq C_{\alpha, \beta} \ , \qquad \forall \ |\alpha| + |\beta| \geq 2.
$$
(ii) The function $\lambda_{k,\ell}^\nu$ belongs to $\wt S_{k,\ell}^\nu$.
\end{remark}
\noi We endow $\wt S_{k,\ell}^\nu$ with the family of semi-norms defined by
\begin{equation}
p_{\alpha \beta}^\nu(A) := \sup_{x, \xi \in \R^{n}}  \lambda_{k, \ell}(x, \xi)^{-(\nu -k|\alpha| - \ell |\beta|)_+} \  \abs{\partial_x^\alpha \partial_\xi^\beta A (x, \xi)} \ ,
\end{equation}
and for every integer $M$ we define
\begin{equation}
|A|_{M,\nu} := \sup_{|\alpha| + |\beta| \leq M} \ p^\nu_{\alpha \beta} (A) \ . 
\end{equation}
We  define  now the   reference operator $H$  to be  
$$
H \equiv \wh H^{k+\ell}_{k,\ell} := {\rm Op}^W_\hbar(\lambda_{k,\ell}^{k+\ell}) \ .
$$
The constant 
	$a>0$ in the definition of $\lambda_{k, \ell}$ is chosen  large enough  such that ${\widehat H}_{k,\ell}^{k+\ell}$ is a  positive   self-adjoint  operator in
	$L^2(\R^n)$.  As usual we define the scale of Hilbert spaces 
	$\cH^{r}:=D\left(\left({\widehat H}_{k,\ell}^{k+\ell}\right)^{r/2}\right)$   for every real $r\geq 0$.  Formally one has 
	\beq\label{caractsob}
	{\cal H}^{r} = \{u \in L^2(\R^n) \  \vert \  u\in H^{\frac{(k+\ell)r}{2k}}(\R^n),\; \vert x\vert^{\frac{(k+\ell)r}{2\ell}}u\in L^2(\R^n)\}
	\eeq
	equipped with a natural norm  of Hilbert  space.

\begin{remark}
In the class of sub-quadratic symbols $\wt S^2_{1, 1}$ one has simply that  $H =\wh H_{osc} \equiv -\Delta + |x|^2$  (harmonic oscillator) and $\cH^r$ are the more classical spaces
\begin{equation}
\label{Hk}
\cH^r(\R^n) := \left\{u \in L^2(\R^n) \ \vert \ x^\alpha \, (\hbar\partial_x)^\beta \, u \in L^2(\R^n) , \  |\alpha| + |\beta| \leq r \right\} \ .
\end{equation}
\end{remark}

In order to study evolution equations we need to consider  time dependent symbols. We give the following
 \begin{definition} Let $\cI \subseteq \R$.
We say that a time-dependent symbol $A(\cdot) \in C^0_b(\cI, \wt S^\nu_{k, \ell})$ iff $A(t) \in  \wt S^\nu_{k, \ell}$ for every $t \in \cI$ and the map $t \mapsto p^\nu_{\alpha\beta}(A(t))$ is continuous  and uniformly bounded for every $\alpha,\beta$.
\end{definition}

We are ready to state the  results:

\begin{theorem}\label{bamb}
	Fix $k, \ell \in \R_+$ and  $\nu \in \R$ with $ \nu \leq k + \ell$. Then the following is true:
	\begin{itemize}
		\item[(i)]  Assume that $A$ is a real  symbol with $A\in \wt S^\nu_{k,\ell}$. Then  $ {\rm Op}^W_\hbar (A)$  is essentially self-adjoint with core
		${\cal S}(\R^n)$.
		\item[(ii)] Assume that $A(\cdot) \in C^0_b(\R, \wt S^\nu_{k,\ell} )$. Then the Schr\"odinger equation \eqref{eq:sc0} with $L(t) \equiv {\rm Op}^W_\hbar (A(t))$ generates a flow $\cU(t,s)$ which fulfills $(i)$--$(iv)$ of Theorem \ref{thm:flow2}.
		\item[(iii)] If $\nu < k+\ell$, then the flow $\cU(t,s)$ fulfills the bound
		$$
		\norm{\cU(t,s)}_{\cL(\cH^r)} \leq C \la t-s \ra^{\frac{r \, (k+\ell)}{2 \, (k + \ell - \nu)}} \ . 
		$$
	\end{itemize} 
\end{theorem}
\begin{proof} (i) It follows by   the same arguments used to prove item (ii) and  Proposition \ref{propone}.\\
(ii) 
We verify the assumptions of Theorem \ref{thm:flow2} using the  symbolic calculus  for symbols in the classes $\wt S_{k,\ell}^\nu$. By Remark \ref{rem:pseudo1}, $\lambda_{k, \ell}^\nu$ is a symbol in $\wt S_{k,\ell}^\nu$ and it is invertible provided $a$ is sufficiently large. By symbolic calculus (see Appendix \ref{app:pseudo}) the operator $H$ is invertible and its inverse    $H^{-1} \in {\rm Op}^W_\hbar( \wt S^{-(k+\ell)}_{k,\ell})$. 
	 It follows easily by symbolic calculus (see Theorem \ref{symb.cal})  that 
	  $[A(t), H]H^{-1} \in {\rm Op}^W_\hbar(\wt S^{\nu - (k+\ell)}_{k,\ell})$. Then  by  Calderon-Vaillancourt theorem (see Theorem \ref{thm:cv}) if   $\nu \leq k+\ell$  such operator is bounded on the scale of Hilbert spaces (\ref{caractsob}).  Theorem \ref{thm:flow2} can be applied.\\
	  (iii) One applies Theorem \ref{thm:norm} remarking  that 
	  $[A(t), H]H^{-\tau} \in {\rm Op}^W_\hbar(\wt S^{\nu - \tau(k+\ell)}_{k,\ell})$. Then if $\nu < k+\ell$, choosing $\tau = \frac{\nu}{k+\ell}$ one has that $\tau <1$ and  $[A(t), H]H^{-\tau}$ is a bounded operator.
\end{proof}
\begin{remark}
 If $A\in\wt S_{1,1}^2$   then $A(t,x,\xi)$ is a sub-quadratic   symbol in $(x,\xi)$ and we recover  a result  already proved by Tataru \cite{ta}  using a  complex  WKB  parametrix  for the  Schrödinger equation.
\end{remark}

\begin{example}[a balance between position and momentum behavior]
	 Consider a symbol $A$ of the form
	 $$
	 A(x, \xi) = f(\xi) + g(x)
	 $$
	 where the functions $f, g$   are smooth and fulfill
	\beq\label{exest}
	 |\partial_x^\alpha f(\xi) | \leq C_\alpha\la \xi \ra^{(p-\vert\alpha\vert)_+} \ , 
	 \quad |\partial_x^\alpha g(x) | \leq C_\alpha\la x \ra^{(q-\vert\alpha\vert)_+} \ ,
	\eeq
	 for some $p, q \in \Q$ such that $\frac{1}{p} + \frac{1}{q} = 1 $, with $1<p<+\infty$. Then ${\rm Op}^W_\hbar(A)$ is essentially self-adjoint. 
	 Indeed in such case it is possible to find integers $k,\ell$ such that
	 $p= (k+\ell)/\ell$, $q=(k+\ell)/k$. Then with such $k, \ell$, one verifies easily that  
	 $A \in \widetilde S_{k,\ell}^{k + \ell}$.\\
	 Moreover if $f, g$  are  time-dependent the operator $A^w(t)$  generates a propagator  satisfying $(i)-(iii)$.
	 It satisfies (iv) if   furthermore estimates (\ref{exest}) are uniform in time $t\in\R$.
\end{example}

%

 \appendix

\section{Essentially self-adjointness}
In this  section we give the proof of essentially self-adjointness which is based to the commutator method of Nelson \cite{ne}. The method was further extended by Faris and Lavine \cite{fala}.
The general  principle is related  with the  Friedrichs smoothing  method \cite{fri}.\\

We start to recall some standard definitions. Let $\cH$ be a complex Hilbert space and $(\cdot, \cdot)_{\cH}$ its inner product. Let $\cK\subset \cH$ be a dense subspace. Let $L$ be a linear operator with domain $D(L) = \cK$ and {\em symmetric}, i.e. verifying
$$
(L u, v)_{\cH} = (u, Lv)_{\cH} \qquad \mbox{ for every } u, v \in \cK \ .
$$
We say that $(L, \cK, \cH)$ is {\em essentially self-adjoint} if $L$ admits a unique self-adjoint extension as an unbounded operator on $\cH$. When this is true $\cK$ is called a {\em core} for $L$.  Let $(L, \cK, \cH)$ be a symmetric operator.  It is known that the operator $(L, \cK, \cH)$ is {\em closable}, i.e. it admits at most one closed extension $(L_{\rm min}, D(L_{\rm min}), \cH)$. $L_{\rm min}$ is the smallest closed extension of $L$, and we call  $(L_{\rm min}, D(L_{\rm min}), \cH)$ the {\em minimal operator} associated to $L$.

We denote by $(L_{\rm min}^*, D(L_{\rm min}^*), \cH)$ the adjoint of $(L_{\rm min}, D(L_{\rm min}), \cH)$. Recall that by definition 
$$
D(L_{\rm min}^*) = \left\{ u \in \cH \ : \ |(u, L v)_{\cH}| \leq C_u \norm{v}_\cH \ , \forall  v \in \cK \right\} \ .
$$ 
It is a classical result \cite[Proposition 2]{robert_book} that $(L_{\rm min}^*, D(L_{\rm min}^*), \cH)$ is the largest closed extension of $L$. Denote $L_{\rm max} := L_{\rm min}^*$. Then we call $(L_{\rm max}, D(L_{\rm max}), \cH)$ the {\em maximal operator} associated to $L$. Thus $L$ is essentially self-adjoint if $L_{{\rm min}}$ is self-adjoint. This means that $(L_{\rm min}, D(L_{\rm min}), \cH)$ and $(L_{\rm max}, D(L_{\rm max}), \cH)$ coincide.\\

Let us introduce   a smoothing family of operators
$\{R_\varepsilon\}_{\varepsilon\in]0, 1]}$   satisfying
\bea\label{sm}
\Vert R_\varepsilon\Vert_{\cL(\cH)}\leq C, \; \forall \varepsilon\in]0, 1], \\
R_\varepsilon {\cH}\subseteq\cK ,\; \; \forall \varepsilon\in]0, 1], \\
\lim_{\varepsilon\rightarrow 0}\Vert R_\varepsilon u - u\Vert_{\cH}= 0,  \forall u\in \cH.
\eea
\begin{proposition}\label{propsm} Let $(L, \cK, \cH)$ be a symmetric operator. 
	Assume that  the commutators $[R_\varepsilon, L]= R_\varepsilon L - LR_\varepsilon$  satisfies 
	\bea
	\label{smsa}
	\Vert [R_\varepsilon, L]u\Vert_{\cH}\leq C\Vert u\Vert_{\cH},\; \forall  u\in \cK,  \ \ \ \forall \varepsilon\in]0, 1],\\
	 \lim_{\varepsilon\rightarrow 0}\Vert [R_\varepsilon, L] u\Vert_{\cH}= 0 \ , \forall u\in \cK \ .
	\eea
	Then $(L, \cK, \cH)$  is essentially self-adjoint.
\end{proposition}
\begin{proof}
	We  have to prove that $D(L_{\rm max})\subseteq D(L_{\rm min})$. Let  $u\in D(L_{\rm max})$. 
	Then by property (\ref{sm}), $u_\varepsilon:=R_\varepsilon u\in D(L_{\rm min})$  and
	$u_\varepsilon\rightarrow u$  in ${\cH}$. But we have
	$$
	Lu_\varepsilon = R_\varepsilon Lu + [L, R_\varepsilon]Au \ . 
	$$
	So by assumption (\ref{smsa})  we get  that $\di{\lim_{\varepsilon\rightarrow 0}Lu_\varepsilon = Lu}$   so  $u\in D(L_{\rm min})$.
	$\square$
\end{proof}
The following criterium  apply Proposition \ref{propsm}    and is due to Nelson \cite{ne}.
\begin{proposition}\label{propone}
	Let  $H$  be a positive self-adjoint operator  in ${\cal H}$ with a dense domain $D(H)$. \\
	Let  $L$  be a linear and symmetric operator  from $D(H)$ into ${\cal H}$.\\
	Assume   that the 
	operators $LH^{-\tau}$ ($\tau >0$)  and  $H^{-1/2}[H, L]H^{-1/2}$  are  bounded on ${\cal H}$  then   $(L, D(H), \cH)$  is essentially self-adjoint.
\end{proposition}
\begin{proof}
	Let us repeat  here  the  rather simple proof.  We have to verify that the assumptions of Proposition  \ref{propsm}    are satisfied with
	$R_\varepsilon = {\rm e}^{-\varepsilon H}$.\\
	First we have, for $u\in D(H^\tau)$,  
	$$
	[{\rm e}^{-\varepsilon H}, L]u =  {\rm e}^{-\varepsilon H}Lu - L {\rm e}^{-\varepsilon H}u \ . 
	$$
	We have $Lu\in {\cH}$  so $\di{\lim_{\varepsilon\rightarrow 0}\Vert{\rm e}^{-\varepsilon H}Lu-Lu\Vert_{\cH} = 0}$. 
	Writing $ L {\rm e}^{-\varepsilon H}u = (LH^{-\tau})({\rm e}^{-\varepsilon H}H^\tau u)$ we also have 
	$ \di{\lim_{\varepsilon\rightarrow 0}\Vert L {\rm e}^{-\varepsilon H}u-Lu\Vert_{\cH} = 0}$. So we have proved
	\beq
	\lim_{\varepsilon\rightarrow 0}\Vert[{\rm e}^{-\varepsilon H}, L]u\Vert_{\cH} = 0,\;\; \forall u\in D(H^\tau).
	\eeq
	Let  us estimate  now $\Vert  [{\rm e}^{-\varepsilon H}, L]\Vert_{\cL(\cH)}$.
	We  start with   the following  known formula
	\beq
	\label{hcom0}
	[{\rm e}^{-\varepsilon H}, L]  =  -\int_0^\varepsilon{\rm e}^{-(\varepsilon-s) H}[L, H]{\rm e}^{-sH}ds \ . 
	\eeq
	Following \cite{ne} we have 
	\beq
	\label{hcom1}
	[{\rm e}^{-\varepsilon H}, L]  =  
	-\int_0^\varepsilon{\rm e}^{-(\varepsilon-s) H}H^{1/2}(H^{-1/2}[L, H]H^{-1/2})H^{1/2}{\rm e}^{-sH}ds \ . 
	\eeq
	Using that
	$$
	\Vert H^{1/2}{\rm e}^{-sH}\Vert_{\cL(\cH)} =\sup_{\lambda\geq 0}\lambda^{1/2}{\rm e}^{-s\lambda}\leq Cs^{-1/2}
	$$
	and the beta function  computation: $\int_0^\epsilon(\epsilon-s)^{-1/2}s^{-1/2}ds = B(1/2, 1/2) = \frac{\pi}{2}$  we get 
	$$
	\sup_{\epsilon\in]0, 1]}\Vert  [{\rm e}^{-\varepsilon H}, L]\Vert_{\cL(\cH)} < +\infty.
	$$
\end{proof}

\section{Technical estimates for perturbations smooth in time}
\label{app:te}
In this section we prove some technical estimates which are useful in the proof of Theorem \ref{thm:grest}.	\\

First we state a result about boundedness of the resolvent. In all the section $H$ will be a self-adjoint, positive  operator in $\cH^0$ fulfilling (Hgap).
Let  $H_W(t) := H+ W(t)$,  $W(t)$ a symmetric operator fulfilling (Vs$)_n$.
%
\begin{lemma}
\label{sandres} 
Assume that  $W$ fulfills {\rm (Vs$)_n$}. Then 
  for every $z \notin\sigma(H_W(t))\cup\sigma(H)$  such that 
  $$R_{n,0} \, \, \Vert H^\nu(H-z)^{-1}\Vert \leq \frac{1}{2} \ , $$
    we have for any integer  $0 \leq p \leq n$, any real  $0 \leq \theta \leq 1$, 
\begin{align}
\label{HL1}
\sup_{t \in \R} 
\Vert H^{p+\theta}(H_W(t)-z)^{-1}H^{-p}\Vert \leq 2 \Vert H^\theta (H-z)^{-1} \Vert \ . 
\end{align}
\end{lemma}
\begin{proof}
This  is a consequence of the resolvent identity:
\begin{equation}
(H_W(t)-z)^{-1} = (H-z)^{-1} -(H_W(t)-z)^{-1}W(t)(H-z)^{-1} \ ,
\end{equation}
so we have for $0 \leq \theta \leq 1$, $0 \leq p \leq n$, 
\begin{equation}
H^{p+\theta}(H_W(t)-z)^{-1}H^{-p} = H^\theta (H-z)^{-1}  - \left(H^{p+\theta}(H_W(t)-z)^{-1}H^{-p}\right)
\left(H^pW(t)H^{-p-\nu}\right)H^\nu(H-z)^{-1} \ . 
\end{equation}
Provided 
$$
\sup_{t \in \R} \Vert H^p\, W(t) \, H^{-p-\nu}\Vert \, \Vert H^\nu(H-z)^{-1}\Vert \leq R_{n,0} \, \, \Vert H^\nu(H-z)^{-1}\Vert \leq \frac{1}{2} \ , 
$$
estimate (\ref{HL1}) follows.
\end{proof}

    \begin{lemma}
   \label{lem:an02}   
   Fix $n \geq 0 $. Let $P(t)$ be an operator fulfilling {\rm (Vs$)_n$} with 
    	\begin{equation}
   	\label{P.est02}
   	\sup_{t\in\R}\Vert H^p \, \partial_t^\ell P(t) H^{-p-\nu}\Vert \leq D_{n,\ell}  \ , \quad \forall \ell \geq 0 \ , \ 0\leq p \leq n \ . 
   	\end{equation}
	Consider the operator $H + P(t)$. Then, provided $\tJ$ is sufficiently large, the following holds true:
	\begin{itemize}
	\item[(i)] $H + P(t)$ fulfills $\wt{{\rm(Hgap)}}$ uniformly in time $t\in\R$. 
	\item[(ii)] Let $\Gamma_j$ be as in \eqref{gamma}.  Any $\lambda \in \Gamma_j$ belongs to the resolvent set of the operator $H+P(t)$.  Denote $R_P(t,\lambda):= \left(H + P(t) - \lambda \right)^{-1}$. Then for any $\lambda \in \Gamma_j$, $j \geq 1$ one has  
	\begin{align}
\label{res01}
&\sup_{t \in \R}\norm{H^{p} R_P(t, \lambda) H^{-p}} \leq \frac{2}{{\rm dist}(\lambda, \sigma(H))}  \ , \quad \forall 0 \leq p \leq n \\
\label{res02}
&\sup_{t \in \R} \norm{H^p \, \partial_t^\ell R_P(t, \lambda) \, H^{-p}} \leq \frac{C_{n,\ell}  }{ {\rm dist} (\lambda, \sigma(H))}  \frac{1}{\wt\Delta_{j-1}^\delta} \ , \quad \forall 0 \leq p \leq n \ , \ \ell \geq 1 \ ,
\end{align}
where $C_{n,\ell}$ does not depend on $j, \tJ$.
	\item[(iii)]  For any $j \geq 1$ define the projector 

\begin{equation}
\label{eq:prj}
\Pi_j(t) := -\frac{1}{2\pi \im} \oint_{\Gamma_j} R_P(t,\lambda) \, d\lambda \ .  
\end{equation}
It fulfills 
$$
    \sup_{t \in \R} \norm{H^p \, \partial_{t}^{\ell+1}  \Pi_j(t) \, H^{-p}} \leq  \frac{C_{n,\ell} }{\wt \Delta_{j-1}^{\delta}} \ , \qquad \forall 0 \leq p \leq n, \ \ell \geq 0 \ ,
   $$
   where $C_{n,\ell}$ does not depend on $j, \tJ$.
   \end{itemize}
   \end{lemma}
   \begin{proof}
   (i) It follows by  Lemma \ref{perturb} provided $\tJ$ is sufficiently large to fulfill condition \eqref{condz}. Thus $\sigma(H+ P(t)) \subseteq \bigcup_{j\geq 1} \wt\sigma_j$, (with $\wt\sigma_j$ as in \eqref{new.clus}). 
      
   (ii) By the previous item  each $\Gamma_j$ is contained in the resolvent set of $H+P(t)$. 
To estimate $\norm{R_P(t, \lambda)} $ we use Lemma \ref{sandres} and  Lemma \ref{lem:est:res0}. Indeed for $\tJ$ sufficiently large and $\lambda \in \Gamma_j$ we have 
$$
D_{n,0} \norm{ H^\nu (H- \lambda)^{-1}} \leq D_{n,0} \, \frac{\wt \tC_H }{ \wt\Delta_{j-1}^{\delta}} \leq D_{n,0} \, \frac{\wt \tC_H}{ 2^{\tJ \mu \delta}} \leq \frac{1}{2} \ , 
$$
hence we can apply Lemma \ref{sandres} with $\theta = 0$ to obtain estimate \eqref{res01}. 

To prove \eqref{res02}, use the formula
   \begin{align}
   \label{ris.der0}
  & \partial_{t}^\ell R_P(t, \lambda) 
  =    \sum_{k =1}^{\ell} \sum_{n_1, \ldots, n_k \in \N \atop n_1 + \cdots +n_k = \ell} \binom{\ell}{n_1 \cdots n_k} \, R_P(t, \lambda) \, (\partial_{t}^{n_1} P(t)) \, R_P(t, \lambda) \, (\partial_{t}^{n_2} P(t)) \, \cdots   (\partial_{t}^{n_k} P(t)) \, R_P(t, \lambda)
   \end{align}
   and take the conjugates with $H^p$ to obtain 
   \begin{align}
   \label{ris.der1}
& H^p \,   \partial_{t}^\ell R_P(t, \lambda) \, H^{-p} \\
\notag
   & = \sum_{k =1}^{\ell} \sum_{n_1, \ldots, n_k \in \N \atop n_1 + \cdots +n_k = \ell} \binom{\ell}{n_1 \cdots n_k} \, (H^p R_P(t, \lambda) H^{-p}) \, [H^p (\partial_{t}^{n_1} P(t)) \, H^{-p-\nu}] \,  [H^{p+\nu} \, R_P(t, \lambda) \, H^{-p}] \cdots\\
   \notag
   & \qquad \cdots     [H^{p}\, (\partial_{t}^{n_k} P(t)) \, H^{-p-\nu}] \, [H^{p+\nu} \, R_P(t, \lambda)H^{-p}]  \ .
   \end{align}
Then using estimate \eqref{res01}, Lemma \ref{sandres} with $\theta = \nu$, estimates \eqref{P.est02}   
and \eqref{est:res0}, we obtain   for $\lambda \in \Gamma_j$, $j \geq 1$
   \begin{align}
   \sup_{\ t \in \R} \norm{H^p \, \partial_{t}^\ell R_P(t, \lambda) H^{-p} } 
   &\leq \frac{C_{n,\ell}  }{ {\rm dist} (\lambda, \sigma(H))}  \frac{1}{\wt\Delta_{j-1}^\delta}\ , \quad \forall \ell \geq 1 \ , \ 0 \leq p \leq n \ , 
   \end{align}
where the $C_{n,\ell}$ can be chosen independent of $j, \tJ$.
(iii) For  $\ell \geq 1$, one has $H^p \, \partial_{t}^\ell \Pi(t) \, H^{-p} = -\frac{1}{2\pi \im} \oint_{\Gamma_j} H^p \, \partial_{t}^\ell R_P(t, \lambda) H^{-p} \, d\lambda $, hence by \eqref{res02}
       \begin{align*}
  \sup_{ t \in \R} \| H^p \, \partial_{t}^\ell \Pi(t) \, H^{-p} \| \leq    \frac{C_{n,\ell}}{\wt\Delta_{j-1}^\delta} \,  \frac{1}{2\pi}\oint_{\Gamma_j} \frac{ \, d\lambda}{{\rm dist} (\lambda, \sigma(H))} \leq  \frac{C_{n,\ell} }{\wt \Delta_{j-1}^\delta}
   \end{align*}
   where to pass from the first to the second inequality  we used that, deforming the contour $\Gamma_j$ to two vertical lines passing between the middle of the gaps one has   
    \begin{equation}
   \label{bb0}
   \begin{aligned}
   \frac{1}{2\pi}\oint_{\Gamma_j} \frac{d\lambda}{{\rm dist} (\lambda, \sigma(H))} \leq    \frac{1}{2\pi} \int_{-\infty}^\infty \left(\frac{1}{(\wt \Delta_{j-1}/2)^2 + x^2)^{1/2}} +  \frac{1}{(\wt \Delta_{j}/2)^2 + x^2)^{1/2}}\right) \ dx \leq 2
   \end{aligned}
   \end{equation}
     \end{proof}
 \begin{lemma}
   \label{lem:an03}   
  Fix $n \geq 0$.   Let $P(t), Q(t)$ be operators fulfilling {\rm (Vs$)_n$} with estimates as in \eqref{P.est02}. 
   Furthermore assume that $B(\cdot) \in C^\infty_b(\R, \cL(\cH^{2p}))$, $\forall 0 \leq p \leq n$, fulfilling the  estimates
   \begin{equation}
   \label{B.aa0}
  \sup_{t \in \R} \norm{H^p \, \partial_{t}^\ell B(t) \, H^{-p} } \leq b_{n, \ell} \ , \quad \forall \ell \geq 0 , \ 0 \leq p \leq n \ .   
   \end{equation}
Provided $\tJ$ is sufficiently large,  the operator
   $$
   K(t) := -\frac{1}{2\pi \im} \oint_{\Gamma_j} R_P(t,\lambda) \, B(t) \, R_Q(t,\lambda) \,  d\lambda
   $$
   is well defined and bounded from $\cH^{2p}$ to itself, $\forall 0 \leq p \leq n$, and  fulfills
   $$
   \sup_{t \in \R} \norm{H^p \, \partial_{t}^\ell  K(t) \, H^{-p}}
    \leq 
   \frac{ C_{n,\ell}}{  \wt \Delta_{j-1}} \, \sup_{l \leq \ell} b_{n,l} \ , \qquad \forall \ell \geq 0 \ , \ 0 \leq p \leq n \ .   
   $$
   \end{lemma}
   \begin{proof}
  By Lemma \ref{lem:an02}, provided $\tJ$ is sufficiently large,  $\Gamma_j$ is contained in the resolvent sets of $H+ P(t)$ and $H+ Q(t)$, thus $K(t)$ is well defined. To estimate it, take first  $\ell = 0$. Then by \eqref{res01} and \eqref{B.aa0}
      $$
    \sup_{t \in \R} \norm{ H^p \,  K(t) \, H^{-p}} \leq    \frac{1}{2\pi} \oint_{\Gamma_j} \frac{ 4\,  b_{n,0} \,  d\lambda}{{\rm dist}(\lambda, \sigma(H))^2} \leq \frac{16 \, b_{n,0} }{ \wt \Delta_{j-1}} \ ,
   $$
   where once again we deformed the contour as in \eqref{bb0}.
   Take now $\ell \geq 1$. By Leibnitz formula we get
   \begin{align}
   \label{est:K10}
   \partial_{t}^\ell  K(t) & = 
    - \frac{1}{2\pi \im} \oint_{\Gamma_j} R_{P}(t, \lambda)  \ (\partial_{t}^{\ell} B(t)) \  R_{Q}(t, \lambda) \, d\lambda \\
 \label{est:K20} 
  & \quad - \sum_{n_1 + n_2 = \ell \atop n_1 \geq 1}
   \binom{\ell}{n_1 \, n_2 }  \frac{1}{2\pi \im} \oint_{\Gamma_j}(\partial_{t}^{n_1} R_{P}(t, \lambda))  \ (\partial_{t}^{n_2} B(t)) \ R_{Q}(t, \lambda) \, d\lambda \\
    \label{est:K30}
    & \quad - \sum_{n_2 + n_3 = \ell \atop n_3 \geq 1}
   \binom{\ell}{n_2\,n_3} \frac{1}{2\pi \im } \oint_{\Gamma_j} R_{P}(t, \lambda)  \ (\partial_{t}^{n_2} B(t)) \ (\partial_{t}^{n_3}R_{Q}(t, \lambda)) \, d\lambda \\
    \label{est:K40}
  & \quad -\sum_{n_1 + n_2 + n_3 = \ell \atop n_1, n_3 \geq 1}
   \binom{\ell}{n_1 \, n_2 \, n_3} \frac{1}{2\pi \im } \oint_{\Gamma_j}(\partial_{t}^{n_1} R_{P}(t, \lambda))  \ (\partial_{t}^{n_2} B(t)) \ (\partial_{t}^{n_3} R_{Q}(t, \lambda)) \, d\lambda 
   \end{align}
    Using \eqref{B.aa0} and \eqref{res01}  one finds easily that $\partial_t^\ell K(t)$ fulfills the claimed estimate (see the proof of Lemma \ref{lem:an3} for the details in the case of perturbations analytic in time).
      \end{proof}

\section{Technical estimates for perturbations  analytic in time }
\label{app:te.a}
In this section we repeat the estimates of the previous section in case of perturbations analytic in time. 

In the following    we fix  $n\in\N\cup\{0\}$ and $\tL \in \N$. Then for any $0\leq p\leq n$, all constants may depend on $n$, but {\em not } on $\tL$. Finally we denote by $A$ a constant as in \eqref{sum.in}.
 \begin{lemma}
\label{lem:an1}   
   Let $P$ and $Q$ be  operators analytic in time fulfilling  $\forall 0 \leq \ell \leq \tL$, $\forall 0 \leq p \leq n$
   	\begin{equation}
   	\label{P.est}
   	\sup_{t \in \R}\norm{H^p \ \partial_{t}^{\ell} P(t) \ H^{-p}} \leq a \,  b^{\min(\ell,1)} \, c^{k+\ell} \,  \frac{(k+\ell)!}{A (1+ \ell)^2}  \ , 
   	\end{equation}
   	and 
   	\begin{equation}
   	\label{Q.est}
   	\sup_{t \in \R} \norm{H^p \ \partial_{t}^{\ell} Q(t) \ H^{-p}} \leq d \, f^{\min(\ell,1)} \, c^{i+\ell} \  \frac{(i+\ell)!}{A (1+ \ell)^2} \ , 
   	\end{equation}
   	for some positive constants $a,b,c,d,f \in \R$ and $k, i \in \N \cup \{0\}$. Then $\forall 0 \leq \ell \leq \tL$, $\forall 0 \leq p \leq n$
   		$$ 
   		\sup_{t \in \R} \norm{H^p \ \partial_{t}^{\ell} (PQ)(t) \ H^{-p}} \leq  a \, d \, (b+f+bf)^{\min(\ell,1)}\, c^{k+i+\ell} \  \frac{(k+i+\ell)!}{A (1+ \ell)^2} \ . 
   		$$
   \end{lemma}
   \begin{proof}
   First consider the case $\ell = 0$. One has $\sup_{t \in \R} \norm{H^p \ P(t) Q(t) \ H^{-p}} \leq a\, d\, c^{k+i} \, \frac{(k+i)!}{A}$, where we used $A \geq 1$.
   Now take $1 \leq \ell \leq \tL$. By Leibnitz formula
   \begin{align*}
   \partial_{t}^{\ell} (PQ) =  (\partial_{t}^{\ell} P) Q + P  (\partial_{t}^{\ell} Q) + \sum_{j=1}^{\ell -1} \binom{\ell}{j} (\partial_{t}^{j}P) \, (\partial_{t}^{\ell-j}Q) \ . 
   \end{align*}
Using \eqref{P.est} and \eqref{Q.est} we get immediately
   \begin{align*}
   \sup_{t \in \R} \norm{H^p \ \partial_{t}^{\ell} (PQ)(t) \ H^{-p}} \leq & a \,  (b+f) \, d \, c^{k+\ell+i} \,  \frac{(k+\ell+i)!}{A (1+ \ell)^2}   \\
     & + a\,d\, b \, f \, c^{k+\ell+i} \, \frac{(k+\ell+i)!}{A^2} \sum_{n=1}^{\ell -1} \binom{\ell}{j} \binom{k+i+\ell}{k+j}^{-1} \frac{1}{(1+j)^2 \,(1+\ell-j)^2 }\ . 
   \end{align*}
   Now use that $ \binom{k+i+\ell}{k+j} \geq \binom{\ell}{j} $	
   and \eqref{sum.in} to conclude the proof.
   \end{proof}
   \begin{lemma}
   \label{lem:an2}   
    Let $P$ be an operator analytic in time  fulfilling $\forall 0 \leq \ell \leq \tL$, $\forall 0 \leq p \leq n$
    	\begin{equation}
   	\label{P.est2}
   	\sup_{t \in \R}\norm{H^p \ \partial_{t}^{\ell} P(t) \ H^{-p-\nu}} \leq a \,  b^{\min(\ell,1)} \,  \frac{\ell! \, c^{\ell} }{A (1+ \ell)^2} 
   	\end{equation}
	for some positive constants $a,b,c \in \R$.
Provided that
	 \begin{equation}
   \label{cond2}
2^{4}\,\tC_H\, a(1+ b) \leq 2^{\tJ\mu\delta} \ , 
   \end{equation}
	 the following holds true:
	\begin{itemize}
	\item[(i)] $H + P(t)$ fulfills $\wt{{\rm(Hgap)}}$ uniformly in time $t\in\R$. 
	\item[(ii)] Let $\Gamma_j$ be as in \eqref{gamma}.  Any $\lambda \in \Gamma_j$ belongs to the resolvent set of the operator $H+P(t)$.  Denote $R_P(t,\lambda):= \left(H + P(t) - \lambda \right)^{-1}$ Then for any $\lambda \in \Gamma_j$, $\forall j \in \N$, estimate \eqref{res01} holds and furthermore
	\begin{align}
\label{res12}
&\sup_{t \in \R} \norm{H^p \, \partial_t^\ell R_P(t, \lambda) \, H^{-p}} \leq  \frac{ \ell! \ c^{\ell}}{A (1+ \ell)^2} \ \frac{2^3 \, \tC_H \, a \, b}{ \wt\Delta_{j-1}^{\delta } \, {\rm dist} (\lambda, \sigma(H))} \ , \quad \forall 0 \leq p \leq n \ , \ 1 \leq \ell \leq \tL \ . 
\end{align}
	\item[(iii)] For any $j \geq 1$ consider the projector \eqref{eq:prj}.
It fulfills 
$$
    \sup_{t \in \R} \norm{H^p \, \partial_{t}^{\ell}  \Pi_j(t) \, H^{-p}} \leq \frac{ \ell! \ c^{\ell}}{A (1+ \ell)^2} \ \frac{2^4 \, \tC_H \, a \, b}{ \wt\Delta_{j-1}^{\delta } }\ , \qquad \forall 0 \leq p \leq n, \ 1 \leq \ell \leq \tL \ . 
   $$
   \end{itemize}
   \end{lemma}
   \begin{proof}
(i) See the proof of Lemma \ref{sandres}(i).

(ii)      We prove only  estimate \eqref{res12}. The other statements are proved as in Lemma \ref{lem:an02} (ii).
First remark that  by  Lemma \ref{sandres} with $\theta = \nu$ and Lemma \ref{lem:est:res0}
\begin{equation}
\label{est90} 
\sup_{t \in \R} \norm{H^{p+\nu} \  R_P(t, \lambda)\  H^{-p} } \leq \frac{2 \, \tC_H} {\wt\Delta_{j-1}^\delta} \ , \quad \forall \lambda\in \Gamma_j \ , \ \ \ \forall 0 \leq p \leq n \ 
\end{equation}
provided
$$
a \, 4\, \tC_H\, \leq \wt\Delta_{j-1}^\delta \ . 
$$
Clearly such condition is implied by \eqref{cond2}. 

Now take  $1 \leq \ell \leq \tL$. Formula \eqref{ris.der1}  and 
    estimates  \eqref{P.est2}, \eqref{res01}, \eqref{est90} give for any $\lambda \in \Gamma_j$, $\forall j \geq 1$, $\forall 0 \leq p \leq n$
   \begin{align}
   \label{ris.est00}
  & \sup_{\ t \in \R} \norm{H^p \ \partial_{t}^\ell R_P(t, \lambda) \ H^{-p} }  \\
  \notag
   &   \leq 
   \ell! \ c^{\ell} \ \sum_{k=1}^{\ell} 
  \frac{2}{ {\rm dist} (\lambda, \sigma(H))} \left(\frac{2 \, \tC_H \, a\,b}{\wt\Delta_{j-1}^{\delta }}\right)^k  \frac{1  }{A^k} \sum_{n_1, \ldots, n_k \in \N \atop n_1 + \cdots +n_k = \ell} \frac{1}{(1+n_1)^2}\cdots \frac{1}{(1+n_k)^2}  \\
\notag
   & \leq \frac{ \ell! \ c^{\ell}}{A (1+ \ell)^2} \ \frac{2}{ {\rm dist} (\lambda, \sigma(H))}  \  \sum_{k=1}^{\ell} \left(\frac{2 \, \tC_H \, a\,b}{\wt\Delta_{j-1}^{\delta }}\right)^k  \\
   \notag
   & \leq \frac{ \ell! \ c^{\ell}}{A (1+ \ell)^2} \ \frac{2^3 \, \tC_H \, a \, b}{ \wt\Delta_{j-1}^{\delta } \, {\rm dist} (\lambda, \sigma(H))}
   \end{align}
where to pass from the third to fourth line we used that by \eqref{cond2} $\frac{2 \, \tC_H \, a\,b}{\wt\Delta_{j-1}^{\delta }} \leq \frac{1}{2}$. Thus  \eqref{res12} is proved.  

(iii) By \eqref{res12} one has  $\forall 1 \leq \ell \leq \tL$, $\forall 0 \leq p \leq n$
       \begin{align*}
  \sup_{ t \in \R} \| H^p \ \partial_{t}^\ell \Pi(t) \ H^{-p} \| 
  &\leq
   \frac{ \ell! \ c^{\ell}}{A (1+ \ell)^2} \ \frac{2^3 \, \tC_H \, a \, b}{ \wt\Delta_{j-1}^{\delta } } \,
   \frac{1}{2\pi}\oint_{\Gamma_j} 
   \frac{d \lambda}{{\rm dist} (\lambda, \sigma(H))} \\
 & \leq \frac{ \ell! \ c^{\ell}}{A (1+ \ell)^2} \ \frac{2^4 \, \tC_H \, a \, b}{ \wt\Delta_{j-1}^{\delta } }
   \end{align*}
   where we used also \eqref{bb0}.
   \end{proof}
 \begin{lemma}
   \label{lem:an3}   
    Let $P(t), Q(t)$ be operators analytic in time fulfilling  $\forall 0 \leq \ell \leq \tL$, $\forall 0 \leq p \leq n$
\begin{equation}
   	\label{P.est3}
   \sup_{t \in \R} 	\norm{H^p \ \partial_{t}^{\ell} P(t) \ H^{-p-\nu}} \ , \  \sup_{t \in \R}\norm{H^p \ \partial_{t}^{\ell} Q(t) \ H^{-p-\nu}} \leq a \,  b^{\min(\ell,1)} \,  \frac{\ell! \, c^{\ell} }{A (1+ \ell)^2} \ . 
   	\end{equation} 
   	Assume that \eqref{cond2} holds. Furthermore let  $B(t)$ be an operator analytic in time fulfilling $\forall 0 \leq  \ell \leq \tL$
   \begin{equation}
   \label{B.aa}
  \sup_{t \in \R} \norm{H^p \ \partial_{t}^\ell B(t) \ H^{-p}} \leq h \,  c^{\ell} \,  \frac{\ell!}{A (1+ \ell)^2} \ , \quad \forall 0 \leq p \leq n
   \end{equation}
  for some positive $h \in \R$.
   Then the operator
   $$
   K(t) := -\frac{1}{2\pi \im} \oint_{\Gamma_j} R_P(t,\lambda) \, B(t) \, R_Q(t,\lambda) \,  d\lambda
   $$
   is analytic in time, bounded from $H^p$ to $H^p$ $\ \forall 0 \leq p \leq n$,  and fulfills $\forall 0 \leq \ell \leq \tL$, $\forall 0 \leq p \leq n$
   $$
   \sup_{t \in \R} \norm{\partial_{t}^\ell  K(t)} \leq \frac{\ell! \, c^l}{A (1+ \ell)^2} \, 
   \frac{ h \, 2^5}{ \, \wt \Delta_{j-1}}  \ . 
   $$
   \end{lemma}
   \begin{proof}
   First consider the resolvents $R_P(t,\lambda)$, $R_Q(t, \lambda)$. Proceeding as in the proof of  Lemma \ref{lem:an02}, they are well defined for any $\lambda \in \Gamma_j$, $\forall j \geq 1$, and fulfill estimates \eqref{res01}, \eqref{res12}. Consider now $K(t)$. For  $\ell = 0$ one has  
      $$
    \sup_{t \in \R} \norm{ H^p \  K(t) \ H^{-p}} \leq \frac{ h}{A} \frac{1}{2\pi} \oint_{\Gamma_j} \frac{4 \, d\lambda}{{\rm dist}(\lambda, \sigma(H))^2} \leq \frac{4 \, h}{A \, \wt \Delta_{j-1}} \ , \quad \forall 0 \leq p \leq n \ .
   $$
For  $1 \leq \ell \leq \tL$, consider \eqref{est:K10}--\eqref{est:K40}. We estimate each line. By \eqref{B.aa}, \eqref{res01} one has
   \begin{align*}
  \sup_{t \in \R} \norm{H^p \ \eqref{est:K10} \ H^{-p}} \leq \frac{4 \, h}{ \, \wt \Delta_{j-1}}\frac{\ell! \, c^\ell}{A (1+ \ell)^2}  \ , \quad \forall 0 \leq p \leq n \ .  
   \end{align*}
   To estimate the  second line we use  \eqref{res01}, \eqref{res12},   \eqref{B.aa} and \eqref{cond2} to get  $\forall 0 \leq p \leq n$
\begin{align*}
 \sup_{t \in \R}  \norm{H^p \ \eqref{est:K20} \ H^{-p}} 
& \leq \frac{2^4 \, \tC_H \, a \, b \, h}{\wt\Delta_{j-1}^\delta} \ \frac{1}{2\pi}\oint_{\Gamma_j} \frac{d\lambda}{{\rm dist}(\lambda, \sigma(H))^2}
 \leq    \frac{\ell! \, c^\ell}{A (1+ \ell)^2} \, \frac{ 2^6 \, \tC_H \, h \, a\, b  }{\wt \Delta_{j-1}^{1+\delta}} \ . 
   \end{align*}
   The third line is estimated exactly as the second one. We pass to the last line.   Using \eqref{res12}, \eqref{B.aa} we get $\forall 0 \leq p \leq n$
   \begin{align*}
   \sup_{t \in \R} \norm{H^p \ \eqref{est:K40} \ H^{-p}} \leq 
   \frac{\ell! \, c^\ell}{A (1+ \ell)^2} \frac{2^6 \, \tC_H^2 \, a^2 \, b^2 \, h}{\wt\Delta_{j-1}^{2\delta}} \ \frac{1}{2\pi}\oint_{\Gamma_j} \frac{d\lambda}{{\rm dist}(\lambda, \sigma(H))^2}
   \leq
    \frac{\ell! \, c^\ell}{A (1+ \ell)^2} \, \frac{2^8 \, h  \,(\tC_H \, a\, b)^2 }{\wt \Delta_{j-1}^{1+2\delta}} \ . 
   \end{align*}
   Altogether we find that for $1 \leq \ell \leq \tL$, $\forall 0 \leq p \leq n$
   $$
  \sup_{t \in \R} \norm{ H^p \ \partial_{t}^\ell  K(t) \ H^{-p}} \leq  \frac{\ell! \, c^\ell}{A (1+ \ell)^2} \, 
   \frac{2^2 \, h}{ \, \wt \Delta_{j-1}} \left(1+ \frac{ \tC_H\,a\, b \, 2^{5} }{\wt \Delta_{j-1}^\delta} +  \frac{(\tC_H\,a\, b)^2 \, 2^{6}}{\wt \Delta_{j-1}^{2\delta}} \right) \leq 
    \frac{\ell! \, c^\ell}{A (1+ \ell)^2} \, 
   \frac{  h \, 2^5}{ \, \wt \Delta_{j-1}}
   $$
   where we used again \eqref{cond2}.
      \end{proof}

\section{Proof of Lemma \ref{compk}}
\label{app:normcontrol}

We start with an abstract result. Let  $H_W(t) := H+ W(t)$, $H$ being a self-adjoint positive operator  in $\cH^0$, $W(t)$ a symmetric operator, $H^\nu$-bounded with  $\nu<1$. We assume that,  for a fixed $n\in\N$,   we have 
\begin{itemize}
\item[{\rm (W$)_n$}]  $ H^p \ W(\cdot) \ H^{-p-\nu}\in C_b^0(\R, \cL(\cH^0))\ $, $\sup_{t\in\R}\Vert H^p\  W(t) \  H^{-p-\nu}\Vert \leq D_n$, $\forall 0\leq p \leq n$.
\end{itemize}

\begin{lemma}\label{Hn} Let $n\geq 1$. Assume  that $W$  satisfies  condition {\rm (W$)_{n}$}.\\
Define $W_n(t)  = (H+W(t))^n -H^n$. Then we have $W_nH^{1-n-\nu}\in\cL(\cH^0)$. Furthermore there  exist positive  constants $\gamma_0$, $\gamma_1$  depending only on $H$  such that 
\beq\label{pert}
\Vert W_nH^{1-n-\nu}\Vert \leq \gamma_0\gamma_1^nD_n^{n+1}.
\eeq
Finally we have 
\beq\label{Hkest}
 c_n\Vert\psi\Vert_{2n}\leq \Vert(H +W(t)+c_0)^n\Vert_0 \leq C_n\Vert\psi\Vert_{2n},\;\; \forall \psi\in\cH^{2n},\;\forall t\in\R,
\eeq
where   $c_n, C_n$ depend only on $D_n$.
\end{lemma}
\begin{proof}
We   proceed  by induction on $n$. For $n=1$ the two side estimate is a classical   perturbation result    using  {\rm (W$)_{0}$}.  For $n>1$  we have  
\beq\label{indn}
W_{n+1}(t) = W_n(t) \ H + W_n(t) \ W(t)+ H^n \ W(t).
\eeq
Let us denote $a_n(t)=\Vert H^n \ W(t) \ H^{-n-\nu}\Vert$  and $f_n(t)  = \Vert W_n(t) \ H^{1-n-\nu}\Vert$. By induction on $n$, using (\ref{indn}),  we get 
\begin{equation}
\label{ind2}
f_{p+1}(t) \leq a_p(t) + f_p(t) + \gamma_0 \, a_p(t) \, f_p(t)
\end{equation}
where $\gamma_0$ is a constant depending only on $H$.
From (\ref{ind2}) we get easily (\ref{pert}).\\

Now we can conclude   easily to get (\ref{Hkest}) using the interpolation  inequality:
 for $0\leq s<n$ and  $\varepsilon\in]0, 1] $  we have :
$$
\Vert H^s\psi\Vert_0^2\leq \varepsilon^2\Vert H^n\psi\Vert_0^2  + \varepsilon^{\frac{2s}{s-n}}\Vert\psi\Vert_0^2.
$$
From (\ref{pert}) we have 
$$
\Vert W_n(t) \psi\Vert_0 \leq \Vert W_n(t) \ H^{1-n-\nu}\Vert \  \Vert H^{n+\nu-1}\psi\Vert_0 \ .
$$ 
Taking $s=n+\nu-1$ and $\varepsilon$ small enough, we get (\ref{Hkest})  where $c_n$ and $C_n$  depend only on $D_n$. 
\end{proof}

\vspace{1em}
\noindent {\em Proof of Lemma \ref{compk}}. 
(i) Recall that $H_{ad, m}(t) = H + V(t) - B_{m}(t)$. 
We apply Lemma \ref{Hn} with $W = V - B_m$. By the assumptions on $V$ and Lemma \ref{estadB}, $W$ fulfills (W$)_n$, thus we get 
\eqref{Hadk}.

(ii) If $\tJ$ is sufficiently large, by Lemma \ref{perturb} the Hamiltonian $H_{ad, m}(t)$ satisfy $\wt{{\rm (Hgap)}} $ uniformly in $t \in \R$ (see Corollary \ref{cor:hgap}). 
 Then writing
$$
H_{ad,m}(t) = \sum_{j \geq 1} \Pi_{m,j}(t) \, H_{ad,m}(t)  \, \Pi_{m,j}(t) \ , 
$$ 
 one gets easily that 
 \begin{equation}
\sum_{j \geq 1} \, (\lambda_j^- +c_0)^{2p} \norm{\Pi_{m,j}(t) \psi}_0^2 \leq \norm{ (H_{ad,m}(t) +c_0)^p \psi}_0^2 \leq \sum_{j \geq 1} \, (\lambda_j^+ + c_0)^{2p} \norm{\Pi_{m,j}(t) \psi}_0^2 
\end{equation}
and
\begin{align*}
\sum_{j \geq 1} \, (\lambda_j^+ +c_0)^{2p} \norm{\Pi_{m,j}(t) \psi}_0^2 \leq C_p 2^{\tJ(\mu+1)2p} \norm{\Lambda_m^{p}(t) \psi}_0^2 \\
\sum_{j \geq 1} \, (\lambda_j^- +c_0)^{2p} \norm{\Pi_{m,j}(t) \psi}_0^2 \geq c_p 2^{\tJ(\mu+1)2p} \norm{\Lambda_m^{p}(t) \psi}_0^2 
\end{align*}
from which \eqref{E9} follows.

\section{Some properties of the pseudodifferential calculus}
 \label{app:pseudo}

We recall here some  fundamental results of symbolic calculus. For the proof see \cite{ro1, hero}.
\begin{theorem}[Symbolic calculus]
\label{symb.cal}
Let $A \in \wt S^\nu_{k,\ell}$,  $B \in \wt S^\mu_{k, \ell}$ be  symbols. Then  there exists a unique  semi-classical  symbol $A\sharp B\in \wt S^{\nu+\mu}_{k, \ell}$  such that ${\rm Op}_\h^W(A) \, {\rm Op}_\h^W(B) = {\rm Op}_\h^W(A\sharp B)$.  $A\sharp B$ is the {\em Moyal product}  of $A$ and $B$.
\end{theorem}
The Moyal product is a bilinear continuous map. More precisely it holds the following: for every $\alpha, \beta$, there exists a positive constant $C_{\alpha \beta}$ (independent of $A$ and $B$) and an integer $M \equiv M(\alpha, \beta) \geq 1$ such that
$$
p_{\alpha \beta}^{\nu+\mu} (A\sharp B) \leq C_{\alpha \beta} \, |a|_{M,\nu} \, |b|_{M,\mu} \ .
$$

The symbolic calculus implies the following result on the commutator of two pseudodifferential operators:
\begin{corollary}[Commutator]
\label{cor:comm}
Let $A \in \wt S^\nu_{k,\ell}$,  $B \in \wt S^\mu_{k, \ell}$ be  symbols. Then  there exists a unique  semi-classical  symbol $C \in  \wt S^{\nu+\mu -(k+\ell)}_{k, \ell}$  such that $[{\rm Op}_\h^W(A),  {\rm Op}_\h^W(B)] = {\rm Op}_\h^W(C)$.
\end{corollary}
The second result concerns the boundedness of pseudodifferential operators:
\begin{theorem}[Calderon-Vaillancourt]
\label{thm:cv} 
Let $A \in \wt S^0_{k, \ell}$ be a symbol. Then there exist  constants $C, N >0$ such that  ${\rm Op}_\h^W(A)$ extends to a linear bounded operator from $L^2$ to itself, and the following estimate holds:
\begin{equation}
\norm{{\rm Op}_\h^W(A)}_{\cL(L^2)} \leq C  \ |A|_{N,0} \ , \quad \forall \h\in]0, 1].
\end{equation}
\end{theorem}
Notice that $C$ and $N$ are universal constants, independent on $A$ (see  for example \cite{robert_book}).


%

\bibliographystyle{alpha}
\end{document}